\documentclass[10pt]{article}
\usepackage[english]{babel}
\usepackage{url}
\usepackage{inputenc}
\usepackage{amsfonts}
\usepackage{amsmath}
\usepackage{empheq}
\usepackage{graphicx}
\graphicspath{{images/}}
\usepackage{parskip}
\usepackage{fancyhdr}
\usepackage{bbold}
\usepackage{amssymb}
\usepackage{vmargin}
\usepackage{array}
\usepackage{amsthm}
\usepackage[ruled,vlined]{algorithm2e}
\usepackage{hyperref}
\usepackage{caption}
\usepackage{float}
\usepackage{fourier-orns}
\usepackage{listings}
\usepackage{epsfig}
\usepackage{amsmath}     

\usepackage[bbgreekl]{mathbbol}

\usepackage[scr=boondoxo,scrscaled=1.05]{mathalfa}

\newtheorem{assumption}{Assumption}
\newtheorem{theorem}{Theorem}[section]

\newtheorem{lemma}[theorem]{Lemma}
\newtheorem{remark}{Remark}[section]
\theoremstyle{definition}

\newtheorem{prop}{Proposition}[section]
\usepackage{listings}

\newcommand{\cha}{\mathbb{1}_{\Omega_0}}
\newcommand{\out}{\mathbb{1}_{\mathbb{R}\setminus\Omega_0}}

\newcommand{\om}{{\Omega_{0}}}
\newcommand{\mL}{{\mathbb{L}}}

\newcommand{\R}{{\mathbb{R}}}

\newcommand{\G}{{\mathbb{G}}}
\newcommand{\B}{{\mathbb{B}}}
\newcommand{\mA}{{\mathbb{A}}}
\newcommand{\mB}{{\mathbb{B}}}

\newcommand{\oF}{\overline{\mathbb{F}}}
\newcommand{\oL}{\overline{\mathbb{L}}}
\newcommand{\oA}{\overline{\mathbb{A}}}
\newcommand{\oB}{\overline{\mathbb{B}}}
\newcommand{\gdag}{{\gamma^\dagger}}
\newcommand{\og}{\overline{\gamma}}
\newcommand{\ogd}{\overline{\gamma}^\dagger}

\newcommand{\mln}{{\mathbb{\Lambda}_\nu}}

\newcommand{\bydef}{\stackrel{\mbox{\tiny\textnormal{\raisebox{0ex}[0ex][0ex]{def}}}}{=}}

\usepackage[shortlabels]{enumitem}
\mathtoolsset{showonlyrefs=true}

\usepackage[dvipsnames]{xcolor}

\title{Constructive proofs of existence and stability of solitary waves in the Whitham and capillary-gravity Whitham  equations}
\author{
Matthieu Cadiot
\footnote{McGill University, Department of Mathematics and Statistics, 805 Sherbrooke Street West, Montreal, QC, H3A 0B9, Canada. {\tt matthieu.cadiot@mail.mcgill.ca}}}

\begin{document}

\maketitle

\begin{abstract}
In this manuscript, we present a method to prove constructively the existence and  spectral stability of solitary waves in both the Whitham and the capillary-gravity Whitham  equations. By employing Fourier series analysis and computer-aided techniques, we successfully approximate the Fourier multiplier operator in this equation, allowing the construction of an approximate inverse for the linearization around an approximate solution $u_0$. Then, using a Newton-Kantorovich approach, we provide a sufficient condition under which the existence of a unique solitary wave $\tilde{u}$ in a ball centered at $u_0$ is obtained. The verification of such a condition is established combining analytic techniques and rigorous numerical computations. Moreover, we derive a methodology to control the spectrum of the linearization around $\tilde{u}$, enabling the study of spectral stability of the solution. As an illustration, we provide a (constructive) computer-assisted proof of existence of stable  {solitary waves} in both the case with capillary effects ($T>0$) and without capillary effects ($T=0$).  { Moreover, we provide an existence proof for a branch of solitary waves in the case $T=0$ via a rigorous continuation in the wave velocity.} The methodology presented in this paper can be generalized and provides a new approach for addressing the existence and spectral stability of solitary waves in nonlocal nonlinear equations. All computer-assisted proofs, including the requisite codes, are accessible on GitHub at \cite{julia_cadiot}.
\end{abstract}

\begin{center}
\end{center}

\section{Introduction}
In this paper, a computer-assisted analysis of solitary wave solutions to the Whitham equation ( {WE}) and to the capillary-gravity Whitham  equation ( {cgWE}) is presented. Building upon the findings in \cite{unbounded_domain_cadiot}, which primarily addresses the PDE case, we introduce new techniques to rigorously treat Fourier multiplier operators in nonlocal equations. These techniques are applied to both  {WE} and  {cgWE}, and the details for their specific analysis are exposed. These equations were originally proposed by Whitham to offer a more accurate model for surface water waves than the celebrated Korteweg-de-Vries (KdV) equation (see \cite{lannes_water_wave, whitham_variational, whitham2011linear} for a introduction to this model).  {Whitham's model} captures intricate fluid dynamics phenomena, such as wave breaking  {(see \cite{Hur_wave_breaking, saut_wave_breaking})} and cusped solutions  {(see \cite{EHRNSTROM2019on_whitham_conjecture})}.  Notably, it features solitary wave solutions, the central topic of study of this article. More specifically,
we investigate the existence and spectral stability of traveling solitary waves in the following equation  
\begin{equation}\label{eq : original Whitham}
   u_t +  \partial_x{\mathbb{M}_T}u + \frac{1}{2} u \partial_xu =0,
\end{equation}
where $\mathbb{M}_T$ is a Fourier multiplier operator defined via its symbol 
\begin{equation}\label{def : Kernel whitham}
        \mathcal{F}(\mathbb{M}_Tu)(\xi) \bydef m_T(2\pi\xi) \hat{u}(\xi) \bydef \sqrt{\frac{\tanh(2\pi\xi)(1+T(2\pi\xi)^2)}{2\pi\xi}}\hat{u}(\xi)
\end{equation}
for all $\xi \in \mathbb{R}$. The quantity $T\geq 0$ is the Bond number accounting for the capillary effects (also known as surface tension). If $T=0$, \eqref{eq : original Whitham} is fully gravitational and becomes the ``Whitham equation",  {denoted  {WE }along this paper. On the other hand, if $T>0$,  \eqref{eq : original Whitham} is called the ``capillary-gravity Whitham equation" and will be denoted  {cgWE }}. We will keep this distinction of name in mind as the analysis of the cases $T=0$ and $T>0$ will have to be handled separately. Using the traveling wave ansatz $X = x-ct$ in \eqref{eq : original Whitham}, we look for a solitary wave  $u : \R \to \R$ such that
\begin{align}\label{eq : whitham stationary}
   \mathbb{F}(u) \bydef {\mathbb{M}_T}u -cu + u^2 =0
\end{align}
where $c \in \mathbb{R}$ and $u(x) \to 0$ as $|x| \to \infty.$ Moreover, we look for even solutions to \eqref{eq : whitham stationary}, that is solutions $u$ satisfying $u(x) = u(-x)$ for all $x \in \R$. Acknowledging the presence of cusped solutions (see \cite{EHRNSTROM2019on_whitham_conjecture}), we restrict to smooth solutions to simplify the analysis.  
 {Specifically, we look for a solutions in an Hilbert space $\mathcal{H}_e$ (defined in \eqref{def : definition Hl}), which is a subspace of $H^2(\R).$ }
Our investigation of solitary waves is then achieved  by studying the zeros of  {$\mathbb{F} : \mathcal{H}_e \to H^2(\R)$}
In addition, smooth (classical) and even solutions to \eqref{eq : whitham stationary} are equivalently zeros of $\mathbb{F}$ in $\mathcal{H}_e$ (cf. Proposition \ref{prop : regularity of the solution}).

Constructively proving the existence of solutions to nonlocal equations is in general  {a very complex task}.  As we will present later on, non-constructive existence results are numerous, but only a few provide quantitative results about the solution itself (e.g. its shape, its amplitude, its symmetry, etc). This difficulty arises from the fact that solutions live in an infinite-dimensional function space and the position of the solution in this function space is usually unknown. From that aspect, computer-assisted proofs (CAPs) have become a natural tool to prove constructively the existence of solutions to nonlinear equations.  { Indeed, computer-aided techniques have displayed their potential through a wide variety of results, including
the Feigenbaum conjectures \cite{MR648529}, the existence of chaos
and global attractor in the Lorenz equations \cite{MR1276767,MR1701385,MR1870856},
 Wright's conjecture \cite{MR3779642},
chaos in the Kuramoto-Sivashinsky PDE \cite{MR4113209}, blowup in 3D
Euler \cite{blowup_3D_Euler} and imploding solutions for 3D compressible
fluids \cite{imploding_sols_3D_fluids}.}
We refer the interested reader to the following review papers \cite{nakao_numerical, gomez_cap, jb_rigorous_dynmamics, koch_computer_assisted} and the book \cite{plum_numerical_verif} for additional details. We want to emphasize the work by Enciso et al. in \cite{javier_convexity_highest} where a constructive proof of existence of a cusped periodic wave was obtained in the WE . The authors successfully demonstrated the convex profile of the periodic wave, resolving the conjecture proposed by Ehrnstr\"om and Wahlén in \cite{EHRNSTROM2019on_whitham_conjecture}. The proof is computer-assisted and relies on the approximation of the solution by a mix of Clausen functions (which allows to approximate precisely the cusp) and Fourier series.  {Note that the exact leading-order asymptotic behavior of the cusped solution was recently established analytically in \cite{cusped_behavior_ehsnstrom}}.

To lay the groundwork for subsequent discussions in this paper, we provide a brief overview of existing techniques and their applications to \eqref{eq : whitham stationary}. In particular, we distinguish two categories of methods : those relying on the concentration-compactness method and the ones arising from perturbation or bifurcation arguments.

One of the most general approach to tackle nonlocal equations is the concentration-compactness method \cite{lions1984concentration}. Indeed, defining a well-chosen functional $\mathcal{E} : H^k(\R) \to \R$ such that the minimizers of $\mathcal{E}$ (under some constraints) are solutions to \eqref{eq : whitham stationary}, provides general results of existence and energetic stability. For instance,  \cite{albert_concentration_compactness} and \cite{arnesen_existence_solitary} present a general setting to prove the existence and conditional energetic stability of solitary waves in a large class of nonlocal 1D equations. In particular, \cite{arnesen_existence_solitary} obtained the existence of solitary waves (with conditional energetic stability) in \eqref{eq : original Whitham} for any $T>0$ and $c < \min_{\xi \in \R} m_T(\xi)$. Note that under the assumption $c < \min_{\xi \in \R} m_T(\xi)$, the operator $\mathbb{M}_T - cI_d : H^{\frac{1}{2}}(\R) \to L^2(\R)$ has a bounded inverse. This assumption will be required in our set-up as well.  {It is worth noting that \cite{Buffoni_existence_conditional, BUFFONI20131006} provide the existence of solitary waves higher dimensional wave equations using improvements of the concentration-compactness method.} In the case of the  {WE  }($T=0$), the existence of a family of solitary waves  of small amplitude has been obtained in \cite{Ehrnstrom_existence_stability_solitary}. Similarly as in \cite{arnesen_existence_solitary}, the conditional energetic stability is obtained.  Their proof relies on the concentration-compactness method combined with an approximation of a solution by the known KdV solitary waves. In fact, when $T=0$ and  $c$ is close to $1$, the solitary waves in the  {WE  }\eqref{eq : whitham stationary} can be approximated by those of KdV equation. This local result  allowed the development of bifurcation or perturbation methods to study solitary waves, leading to a second category of techniques.

Indeed, the existence of solitary waves in \eqref{eq : original Whitham} can be obtained locally using the known explicit solutions in $\text{sech}^2$ of the KdV equation. For instance, Stefanov and Wright \cite{Stefanov2018SmallAT} proved the existence of small amplitude solitary waves as well as periodic traveling waves in the  {WE }for $c$ slightly bigger than $1.$ In addition, using the known solitons in the KdV equation, they were able to prove spectral stability by controlling the spectrum of the linearization. Johnson et al.  \cite{johnson_generalized} extended this result and proved the existence of generalized solitary waves in the  {cgWE}. Recently, Truong et al. \cite{truong_global_whitham, johnson_global_capillary_whitham} used the center manifold theorem \cite{scheel_center_manifold, faye2020corrigendum} to prove the existence of a bifurcation of solitary waves in both the case $T=0$ and $T>0$. Using a specific system of ODEs as an approximation, a local branch of solutions can be proven. It is worth mentioning that a global bifurcation is obtained in \cite{truong_global_whitham} in the WE , leading to the proof of existence of a cusped  {solitary wave} with a $C^{\frac{1}{2}}$ regularity ( {the regularity of the solitary waves on the branch being already established in \cite{EHRNSTROM2019on_whitham_conjecture})}.  Ehrnstrom et al. \cite{ehrnstrom_direct} obtained a similar result by using a family of periodic solutions converging to the solitary wave as the period goes to infinity.  Notably, the concept of a period-limiting sequence had been previously employed in  {\cite{ Buffoni_existence_conditional, Ehrnstrom_existence_stability_solitary}} or \cite{Hildrum_2020}.  {Moreover, the authors in \cite{arnesen_maximization_technique} recently established the existence of a family of solitary waves by considering subproblems on intervals $[-2^l,2^l]$ and successfully took the limit as $l \to \infty$ thanks to the development of sharp estimates and an innovative use of Orlicz spaces.}   {The combination of period-limiting and compactly supported functions} is central in our analysis, as we leverage the strong connection between the periodic problem and the problem defined on $\R$ to constructively establish the existence of solitary waves through a periodic approximation over a sufficiently large interval.  {In particular, we are able to prove that the obtained solitary waves are the limit of a branch of periodic solutions when letting the period tend to infinity (cf. Theorems \ref{th : proof whitham} and \ref{th: : proof capillary whitham}).}

In general, proving constructively the existence of solitary waves and determining their spectral stability , without restriction on the parameters (e.g. $c$ being in an epsilon neighborhood of $1$ in the WE ), is a highly complex problem. In  {a more general context}, the non-local equation might not always be locally  approximated by an  {explicitly known} differential equation. In this paper, we partially address this question for \eqref{eq : whitham stationary} and provide a general methodology to establish constructively the existence of  {solitary waves} as well as their spectral stability. Specifically, we develop new computer-assisted techniques to handle directly non-local equations, which we present in the following paragraphs.

As a matter of fact, we  use the method developed in \cite{unbounded_domain_cadiot}, which is based on Fourier series analysis, and extend it to nonlocal equations. First, one needs to construct an approximate even solution $u_0 : \R \to \R$ in an Hilbert space $\mathcal{H}_e$ (cf. \eqref{def : definition of He}) such that its support is contained in an interval $\om \bydef (-d,d)$. In particular, $u_0$ is defined via its Fourier coefficients $(a_n)_{n \in \mathbb{N}\cup\{0\}}$ on $\om$, which is chosen in such a way that $u_0$ is smooth (cf. Section \ref{ssec : approximate solution}) and even. In other terms,
\begin{align*}
    u_0(x) = \cha(x) \left(a_0 + {2}\sum_{n \in \mathbb{N}}a_n\cos\left(\frac{n\pi}{d}x\right) \right)
\end{align*}
for all $x \in \R$,  where $\cha$ is the characteristic function on $\om.$
Intuitively, the restriction of $u_0$ to $\om$ approximates a periodic solution to \eqref{eq : whitham stationary}. If $d$ is big enough, then $u_0$ is supposedly a good approximation for a solitary wave as well.   
Now, in the PDE case  presented in \cite{unbounded_domain_cadiot}, given a linear differential operator $\mathbb{L}$ with constant coefficients and with an even symbol $l$, we have
\[(\mathbb{L}u_0)(x) = \cha(x)\left(l(0) a_0 +  {2}\sum_{n \in \mathbb{N}}l\left(\frac{n\pi}{d}\right)a_n\cos\left(\frac{n\pi}{d}x\right)\right)\]
for all $x \in \R$. The above is simply obtained leveraging the facts that $u_0$ is smooth  {(providing convergence of the Fourier series and regularity at $\pm d$)}  and that $\mathbb{L}$ is a local operator. The main difficulty in \eqref{eq : whitham stationary} arises from the presence of the nonlocal operator $\mathbb{M}_T.$ More specifically, the simple evaluation of the function $\mathbb{M}_Tu_0$ on $\R$ is  challenging.
In this paper, we present a computer-assisted approach to answer this problem. In fact, using the transformation $\Gamma^\dagger$ defined in \eqref{def : Gamma and Gamma dagger}, we prove that $\mathbb{M}_Tu_0$ can be approximated by $\Gamma^\dagger\left(M_T\right) u_0$, which is given by 
\begin{align*}
    \left(\Gamma^\dagger\left(M_T\right) u_0\right)(x) = \cha(x)\left(m_T(0) a_0 +  {2}\sum_{n \in \mathbb{N}}m_T\left(\frac{n\pi}{d}\right)a_n\cos\left(\frac{n\pi}{d}x\right)\right).
\end{align*}
Intuitively, $\Gamma^\dagger\left(M_T\right)$ is the periodization on $\om$ of the operator $\mathbb{M}_T$, which had already been introduced and studied in \cite{EHRNSTROM2019on_whitham_conjecture}. 
 In particular, we prove that an upper bound for $\|\mathbb{M}_Tu_0 - \Gamma^\dagger\left(M_T\right) u_0\|_{L^2(\R)}$ can be computed explicitly with the use of rigorous numerics (cf. Section \ref{sec : Y0 bound}). This upper bound depends on  the domain of analyticity of $m_T$ and we prove that it is exponentially decaying with $d$. Therefore, given a function $u_0$ with compact support on a big enough domain $\om$, our approach provides a rigorous approximation to $\mathbb{M}_Tu_0$ with high accuracy. This approach can be generalized to Fourier multiplier operators which symbol is analytic on some strip of the complex plane (see Section \ref{conclusion}).  {This is, to the best of our knowledge, a new result for the treatment of Fourier multiplier operators in the computer-assisted field.}
 
 Now, given a fixed approximate solution $u_0$, our goal is to develop a Newton-Kantorovich approach in a neighborhood of $u_0$. In particular, we require the construction of an approximate inverse $\mathbb{A}_T$ of $D\mathbb{F}(u_0)$. By approximate inverse, we mean an operator $\mathbb{A}_T : H^2_e \to \mathcal{H}_e$ (where $H^2_e$ is the restriction of $H^2(\R)$ to even functions) such that the $\mathcal{H}_e$ operator norm satisfies $\|I_d - \mathbb{A}_TD\mathbb{F}(u_0)\|_{\mathcal{H}} < 1$ and can be computed explicitly (see Section \ref{sec : computation of the bounds} and the computation of $\mathcal{Z}_1$). Generalizing the results of \cite{unbounded_domain_cadiot}, we can readily treat the case $T>0$  since $m_T(\xi) \to \infty$ as $|\xi| \to \infty$.\\ However,  {in the case $T=0$,} notice that $m_0(\xi) \to 0$ as $|\xi| \to \infty$ meaning the theory developed in \cite{unbounded_domain_cadiot} does not apply anymore. In fact, as $\xi$ gets large, then $D\mathbb{F}(u_0)  \approx  \left(I_d - \frac{2}{c}\mathbb{u}_0\right)\mathbb{L}$, where $\mathbb{u}_0$ is the multiplication operator by $u_0$. Consequently, we need to ensure that $ I_d - \frac{2}{c}\mathbb{u}_0 : L^2 \to L^2$ is invertible and has a bounded inverse. Having this strategy in mind, we use the construction presented in  \cite{breden_polynomial_chaos} to build $\mathbb{A}_0$.  {For $w_0$ chosen appropriately,  we choose}  $\mathbb{A}_0 \approx \mathbb{L}^{-1}\mathbb{w}_0$ for high frequencies, where $\mathbb{w}_0$ is the multiplication operator associated to the function $w_0 \in L^2(\R)$. In particular, $w_0$ is chosen so that $(1 - \frac{2}{c}{u}_0)w_0 \approx 1$. In practice $w_0$ is  {determined} numerically and the quantity $(1 - \frac{2}{c}{u}_0)w_0$ can be controlled rigorously thanks to the arithmetic on intervals (see \cite{Moore_interval_analysis} and \cite{julia_interval}). Consequently, we can verify explicitly that $\mathbb{A}_0$ is indeed an accurate approximate inverse (see Section \ref{sec : bound Z1}).  Note that this approach only makes sense if $|u_0 - \frac{c}{2}| >0$ uniformly. 
 Consequently,  {in the case $T=0$}, we assume  that there exists $\epsilon>0$ (for which the existence is verified numerically) such that $u_0 + \epsilon < \frac{c}{2}$ (cf. Assumption \ref{ass : u0 is smaller than}) and the construction of the approximate inverse $\mathbb{A}_0$ is presented in Section \ref{ssec : approximate inverse}.  { Note that such a requirement is not needed for  {cgWE }, that is when $T>0$}. This assumption is not surprising as it allows to restrict to smooth solutions (cf. \cite{EHRNSTROM2019on_whitham_conjecture} for instance) and avoids the critical cusped solution which happens exactly when the maximal height of the solution is $\frac{c}{2}$ (hence $D\mathbb{F}(u_0)$ would be singular).
 
 The construction of an approximate inverse is  essential in our analysis as it allows to develop a Newton-Kantorovish approach. Indeed, defining 
 \[
 \mathbb{T}(u) = u - \mathbb{A}_T\mathbb{F}(u)
 \]
 and assuming that $\mathbb{A}_T$ is injective, we prove that $\mathbb{T}$ is contracting from a closed ball $\overline{B_r(u_0)}$ in $\mathcal{H}_e$ of radius $r$ and centered at $u_0$ to itself. Using the Banach fixed point theorem, we obtain the existence of  a unique solution in $\mathcal{H}_e$ close to $u_0$ (cf. Theorem \ref{th: radii polynomial}). In particular, the radius $r$ controls rigorously the accuracy of the  numerical approximation $u_0$. In practice, $r$ is usually relatively small and provides a sharp control on the true solution (cf. Theorems \ref{th : proof whitham} and \ref{th: : proof capillary whitham}). 
 Similarly as in \cite{unbounded_domain_cadiot}, the proof of  $\mathbb{T} : \overline{B_r(u_0)} \to \overline{B_r(u_0)}$ being contractive and $\mathbb{A}_T$ being injective relies on the explicit computation of some upper bounds $\mathcal{Y}_0, \mathcal{Z}_1, \mathcal{Z}_2$. We expose the computation of such bounds in Section \ref{sec : computation of the bounds} in a general framework.
 
 On the other hand, being able to construct an approximate inverse allows  {to tackle different problem of interest such as stabillity and continuation}. Indeed, we can first obtain rigorous enclosures of simple eigenvalues using a Newton-Kantorovish approach, but also prove non-existence of eigenvalues. This strategy is used in Section \ref{sec : stability} where we  control the non-positive part of the spectrum of the linearization around the solutions of Theorems \ref{th : proof whitham} and \ref{th: : proof capillary whitham}. In particular, we prove that the zero eigenvalue is simple and that there exists only one negative eigenvalue which is also simple. This allows to conclude about the spectral stability of the aforementioned solitary wave solutions. 
On the other hand, the framework of the presented method allows to combine the computer-assisted proof of existence with a rigorous continuation.  { In fact, following the set-up introduced in \cite{breden_polynomial_chaos, henot_marchal}, we are able to use a numerical Chebyshev expansion in the wave velocity $c$ and obtain a constructive proof of a branch of solitary waves with high accuracy (cf. Theorem \ref{th : proof of a branch}). This is, to the best of our knowledge, the first computer-assisted proof of a branch of solutions in nonlocal equations. }

 Consequently, this paper introduces innovative methodologies  {to first evaluate Fourier multiplier operators and, most importantly,  approximate rigorously the inverse of the linearization} of \eqref{eq : whitham stationary} around an approximate solution.  The existence and spectral stability of solutions to \eqref{eq : whitham stationary} can then be studied via a Newton-Kantorovish approach. These results can  be generalized to a large class of nonlocal equations on $\R^n$ which we describe in Section \ref{conclusion}. We organize the paper as follows. Section \ref{sec : formulation of the problem} provides the required notations as well as the set-up of the problem. Then we expose in Section \ref{sec : computer-assisted approach} the construction of the approximate solution $u_0$ as well as the approximate inverse $\mathbb{A}_T$. Moreover, we establish the computer-assisted strategy for the (constructive) proof of solitary waves. In Section \ref{sec : computation of the bounds}, we provide explicit computations of bounds which are required in the computer-assisted approach. Combining these analytic bounds with rigorous numerics, we prove the existence of four solitary waves, three in the  {cgWE } $(T = 0.25, 0.5, 3)$ and one in the  {WE }$(T=0)$ (cf. Theorems \ref{th : proof whitham} and \ref{th: : proof capillary whitham}).  {Then, using the estimates of Section \ref{sec : computation of the bounds}, we provide an existence proof of a branch of solitary waves in the WE, parametrized by the velocity $c$.}
 Finally, we expose in Section \ref{sec : stability} a methodology to control the spectrum of the linearized operator. Proofs of spectral stability are obtained for the aforementioned solutions. The codes to perform the computer-assisted proofs are available at \cite{julia_cadiot}. In particular, all computational aspects are implemented in Julia (cf. \cite{julia_fresh_approach_bezanson}) via the package RadiiPolynomial.jl (cf. \cite{julia_olivier}) which relies on the package IntervalArithmetic.jl (cf. \cite{julia_interval}) for rigorous floating-point computations.

\section{Formulation of the problem}\label{sec : formulation of the problem}

 Recall the Lebesgue notation $L^2 = L^2(\mathbb{R})$ or $L^2(\Omega_0)$ on a bounded domain $\om$. More generally, $L^p$ denotes the usual $p$ Lebesgue space on $\mathbb{R}$ associated to its norm $\| \cdot \|_{p}$. For a bounded linear operator $\mathbb{B} : L^2 \to L^2$, we define $\mathbb{B}^*$ as the adjoint of $\mathbb{B}$ in $L^2$. Moreover, if $v \in L^2$, we define $\hat{v} \bydef \mathcal{F}(v)$ as the Fourier transform of $v$. More specifically,  $\displaystyle \hat{v}(\xi)  \bydef  \int_{\mathbb{R}}v(x)e^{-2\pi i x\xi}dx$ for all $\xi \in \mathbb{R}.$ 
Given $u \in L^\infty$, denote by
\begin{align}\label{eq : multiplication operator L2}
    \mathbb{u} \colon L^2 &\to L^2\\
    v &\mapsto \mathbb{u}v \bydef uv
\end{align} 
the linear multiplication operator associated to $u$.
Finally, given $u, v \in L^2$, we denote  $u*v$ the convolution of $u$ and $v$.

In this paper, given $c \in \R$ and $T\geq 0$, we wish to study the existence of solutions to 
\begin{align*}
    {\mathbb{M}_T}u -cu + u^2 =0
\end{align*}
such that $u$ is even, $u(x) \to 0$ as $|x| \to \infty$ and where $\mathbb{M}_T$ is defined in \eqref{def : Kernel whitham}.  

Having in mind a set-up similar as the one presented in \cite{unbounded_domain_cadiot}, we impose the invertibility of the linear operator $\mathbb{M}_T - cI_d$.
Equivalently, we require $|l(\xi)|>0$ for all $\xi \in \R$. This assumption is essential in our analysis and will make sense later on (see \eqref{def : definition Hl} or Section \ref{sec : decay of the kernel operator} for instance). 
The following lemma provides values of $c$ and $T$ for which such a condition is satisfied.

\begin{lemma}\label{lem : value for c and T}
If $T >0$, then there exists $c_T\leq 1$ such that if $c < c_T$, then  $m_T(\xi) - c >0$
for all $\xi \in \mathbb{R}$.  If $T=0$ and if $c>1$ or $c<0$, then $|m_0(\xi) - c| >0$ for all $\xi \in \R$.
\end{lemma}
\begin{proof}
The proof is a simple analysis of the function $m_T$. When $T \geq \frac{1}{3}$, the minimum of $m_T$ is reached at $0$ and has a value of 1. When $0 < T< \frac{1}{3}$, the minimum is reached at $x = \pm x_T$, for some $x_T >0$, and has a value $c_T \bydef m_T(x_T) <1.$ Finally, if $T=0$, notice that $m_0$ has a global maximum at $0$ and $m_0(0) = 1$. Moreover $m_0 \geq 0$ and $m_0(\xi) \to 0$ as $\xi \to \infty$. Therefore, if $c>1$ or if $c<0$, then $|m_0(\xi) - c| >0$ for all $\xi \in \R$.
\end{proof}

In order to ensure that $|l| >0$, we require the following Assumption \ref{ass : value of c and T}, which is justified by Lemma \ref{lem : value for c and T}.
\begin{assumption}\label{ass : value of c and T}
    Assume that $T \geq 0$ and $c \in \R$ are chosen such that :\\
   \noindent \underline{Case $T>0$} : If $T >0$, then assume $c <c_T$, where $c_T$ is defined in Lemma \ref{lem : value for c and T}.\\
   \noindent \underline{Case $T=0$} : If $T =0$, then assume $c >1$.
\end{assumption}
 {Now, using similar notations as \cite{unbounded_domain_cadiot}, let us define $\mathbb{L}$ as the linear part of $\mathbb{F}$ (cf. \eqref{eq : whitham stationary}) and $\mathbb{G}$ the nonlinear one, that is
\begin{align}
    \mathbb{L} \bydef 
        \mathbb{M}_T - cI_d 
    ~~ \text{ and } ~~  \mathbb{G}(u) \bydef
        u^2 
\end{align}}
At this point, we need to choose a space of functions for our analysis. Let us define $\nu > 0$ as
\begin{align}\label{def : nu}
    \nu \bydef \begin{cases}
        T &\text{ if } T>0\\
        \frac{4}{\pi^2} &\text{ if } T=0
    \end{cases}
\end{align}
and denote 
\begin{equation}\label{def : definition Lnu}
     {\mathbb{\Lambda}_{\nu}} \bydef I_d - \nu\Delta,
\end{equation}
where $\Delta$ is the one dimensional Laplacian. In particular,  {$ \mathbb{L}$ and $ {\mathbb{\Lambda}_{\nu}}$ have a symbol $l$ and $l_\nu : \mathbb{R} \to \mathbb{R}$ respectively given by 
\begin{align}\label{def : fourier transform linear part}
  l(2\pi \xi) \bydef m_T(2\pi\xi) - c ~~ \text{ and } ~~ l_\nu(2\pi\xi) \bydef 1+ \nu (2\pi\xi)^2 >0
\end{align}
 for all $\xi \in \mathbb{R}$}.  {Now, we slightly modify the definition of the classical Sobolev space $H^2 \bydef H^2(\R)$ to comply with the lifting operator $\mathbb{\Lambda}_\nu$. Specifically, we consider the norm $\|\cdot\|_{H^2}$ given by
\begin{align*}
    \|u\|_{H^2} = \|\mathbb{\Lambda}_\nu u \|_{2}
\end{align*}
for all $u \in H^2.$ Moreover, we introduce $\mathcal{H}$ as the  Hilbert space given by
\begin{align}\label{def : definition Hl}
    \mathcal{H} \bydef \{u \in L^2, ~~\|u\|_{\mathcal{H}} < \infty\}
\end{align}
and associated with the following inner product and norm
\begin{align}
\nonumber
(u,v)_{\mathcal{H}} \bydef (\mathbb{L}\mathbb{\Lambda}_\nu u,\mathbb{L}\mathbb{\Lambda}_\nu v)_2, ~~
    \|u\|_{\mathcal{H}} \bydef \|\mathbb{L}\mathbb{\Lambda}_\nu u\|_2 
\end{align}
for all $u,v \in \mathcal{H}$.
}
Note that,  under Assumption \ref{ass : value of c and T},  $|l(\xi)|>0$ (cf. \eqref{def : fourier transform linear part})  for all $\xi \in \R$. This provides the well-definedness of the Hilbert space $\mathcal{H}$. Now, we show that the additional regularity provided by $\mathbb{\Lambda}_\nu$ allows to obtain the well-definedness of the non-linearity $\G$. In the case $T>0$,  we have $\mathcal{H} = H^{\frac{5}{2}}(\mathbb{R})$ by equivalence of norms and $H^{\frac{5}{2}}(\mathbb{R})$ is a Banach algebra. In particular, it implies that $\mathbb{G} : \mathcal{H} \to H^2$ is a well-defined and  smooth operator. Similarly, if $T=0$, then we have $\mathcal{H} = H^{2}(\mathbb{R})$, which is also a Banach algebra. In particular, we obtain the following result.

\begin{lemma}\label{lem : banach algebra}
    If $\kappa_T >0$ satisfies 
    \begin{equation}\label{def : definition of kappa}
        \kappa_T \geq \frac{1}{\min_{\xi \in \R}|l(\xi)|^2\sqrt{\nu}},
    \end{equation}
    then $\|uv\|_{H^2} \leq \kappa_T \|u\|_{\mathcal{H}}\|v\|_{\mathcal{H}}$ for all $u, v \in \mathcal{H}$, where $ {\Lambda_{\nu}}$ is defined in \eqref{def : definition Lnu}.
\end{lemma}
\begin{proof}
    The proof follows the steps of Lemma 2.4 in \cite{unbounded_domain_cadiot}. Let $u, v \in \mathcal{H}$ and notice that 
    $1 + \nu(2\pi\xi)^2 \leq 2(1+\nu(2\pi x)^2) + 2(1+\nu(2\pi x-2\pi \xi)^2)$ for all $x, \xi \in \R$. Therefore,
    \begin{align}\label{eq : banach first step}
    \nonumber
         \left|l_\nu(2\pi\xi)\mathcal{F}\left( uv\right)(\xi)\right|
    &= \left|(1+\nu(2\pi\xi)^2)\left(\hat{u}*\hat{v}\right)(\xi)\right| \\ \nonumber
    &\le 2\int_{\mathbb{R}} |(1+\nu(2\pi x)^2)\hat{u}(x)\hat{v}(\xi-x)| dx  +  |\hat{u}(x)(1+\nu(2\pi(\xi-x))^2)\hat{v}(\xi-x)| dx\\ 
     \le 2\max_{\xi \in \R}&\frac{1+\nu(2\pi \xi)^2}{|l(2\pi\xi)|}\int_{\mathbb{R}} |l(2\pi x)\hat{u}(x)\hat{v}(\xi-x)| dx  +  |\hat{u}(x)l\left(2\pi(\xi-x)\right)\hat{v}(\xi-x)| dx.
    \end{align}
    Then, notice $\|l(2\pi\xi) l_\nu(2\pi\xi) \hat{u}(\xi)\|_2 = \|u\|_{\mathcal{H}}$ by Plancherel's identity and the definition of the norm on $\mathcal{H}$. Therefore, using Plancherel's identity again in \eqref{eq : banach first step} and  using Young's inequality for the convolution we get
\begin{align}\label{eq : banach last step}
    \|uv\|_{H^2} = \|l_\nu(2\pi\xi) (u*v)(\xi)\|_2 \leq 2\max_{\xi \in \R}\frac{1}{|l(\xi)|}\left( \|u\|_{\mathcal{H}} \|\hat{v}\|_1 + \|v\|_{\mathcal{H}} \|\hat{u}\|_1 \right).
\end{align}
Now, note that
 {
\begin{align}\label{eq : cauchy S lem banach}
    \|\hat{u}\|_1  = \left\|\frac{1}{l_\nu(2\pi \cdot)} (l_\nu(2\pi \cdot)\hat{u})\right\|_1  
    \leq \left\|\frac{1}{l_\nu(2\pi \cdot)}\right\|_2\|\frac{l(2\pi\cdot)}{l(2\pi\cdot)} l_\nu(2\pi \cdot)\hat{u}\|_2 \leq \max_{\xi \in \R}\frac{1}{|l(\xi)|} \left\|\frac{1}{l(2\pi \cdot)}\right\|_2\|u\|_{\mathcal{H}}.
\end{align}
Therefore, combining \eqref{eq : banach last step} and \eqref{eq : cauchy S lem banach}, we obtain \[ \|uv\|_{H^2} \leq 4\left\|\frac{1}{l_\nu(2\pi \cdot)}\right\|_2 \max_{\xi \in \R}\frac{1}{|l(\xi)|^2}.\] To conclude the proof it remains to compute $\|\frac{1}{l_\nu(2\pi \cdot)}\|_2$. We have
\begin{align}\label{eq : computation norm 1/l}
    \int_\R \frac{1}{l_\nu(2\pi\xi)^2}d\xi = \int_\R \frac{1}{(1+(2\pi\nu\xi)^2))^2}d\xi 
    &= \frac{1}{4\sqrt{\nu}}.
\end{align}}
\end{proof}

The previous lemma provides the explicit computation of a constant $\kappa_T$ such that $\|\G(u)\|_{H^2} \leq \kappa_T \|u\|_{\mathcal{H}}^2$. In particular, the value of $\kappa_T$ is essential in our computer-assisted approach  (cf. Lemma \ref{lem : bound Z_2} for instance). Moreover, we define 
\[
\mathbb{F}(u) \bydef \mathbb{L}u + \mathbb{G}(u)
\]
where $\mathbb{F} : \mathcal{H} \to H^2$.  In particular, the zero finding problem $\mathbb{F}(u) =0$ is well-defined on $\mathcal{H}.$ The condition $u  \to 0$ as $|x| \to \infty$ is satisfied implicitly if $u \in \mathcal{H}.$

Notice that the set of solutions to \eqref{eq : whitham stationary} possesses a natural translation invariance. In order to isolate this invariance, we choose to look for even solutions. Consequently, denote by $\mathcal{H}_e \subset \mathcal{H}$ the Hilbert subspace of $\mathcal{H}$ consisting of real-valued even functions
\begin{align}\label{def : definition of He}
    \mathcal{H}_e \bydef \{ u \in \mathcal{H},~ u(x) = u(-x)  \in \R \text{ for all } x \in \mathbb{R} \}.
\end{align} 
We similarly denote $H^2_e, L^2_e$ as the restriction of $H^2, L^2$ respectively to even functions.  In particular, notice that $l$ (defined in \eqref{def : fourier transform linear part}) is even so $\mathbb{L} u \in H^2_e$. Similarly $\mathbb{G}(u) \in H^2_e$ for all $u \in \mathcal{H}_e.$ Therefore we consider $\mL$, $\G$ and $\mathbb{F}$ as operator from $\mathcal{H}_e$ to $H^2_e$ 

\begin{remark}
    Note that the choice of the operator $ {\mathbb{\Lambda}_{\nu}}$ is justified by the fact that $ {\mathbb{\Lambda}_{\nu}}$ is not only an invertible differential operator, but also conserves the even symmetry. This point is essential in our set-up.
\end{remark}

Finally, we look for solutions of the following problem
\begin{equation}\label{eq : f(u)=0 on He}
    \mathbb{F}(u) = 0 ~~ \text{ and } ~~ u \in \mathcal{H}_e.
\end{equation}
In the case $T>0$, one can easily prove that solutions to \eqref{eq : f(u)=0 on He} are equivalently classical solutions of \eqref{eq : whitham stationary} using some bootstrapping argument (see \cite{unbounded_domain_cadiot}). In the case $T=0$, because $m_0(\xi) \to 0$ as $\xi \to \infty$, the same argument does not apply and an additional condition is needed to obtain the a posteriori regularity of the solution. We summarize these results in the next proposition.
\begin{prop}\label{prop : regularity of the solution}
Let $u \in \mathcal{H}_e$ such that $u$ solves ${\mathbb{F}}(u) = 0,$ then we separate two cases.

If $T>0$, then $u \in H^\infty(\R) \subset C^\infty(\mathbb{R})$ and $u$ is an even classical solution of \eqref{eq : whitham stationary}.\\
If $T=0$ and if in addition $u(x) < \frac{c}{2}$ for all $x \in \R$, then $u \in H^\infty(\R) \subset C^\infty(\mathbb{R})$ and $u$ is an even classical solution of \eqref{eq : whitham stationary}.
\end{prop}
\begin{proof}
The proof of the case $T>0$ follows identical steps as the one of Proposition 2.5 in \cite{unbounded_domain_cadiot}. For the case $T=0$, the proof is derived in \cite{EHRNSTROM2019on_whitham_conjecture}.
\end{proof}

The above proposition shows that if one looks for a smooth solution to \eqref{eq : whitham stationary}, then equivalently, one can study \eqref{eq : f(u)=0 on He} instead. Consequently, without loss of generality, the strong even solutions of \eqref{eq : whitham stationary} can be studied through the  zero finding problem \eqref{eq : f(u)=0 on He}. Note that, in the case $T=0$, one needs to verify a posteriori that $u(x) < \frac{c}{2}$ for all $x \in \R$ in order to obtain a smooth solution. We illustrate this point in Section \ref{ssec : approximate inverse}.

Finally, denote by $\|\cdot\|_{\mathcal{H},H^2}$ the operator norm for any bounded linear operator between the two Hilbert spaces $\mathcal{H}$ and $H^2$. Similarly denote by $\|\cdot\|_{\mathcal{H}}$, $\|\cdot\|_{H^2}$ and $\|\cdot\|_{H^2,\mathcal{H}}$ the operator norms for bounded linear operators on $\mathcal{H} \to \mathcal{H}$, $H^2\to H^2$ and $H^2 \to \mathcal{H}$ respectively.

\subsection{Periodic spaces}\label{sec : periodic spaces}

In this section we recall some notations on periodic spaces introduced in Section 2.4 of \cite{unbounded_domain_cadiot}. Indeed, the objects introduced in the previous section possess a Fourier series counterpart. We want to use this correlation in our computer-assisted approach to study \eqref{eq : whitham stationary}. Denote $$\Omega_0 \bydef (-d,d)$$  where $1<d <\infty$. $\Omega_0$ is the domain on which we which we construct the approximate solution $u_0$ (cf. Section \ref{ssec : approximate solution}).
Then, define $$\tilde{n} \bydef \frac{n}{2d} \in \mathbb{R}$$ for all $n \in \mathbb{Z}$. Denote by $\ell^p$  the usual $p$ Lebesgue space for sequences indexed on $\mathbb{Z}$ associated to its norm $\| \cdot \|_{\ell^p}$. Then, using coherent notations, we introduce $\mathscr{h}$ as the Hilbert space defined as 
\begin{align}\label{def : h}
     \mathscr{h} \bydef \{ U=(u_n)_{n\in \mathbb{Z}} \text{ such that $(U,U)_{\mathscr{h}} < \infty$} \}
 \end{align} with $(U,V)_{\mathscr{h}} \bydef \sum\limits_{\underset{}{n \in \mathbb{Z}}} u_n\overline{v_n}|l(2\pi\tilde{n})|^{2}$. 
 Similarly, we denote $(\cdot,\cdot)_{\ell^2}$ the usual inner product in $\ell^2.$ Moreover, denote $h^k$ the Hilbert space defined as 
 \[
 h^k \bydef \left\{U=(u_n)_{n\in \mathbb{Z}} \text{ such that } \sum_{n \in \mathbb{Z}} |u_n|^2(1+(2\pi\tilde{n})^2)^k < \infty\right\}.
 \]
 $h^k$ is the Fourier series equivalent of $H^k.$ Now, given a sequence  of Fourier coefficients $U = (u_n)_{n \in \mathbb{Z}}$ representing an even function, $U$  satisfies
\begin{align*}
    u_{n} = u_{-n} \text{ for all } n \in \mathbb{Z}.
\end{align*}
 {Consequently, for a given $p \geq 1$, we define $\ell^p_e$ as the following subset of $\ell^p$
\[
    \ell^p_{e} \bydef \left\{U = (u_n)_{n \in \mathbb{Z}} \in \ell^p, ~ u_n = u_{-n} \text{ for all } n \in \mathbb{Z}\right\}.
\]
}
Now, using the notations of \cite{unbounded_domain_cadiot}, we  define $\gamma : L^2 \to \ell^2$   and $\gamma^\dagger : \ell^2 \to L^2$ as
 {
\begin{align}\label{def : gamma}
  \left(\gamma(u)\right)_n \bydef  \frac{1}{|\om|}\int_\om u(x) e^{-2\pi i \tilde{n} x}dx ~~ \text{ and } ~~ \gamma^\dagger\left(U\right)(x) \bydef \cha(x) \sum_{n \in \mathbb{Z}} u_n e^{2\pi i \tilde{n}x}
\end{align}
for all $n \in \mathbb{Z}$, all $x\in \R$ and all $U = \left(u_n\right)_{n \in \mathbb{Z}} \in \ell^2$, where $\cha$ is the characteristic function on $\om$. Given $u \in L^2$, $\gamma(u)$ represents the Fourier coefficients  of the restriction of $u$ on $\om$.  Conversely, given a sequence $U\in \ell^2$ of Fourier coefficients, $\gdag\left(U\right)$  is the function representation of $U$ in $L^2$ with support contained in $\om$. In particular, notice that $\gdag\left(U\right)(x) =\gamma^\dagger\left(V\right)(x) = 0$ for all $x \notin \om.$  }
Then, define
\begin{align}\label{def : space for functions with support on omega}
    H_{\om} &\bydef \left\{u \in H ~: ~ \text{supp}(u) \subset \overline{\om} \right\}, ~ \text{ where $H$ is a Hilbert space of functions on $\R$}.
\end{align}
For instance, $L^{2}_{\om} = \{u \in L^2 ~: ~ \text{supp}(u) \subset \overline{\om} \}$ and $H^{l}_{e,\om} = \{u \in \mathcal{H}_e ~: ~ \text{supp}(u) \subset \overline{\om} \}$.

Then, given an Hilbert space $H$, denote by $\mathcal{B}\left(H\right)$ the set of bounded linear operators from $H$ to itself. Moreover, if $H$ is an Hilbert space on functions defined on $\R$, define $\mathcal{B}_\om\left(H\right) \subset \mathcal{B}\left(H\right)$ as 
\begin{equation}\label{def : Bomega}
    \mathcal{B}_\om(H) \bydef \{\mathbb{B}_\om \in \mathcal{B}(H) :  \mathbb{B}_\om = \cha \mathbb{B}_\om \cha\}.
\end{equation}
Finally, define $\Gamma : \mathcal{B}(L^2) \to \mathcal{B}(\ell^2)$ and $\Gamma^\dagger : \mathcal{B}(\ell^2) \to \mathcal{B}(L^2)$ as follows
\begin{equation}\label{def : Gamma and Gamma dagger}
    \Gamma(\mathbb{B}) \bydef \gamma \mathbb{B} \gdag ~~ \text{ and } ~~  \Gamma^\dagger(B) \bydef \gamma^\dagger {B} \gamma 
\end{equation}
for all $\mB \in \mathcal{B}(L^2)$ and all $B \in \mathcal{B}(\ell^2)$. Intuitively, given a bounded linear operator $\mathbb{B} : L^2 \to L^2$, $\Gamma(\mathbb{B})$ provides a corresponding bounded linear operator $B : \ell^2 \to \ell^2$ on Fourier coefficients. $\Gamma^\dagger$ provides the converse construction. Note that since supp$(\gamma^\dagger(U)) \subset \overline{\om}$, then $\mathbb{B} = \Gamma^\dagger\left(\Gamma(\mathbb{B})\right)$ if and only if $\mathbb{B} = \cha \mathbb{B} \cha$, that is if and only if $\mathbb{B} \in \mathcal{B}_\om(L^2)$. 

The maps defined above in \eqref{def : gamma} and \eqref{def : Gamma and Gamma dagger} are fundamental in our analysis as they allow to pass from the problem on $\R$ to the one in  $\ell^2$ and vice-versa. Furthermore, we provide in the following lemma, which is proven in \cite{unbounded_domain_cadiot} using Parseval's identity, that this passage is actually an isometric isomorphism when restricted to the relevant spaces.
\begin{lemma}\label{lem : gamma and Gamma properties}
    The map $\sqrt{|\om|} \gamma : L^2_{\om} \to \ell^2$ (resp. $\Gamma : \mathcal{B}_\om(L^2) \to \mathcal{B}(\ell^2)$  is an isometric isomorphism whose inverse is given by $\frac{1}{\sqrt{|\om|}} \gamma^\dagger: \ell^2 \to L^2_{\om}$   (resp. $\Gamma^\dagger :   \mathcal{B}(\ell^2) \to \mathcal{B}_\om(L^2)$. In particular,
    \begin{align*}
        \|u\|_2 = \sqrt{|\om|}\|U\|_2 ~\text{ and }~ \|\mathbb{B}\|_2 = \|B\|_2
    \end{align*}
    for all $u \in L^2_\om$ and all $\mathbb{B} \in \mathcal{B}_\om(L^2)$, where $U \bydef \gamma(u)$ and $B \bydef \Gamma(\mathbb{B})$.
\end{lemma}

The above lemma not only provides a one-to-one correspondence between the elements in $L^2_\om$ (resp. $\mathcal{B}_\om(L^2)$) and the ones in $\ell^2$ (resp. $\mathcal{B}(\ell^2)$) but it also provides an identity on norms. In particular, given a bounded linear operator $B : \ell^2_e \to \ell^2_e$, which has been obtained numerically for instance, then Lemma \ref{lem : gamma and Gamma properties} provides that $\mathbb{B} \bydef \Gamma^\dagger(B)$  {satisfies} $\|\mathbb{B}\|_2 = \|B\|_2$. Consequently, the previous lemma provides a convenient strategy to build bounded linear operators on $L^2_e$, from which norms computations can be obtained throughout their Fourier coefficients counterpart. This property is essential in our construction of an approximate inverse  in Section \ref{ssec : approximate inverse}.

Now, given $k>0$, we define the Hilbert spaces $\mathscr{h}_e$ and  $h^k_e$ as 
\begin{align*}
    \mathscr{h}_e \bydef \mathscr{h}\cap \ell^2_e ~~ \text{ and } ~~ h^k_e \bydef h^k\cup \ell^2_e.
\end{align*}
Such spaces allow us to define the Fourier coefficients version of the operators introduced earlier.
Denote by $L : \mathscr{h}_e \to \ell^2_e$, $M_T : \mathscr{h}_e \to \ell^2_e$, $ {\Lambda_{\nu}} : \mathscr{h}_e \to \ell^2_e$ and $G : \mathscr{h}_e \to \ell^2_e$ the Fourier coefficients representations of $\mathbb{L}$, $\mathbb{M}_T$, $  {\mathbb{\Lambda}_{\nu}}$ and $\mathbb{G}$ respectively. More specifically, $L$, $M_T$ and ${\Lambda}_\nu$ are  infinite diagonal matrices with, respectively, coefficients $\left(l(2\pi\tilde{n})\right)_{n\in \mathbb{Z}}$,  $\left(m_T(2\pi\tilde{n})\right)_{n\in \mathbb{Z}}$ and  $\left({l}_\nu(2\pi\tilde{n})\right)_{n\in \mathbb{Z}}$ on the diagonal.  {Moreover, we have $G(U) =  U*U$, where $U*V$ is defined as the usual discrete convolution given by
\begin{align}\label{def : discrete convolution usual form}
    (U*V)_n = \sum_{k \in \mathbb{Z}}u_{n-k}v_k = \left(\gamma\left(\gamma^\dagger\left(U\right)\gamma^\dagger\left(V\right)\right)\right)_n
\end{align}
for all $n \in \mathbb{Z}.$} In other terms, $U*V$ is the sequence of Fourier coefficients of the product $uv$ where $u$ and $v$ are the function representation of $U$ and $V$ respectively. In particular, notice that Young's inequality for the convolution is applicable and we have
\begin{align}\label{eq : youngs inequality}
    \|U*V\|_2 \leq \|U\|_2 \|V\|_1
\end{align}
for all  {$U \in \ell^2$ and $V\in \ell^1.$} Now, we define $F(U) \bydef LU + G(U)$ and introduce 
\begin{equation}\label{eq : F(U)=0 in h}
    F(U) =0 ~~ \text{ and } ~~ U \in \mathscr{h}_e
\end{equation}
as the periodic equivalent on $\om$ of \eqref{eq : f(u)=0 on He}. Finally, similarly as in \eqref{eq : multiplication operator L2}, given $U \in \ell^1$, we define the linear  discrete  convolution operator
\begin{align}\label{def : discrete conv operator}
    \mathbb{U} \colon \ell^2 &\to \ell^2\\
                    V &\mapsto \mathbb{U}V \bydef U*V.       
\end{align}
 Finally, we slightly abuse notation and  {denote by $\partial_x$ the linear operator given by 
 \begin{align*}
    \partial_x \colon &h^1 \to \ell^2\\
    &U \mapsto \partial_x U = (2\pi i\tilde{n}u_n)_{n\in \mathbb{Z}}.
\end{align*}
In other words, given $U \in h^1$ and $u_{p}$ its function representation on $\om$, then $\partial_x U$ is the sequence of  Fourier coefficients of $\partial_x u_{p}.$}

 {
\begin{remark}
    Note that sequences in $\ell^2_e$ can be represented by their restriction to the reduced set $\mathbb{N}\cup\{0\}$ using the symmetry $u_n=u_{-n}$. Indeed, there is a bijection between $\ell^2_e$ and $\ell^2(\mathbb{N}\cup\{0\})$. Consequently, we numerically store finite sequences in $\ell^2_e$ as finite sequences in $\ell^2(\mathbb{N}\cup\{0\})$ to gain computer memory. The same idea applies for operators on $\ell^2_e \to \ell^2_e$, which can be stored as operators on $\ell^2(\mathbb{N}\cup\{0\}) \to \ell^2(\mathbb{N}\cup\{0\}).$ Finally, the even symmetry provides an isometry between $\ell^2_e$ and $\ell^2(\mathbb{N}\cup\{0\})$, when $\ell^2(\mathbb{N}\cup\{0\})$ is associated with the following norm
    \[
    \|U\|_{\ell^2(\mathbb{N}\cup\{0\})}^2 =  |u_0|^2 + 2 \sum_{n \in \mathbb{N}} |u_n|^2. 
    \]
   Such an isometry provides a natural way to reduce numerical complexity. 
\end{remark}

}

\section{Computer-assisted approach}\label{sec : computer-assisted approach}

In this section, we present a  Newton-Kantorovich approach and the construction of the required objects to apply it. More specifically, we want to turn the zeros of  \eqref{eq : f(u)=0 on He} into fixed points of some contracting operator $\mathbb{T}$ defined below.

Let $u_0 \in  \mathcal{H}_e$, such that supp$(u_0) \subset \overline{\om}$, be an approximate solution of \eqref{eq : f(u)=0 on He}. Given a  bounded injective linear operator $\mathbb{A}_T : H^2_e \to \mathcal{H}_e$, which will be defined in Section \ref{ssec : approximate inverse}, we want to prove that there exists $r>0$ such that $\mathbb{T} : \overline{B_r(u_0)} \to \overline{B_r(u_0)}$ defined as
\[
\mathbb{T}(u) \bydef u - \mathbb{A}_T\mathbb{F}(u)
\]
is well-defined and is a contraction, where ${B_r(u_0)} \subset \mathcal{H}_e$ is the open ball centered at $u_0$ and of radius $r$.  Note that we explicit the dependency in $T$ of $\mathbb{A}_T$ as we will need to separate the cases $T=0$ and $T>0$. In order to determine a possible value for $r>0$ that would provide the contraction and the well-definedness of $\mathbb{T}$, we want to build $\mathbb{A}_T : H^2_e \to \mathcal{H}_e$, $\mathcal{Y}_0, \mathcal{Z}_1$ and $\mathcal{Z}_2 >0$ in such a way that the hypotheses of the following Theorem \ref{th: radii polynomial} are satisfied.
\begin{theorem}\label{th: radii polynomial}
Let $\mathbb{A}_T : H^2_e \to \mathcal{H}_e$ be an injective bounded linear operator and let $\mathcal{Y}_0, \mathcal{Z}_1$ and $\mathcal{Z}_2$ be non-negative constants such that
  \begin{align}\label{eq: definition Y0 Z1 Z2}
  \nonumber
    \|\mathbb{A}_T{\mathbb{F}}(u_0)\|_{\mathcal{H}} &\leq \mathcal{Y}_0\\
    \|I_d - \mathbb{A}_TD{\mathbb{F}}(u_0)\|_{\mathcal{H}} &\leq \mathcal{Z}_1\\\nonumber
    \|\mathbb{A}_T\left({D}{\mathbb{F}}(v) - D{\mathbb{F}}(u_0)\right)\|_{\mathcal{H}} &\leq \mathcal{Z}_2r, ~~ \text{for all } v \in \overline{B_r(u_0)}.
\end{align}  
If there exists $r>0$ such that
\begin{equation}\label{condition radii polynomial}
    \mathcal{Z}_2r^2 - (1-\mathcal{Z}_1)r + \mathcal{Y}_0 < 0,
 \end{equation}
then there exists a unique $\tilde{u} \in \overline{B_r(u_0)} \subset \mathcal{H}_e$ solving \eqref{eq : f(u)=0 on He}.
\end{theorem}
\begin{proof}
The proof can be found in \cite{van2021spontaneous}.
\end{proof}

\subsection{Construction of an approximate solution \texorpdfstring{$u_0$}{u0}}\label{ssec : approximate solution}

In order to use Theorem \ref{th: radii polynomial}, one first needs to build an approximate solution $u_0 \in \mathcal{H}_e$. We actually need to add some additional constraints on $u_0$ which will be necessary in order to perform a computer-assisted approach. These requirements are presented in this section.

We begin by introducing some notation. Define $H^4_e(\R)$  as the subspace of $H^4(\R)$ restricted to even functions, that is
\[
H^4_e(\R) \bydef \{u \in H^4(\R), ~ u(x) = u(-x) \text{ for all } x \in \R\}.
\]
Now, let us fix $N \in \mathbb{N}$. $N$ represents the size of the numerical problem ; that is $N$ is the size of the Fourier series truncation. Then, we introduce the following projection operators
 {
 \begin{align}\label{def : projection truncation}
     \left(\pi^N(U)\right)_k  &=  \begin{cases}
          u_k,  & |k| \leq N \\
              0, & |k| > N
    \end{cases} ~~ \text{ and } ~~
     \left(\pi_N(U)\right)_k  =  \begin{cases}
          0,  & |k| \leq N \\
              u_k, & |k| > N
    \end{cases}
 \end{align}
  for all $k \in \mathbb{Z}$ and all  $U = (u_k)_{k \in \mathbb{Z}} \in \ell^2$. 
    In particular if $U$ satisfies $U = \pi^{N} U$, it means that $U$  only has a finite number of non-zero coefficients ($U$ can be seen as a vector).}

 Now that the required notations are introduced, we can present the construction of $u_0$. In practice, we start by numerically computing the Fourier coefficients $\tilde{U}_0 \in \ell^2_e$ of an approximate solution. In particular, because $\tilde{U}_0$ is numerical, we have $\tilde{U}_0 = \pi^N \tilde{U}_0$ ($\tilde{U}_0$ is represented as a vector on the computer). The construction of $\tilde{U}_0$ can be established using different numerical approaches. In our case, $\tilde{U}_0$ is obtained using the relationship that exists between the KdV equation and the WE . Indeed, it is well-known that the KdV equation $u'' - \alpha u + \gamma u^2=0$ provides approximate  {solitary waves} for the  {WE }when the constant $c$ in \eqref{eq : whitham stationary} is close to $1$ (see \cite{johnson_generalized} or \cite{Stefanov2018SmallAT} for instance). In particular, using the known soliton solutions in $\text{sech}^2$ in the KdV equation, we can obtain a first approximate solution $\tilde{U}_0 \in  X^4_e$. Then, we refine this approximate solution using a Newton method and obtain a new approximate solution that we still denote $\tilde{U}_0 \in  X^4_e$. If needed, continuation can be used to move along branches and reach the desired value for $c$ and $T$.
 
 Then, notice that $\tilde{u}_0 \bydef \gamma^\dagger\left(\tilde{U}_0\right) \in L^2_{e,\om}$ but $\tilde{u}_0$ is not necessarily smooth. In order to ensure smoothness on $\R$, we need to project $\tilde{u}_0$ into the space of functions having a null trace at $x= \pm d$. 
 In terms of regularity, we need $u_0 \in \mathcal{H}_{e,\om}$  in order to apply Theorem \ref{th: radii polynomial}. Moreover,  Lemma \ref{lem : bound Y_0} presented below requires $u_0 \in H^4_{e,\om}(\R)$ to be applicable (that is $u_0 \in H^4_e(\R)$ and supp$(u_0) \subset \overline{\om}$). Noticing that $\mathcal{H}_e \subset H^4_e(\R)$, it is enough to construct $u_0 \in H^4_{e,\om}(\R)$. Therefore, using Section 4 from \cite{unbounded_domain_cadiot}, we need to ensure that the trace of order 4 of $u_0$ at $d$ is null. Equivalently, we need to project $\tilde{u}_0$ into the set of functions with a null trace of order 4 and define $u_0$ as the projection.  

 First, note that if $u$ has a null trace of order 4, then equivalently  $u(\pm d) = u'(\pm d) = u''(\pm d) = u'''(\pm d) = 0$. If in addition $u = \gamma^\dagger\left(U\right)$ for some $U \in \ell^2_e$ such that $U = \pi^N U$, then $u$ is even and periodic, and the null trace conditions reduce to $u(d)  = u''(d) = 0$. Now, note that these two equations can equivalently be written thanks to the coefficients $U$. Indeed, define  { $\mathcal{T} = \mathcal{T}_0 \times \mathcal{T}_2 : \ell^2_e \to \R^2$ as 
\begin{align*}
    \mathcal{T}_j(V) \bydef \displaystyle\sum_{|n| \leq N } v_{n} (-1)^n\big(\frac{\pi n}{d}\big)^j  
\end{align*}
for all $V = (v_n)_{n \in \mathbb{Z}} \in \ell^2_e$ and $j \in \{0, 2\}$.} If $U = \pi^N U$ and if $\mathcal{T}(U) =0$, then $\gamma^\dagger\left(U\right)(d) = \gamma^\dagger\left(U\right)''(d) = 0$. Therefore, given our approximate solution $\tilde{U}_0 \in X^4_e$ such that $U_0 = \pi^N U_0$, then by projecting $\tilde{U}_0$ in the kernel of $\mathcal{T}$, we obtain that the function representation of the projection is smooth on $\R$. Now, notice that $\mathcal{T}$ can be represented by a $2$ by $N+1$ matrix. We abuse notation and identify $\mathcal{T}$ by its matrix representation. Following the construction presented in \cite{unbounded_domain_cadiot}, we define $D$ to be the diagonal matrix with entries  { $\left(\frac{1}{l(2\pi\tilde{n})}\right)_{|n| \leq N}$ } on the diagonal and  we build a projection  $U_0$ of $\tilde{U}_0$ in the kernel of ${\mathcal{T}}$ defined as
\begin{equation}\label{eq : projection in h^k_0}
    U_0  \bydef \tilde{U}_0 - D\mathcal{T}^*(\mathcal{T}D\mathcal{T}^*)^{-1}\mathcal{T}\tilde{U}_0,
\end{equation}
where $\mathcal{T}^*$ is the adjoint of $\mathcal{T}.$ We abuse notation in the above equation as $U_0$ and $\tilde{U}_0$ are seen as vectors in $\R^{N+1}.$ Specifically, \eqref{eq : projection in h^k_0} is implemented numerically  using interval arithmetic (cf. IntervalArithmetic on Julia \cite{julia_interval,  julia_fresh_approach_bezanson}). This provides a computer-assisted approach to rigorously construct vectors in Ker$(\mathcal{T})$.
Consequently, we obtain that $U_0 = \pi^N U_0$  and that $u_0 \bydef \gamma^\dagger\left(U_0\right) \in H^4_{e,\om}(\R)$ by construction. In particular, if $\mathcal{T}\left(\tilde{U}_0\right)$ is small, then $U_0$ and $\tilde{U}_0$ are close in norm.

For the rest of the article, we assume that $u_0 = \gamma^\dagger\left(U_0\right) \in H^4_{e,\om}$ with $U_0 \in X^4_e$ satisfying $U_0 = \pi^N U_0$.

\subsection{Construction of the approximate inverse \texorpdfstring{$\mathbb{A}_T$}{AT}}\label{ssec : approximate inverse}

The second ingredient that we require for the use of Theorem \ref{th: radii polynomial} is the linear operator $\mathbb{A}_T$. This operator approximates the inverse of $D\mathbb{F}(u_0)$. The accuracy of the approximation is controlled by the bound $\mathcal{Z}_1$ in \eqref{eq: definition Y0 Z1 Z2}, which, in practice, needs to be strictly smaller than $1$. In this section, we present how to define the operator $\mathbb{A}_T$. The construction is based on the theory exposed in \cite{unbounded_domain_cadiot} but we also develop a new strategy for the case $T=0$.

In fact, when $T=0$,  an extra assumption on $u_0$ is required for the construction of the approximate inverse $\mA_0$ (cf. Section \ref{subsec : AT in the case T=0}). As presented in the introduction, we have $\mathbb{M}_0 - cI_d + 2\mathbb{u}_0 \to -cI_d + 2\mathbb{u}_0$ as the Fourier transform variable $\xi$ goes to infinity. This implies that we need to ensure that $-cI_d + 2\mathbb{u}_0$ is invertible if we want to build an approximate inverse. Assumption \ref{ass : u0 is smaller than} takes care of that condition by imposing that $c - 2u_0$ is uniformly bounded away from zero.

\begin{assumption}\label{ass : u0 is smaller than}
    If $T=0$, assume that there exists $\epsilon >0$ such that $u_0(x) + \epsilon \leq \frac{c}{2}$ for all $x \in \R$.
\end{assumption}

Under Assumption \ref{ass : u0 is smaller than}, we can generalize the theory developed in Section 3 of \cite{unbounded_domain_cadiot} and build an approximate inverse $\mathbb{A}_0$. For the case $T>0$, we recall the construction developed in \cite{unbounded_domain_cadiot}.

\subsubsection{The case \texorpdfstring{$T>0$}{T>0}}

Our construction of an approximate inverse is based on the fact that $\mathbb{L}$ is an isometric isomorphism between $\mathcal{H}_e$ and $H^2_e$. Therefore, we can equivalently build an approximate inverse for $D\mathbb{F}(u_0)\mathbb{L}^{-1} = I_d + D\mathbb{G}(u_0)\mathbb{L}^{-1} : H^2_e \to H^2_e$.  {Then, using that $\mathbb{\Lambda}_\nu : H^2_e \to L^2_e$ is an isometric isomorphism, we equivalently build an approximate inverse for  $\mathbb{\Lambda}_\nu D\mathbb{F}(u_0)\mathbb{L}^{-1}\mathbb{\Lambda}_\nu^{-1} = I_d + \mathbb{\Lambda}_\nu D\mathbb{G}(u_0)\mathbb{L}^{-1}\mathbb{\Lambda}_\nu^{-1}  : L^2_e \to L^2_e$}
Then, using that Fourier series form a basis for $L^2(\om)$, we hope to approximate the inverse of the aforementioned operator thanks to Fourier series operators. The construction goes as follows.  

First, we build numerically an operator $A_T^N : h^2_e \to \mathscr{h}_e$ approximating the inverse of $\pi^N DF(U_0)\pi^N $, where $\pi^N$ is defined in \eqref{def : projection truncation}. In particular, we choose $A_T^N$ is such a way that $A_T^N = \pi^N A_T^N \pi^N$, that is $A_T^N$ is represented numerically by a square matrix of size $2N+1$. Then, define $A_T : h^2_e \to \mathscr{h}_e$ as 
\begin{align}\label{def : definition of AT periodic}
   A_T \bydef L^{-1}\pi_{N} + A_T^N 
\end{align}
and $B_T : \ell^2_e \to \ell^2_e$ as 
\[
B_T \bydef \pi_{N} + B_T^N
\]
where $B_T^N \bydef L\Lambda_s A_T^N \Lambda_s^{-1}.$ Intuitively, $A_T$ is an approximate inverse of $DF(U_0) : \mathscr{h}_e \to h^2_e$. Now, let us define 
\begin{align}\label{def : operator A for T>0}
    \mathbb{A}_T \bydef \mathbb{L}^{-1}\mathbb{\Lambda}_\nu^{-1} \mathbb{B}_T \mathbb{\Lambda}_\nu  ~ \text{ and } ~ \mathbb{B}_T \bydef \out + \Gamma^\dagger\left(B_T\right),
\end{align}
where $\out$ is the characteristic function on $\R \setminus \om$.
Using Lemma \ref{lem : gamma and Gamma properties}, $\mathbb{B}_T$ is well-defined as a bounded linear operator on $L^2_e$. Moreover, since $\mathbb{L} : \mathcal{H}_e \to H^2_e$ and $\mathbb{\Lambda}_\nu : H^2_e \to L^2_e$ are  isometries, then  $\mathbb{A}_T : H^2_e \to \mathcal{H}_e$ is a well-defined bounded linear operator. We will need to prove that it is actually an efficient approximate inverse and that it is injective (cf. Theorem \ref{th: radii polynomial}). This will be accomplished in Lemmas \ref{lem : computation of Z1}, \ref{lem : usual term periodic Z1} and \ref{lem : lemma Zu}.

\subsubsection{The case \texorpdfstring{$T=0$}{T=0}}\label{subsec : AT in the case T=0}

In the case $T=0$, the construction derived in \cite{unbounded_domain_cadiot} does not  provide an accurate approximate inverse.
Intuitively, if $N$ is big and if $T>0$, then $ \pi_N DF(U_0) \approx \pi_N L$ since $m_T(\tilde{n}) \to \infty$ as $|n| \to \infty$ (where $\pi_N$ is defined in \eqref{def : projection truncation}). This justifies our choice in \eqref{def : definition of AT periodic} in which we chose $\pi_N A_T = L^{-1} \pi_N$. However, if $T=0$, then  {$\pi_N DF(U_0) \approx \pi_N \left(I_d - \frac{2}{c} \mathbb{U}_0\right){L}$}, where $\mathbb{U}_0$ is the discrete convolution operator (cf. \eqref{def : discrete conv operator}) associated to $U_0$. Consequently, the tail of the operator  { $A_0 : h^2_e \to \mathscr{h}_e$ needs to approximate the inverse of  $\left(I_d - \frac{2}{c} \mathbb{U}_0\right){L}$, instead of $\pi_N L$ when $T>0$}. In particular,  the construction of $A_T$ in \eqref{def : definition of AT periodic} cannot provide an accurate approximation in the case $T=0.$ Consequently, based on the ideas developed in \cite{breden_polynomial_chaos}, we construct ${A}_0$ in such a way that its tail combines a multiplication operator, approximating the inverse of $I_d - \frac{2}{c}\mathbb{U}_0$, composed with $L^{-1}$. 

First, under Assumption \ref{ass : u0 is smaller than}, note that $ -c + 2u_0(x) < -2\epsilon$ for all $x \in \R$. Therefore, the operator $I_d - \frac{2}{c} \mathbb{U}_0 : \ell^2_e \to \ell^2_e$ is invertible and has a bounded inverse. Now, let $e_0 \in \ell^2_e$ be defined as 
\begin{align}\label{def : definition of e0}
    (e_0)_n \bydef \begin{cases}
        1 \text{ if  } n = 0\\
        0 \text{ otherwise}.
    \end{cases}
\end{align}
In particular, notice that $cI_d - \frac{2}{c} \mathbb{U}_0 = \mathbb{e}_0  - \frac{1}{c} \mathbb{U}_0$ is the discrete convolution operator by $e_0 - \frac{2}{c} U_2 \in \ell^1_e.$ Then, we numerically build $W_0 \in \ell^1_e$ such that $W_0 = \pi^N W_0$ and 
\[
W_0*(e_0 - \frac{2}{c}U_0) \approx e_0.
\]
Equivalently, $\mathbb{W}_0*(\mathbb{e}_0 - \frac{2}{c}\mathbb{U}_0) \approx I_d$ ; that is, $\mathbb{W}_0$ approximates the inverse of $I_d- \frac{2}{c}\mathbb{U}_2.$ We are now in a position to describe the construction of $\mA_0$.

Similarly as for the case $T>0$, we first build $A_0^N : h^2_e \to \mathscr{h}_e$, approximating the inverse of $\pi^N DF(U_0) \pi^N $, such that  $A_0^N = \pi^N A_0^N \pi^N$. Then, define $A_0 : h^2_e \to \mathscr{h}_e$ as 
\[
A_0 \bydef L^{-1}\pi_N \mathbb{W}_0 + A_0^N
\]
and $B_0 : \ell^2_e \to \ell^2_e$ as 
\[
B_0 \bydef \pi_N \mathbb{W}_0 + B_0^N
\]
where $B_0^N \bydef L\Lambda_s A_0^N \Lambda_s^{-1}.$ By construction, $A_0$ approximates the inverse of $DF(U_0).$  Then, similarly as in \eqref{def : operator A for T>0}, define $\mathbb{A}_0$ as 
\begin{align}\label{def : operator A for T=0}
    \mathbb{A}_0 \bydef \mathbb{L}^{-1}\mathbb{\Lambda}_\nu^{-1} \mathbb{B}_0 \mathbb{\Lambda}_\nu  ~ \text{ and } ~ \mathbb{B}_0 \bydef \out + \Gamma^\dagger\left(B_0\right).
\end{align}
By construction, $\mathbb{B}_0 : L^2_e \to L^2_e$  and $\mathbb{A}_0 : H^2_e \to \mathcal{H}_e$ are well-defined bounded linear operators.

Note that the only difference in the construction of $B_T$ between the cases $T>0$ and $T=0$ is that $\pi_N B_0 = \pi_N \mathbb{W}_0$ if $T=0$ and $\pi_N B_T = \pi_N$ if $T>0$. However, notice that $\pi_N = \pi_N I_d = \pi_N \mathbb{e}_0$ where $e_0$ is defined in \eqref{def : definition of e0}. Therefore, defining $W_T \in \ell^1_e$ as 
\begin{align}\label{def : definition of WT}
    W_T \bydef \begin{cases}
        W_0 &\text{ if } T=0\\
        e_0 &\text{ if  } T>0,
    \end{cases}
\end{align}
we have \begin{align}\label{def : operator BT periodic}
    B_T = \pi_N \mathbb{W}_T + B^N_T
\end{align}
for all $T \geq 0$, leading to a more compact notation. In particular, we obtain the following result.
\begin{lemma}
    Let $T \geq 0$ and let $\mathbb{B}_T = \mathbb{1}_{\R \setminus \om}  + \Gamma^\dagger\left(\pi_N\mathbb{W}_T + B_T^N\right)$. Then, 
    \begin{align}\label{eq : equality B operator and B fourier}
        \|\mathbb{B}_T\|_2 \leq  \max\left\{1,~  \max\{\|B^N_T\|_2, \|W_T\|_1 \} + \|\pi_N\mathbb{W}_T \pi^N\|_2\right\}.
    \end{align}
\end{lemma}
\begin{proof}
    Using Lemma 3.3 in \cite{unbounded_domain_cadiot}, we have
    \begin{align*}
         \|\mathbb{B}_T\|_2 = \max\{1,  \|B_{T}\|_2\}.
    \end{align*}
    Now, notice that 
    \begin{align*}
        \|B_T\|_2 &= \|B^N_T + \pi_N\mathbb{W}_T \pi^N + \pi_N\mathbb{W}_T \pi_N \|_2 \\
        &\leq \|B^N_T + \pi_N\mathbb{W}_T \pi_N\|_2 + \|\pi_N\mathbb{W}_T \pi^N\|_2.
    \end{align*}
    Finally, we conclude the proof using that 
    \begin{align*}
        \|B^N_T + \pi_N\mathbb{W}_T \pi_N \|_2 = \max\{\|B^N_T\|_2, \|\pi_N\mathbb{W}_T \pi_N\|_2 \} \leq \max\{\|B^N_T\|_2, \|W_T\|_1 \}
    \end{align*}
    where we used \eqref{eq : youngs inequality} on the last step.
\end{proof}

\begin{remark}
    If $T>0$, then $W_T = e_0$ (defined in \eqref{def : definition of e0}). Therefore $\|\pi_N\mathbb{W}_T\pi^N\|_2 =0$ and $\|W_T\|_1 = 1$. This implies that $\|\mathbb{B}_T\|_2 \leq \max\{1, \|B^N_T\|_2\}$ if $T>0.$
\end{remark}

\subsubsection{A posteriori regularity of the solution}\label{ssec : regularity of the solution}

In the case $T=0$,  Assumption \ref{ass : u0 is smaller than} turns out to be helpful to obtain the regularity of the solution (cf. Proposition \ref{prop : regularity of the solution}) after using Theorem \ref{th: radii polynomial}. Indeed, we can validate a posteriori the regularity of the solution.
In order to do so, we first expose the following result.
\begin{prop}\label{prop : equivalence infinite norm}
    Let $u \in \mathcal{H}$ then\[
    \|u\|_\infty \leq \frac{1}{4\sqrt{\nu}\min_{\xi \in \R}|l(\xi)|}\|u\|_{\mathcal{H}}.
    \]
\end{prop}
\begin{proof}
    Let $u \in \mathcal{H}$, then notice that $\|u\|_\infty \leq \|\hat{u}\|_1$. Then, using \eqref{eq : cauchy S lem banach}, we have $
    \|\hat{u}\|_1 \leq \|\frac{1}{l}\|_\infty\|\frac{1}{l_\nu(2\pi \cdot)}\|_2 \|u\|_{\mathcal{H}}$. We conclude the proof using \eqref{eq : computation norm 1/l}.
\end{proof}

Now, suppose that, using Theorem \ref{th: radii polynomial}, we were able to prove a solution $\tilde{u}$ to \eqref{eq : whitham stationary}  such that $\tilde{u} \in \overline{B_r(u_0)}$ (for some $r>0$). Then it implies that $\|\tilde{u} - u_0\|_{\mathcal{H}} \leq r$ and therefore
\[
\|\tilde{u} - u_0\|_\infty \leq\frac{r}{4\sqrt{\nu}\min_{\xi \in \R}|l(\xi)|}
\]
using Proposition \ref{prop : equivalence infinite norm}. In particular, if $r$ is small enough so that 
\begin{equation}\label{eq : condition epsilon and r}
    r \leq 4\epsilon\sqrt{\nu}\min_{\xi \in \R}|l(\xi)|, 
\end{equation}
then $\tilde{u}(x) = u_0(x) + \tilde{u}(x) - u_0(x) \leq u_0(x) + \|\tilde{u} - u_0\|_\infty \leq u_0(x) + \epsilon < \frac{c}{2}$
using Assumption \ref{ass : u0 is smaller than}. Therefore, using Proposition \ref{prop : regularity of the solution}, we obtain that $\tilde{u}$ is smooth. In practice, the values for $r$ and $\epsilon$ are explicitly known in the computer-assisted approach, leading to a convenient way to prove the regularity of the solution (cf. Section \ref{sec : proof in the case T=0}).

Now that we derived a strategy to compute $u_0$ and $\mathbb{A}_T$ in Theorem \ref{th: radii polynomial}, it remains to compute the bounds $\mathcal{Y}_0$, $\mathcal{Z}_1$ and $\mathcal{Z}_2.$ We present  the required analysis for such computations in the next section.

\section{Computation of the bounds}\label{sec : computation of the bounds}

The strategy for computing of the bounds $\mathcal{Y}_0$, $\mathcal{Z}_1$ and $\mathcal{Z}_2$ in Theorem \ref{th: radii polynomial} has to differ from the PDE case introduced in \cite{unbounded_domain_cadiot}. Indeed, the fact that \eqref{eq : whitham stationary} possesses a Fourier multiplier operator implies that some of the steps derived in \cite{unbounded_domain_cadiot} have to be modified to match the current set-up. Consequently, we derive in this section a new strategy for the computation of $\mathcal{Y}_0$, $\mathcal{Z}_1$ and $\mathcal{Z}_2$ in the case of a nonlocal equation. In particular, we expose in Lemma \ref{lem : bound Y_0} a computer-assisted approach to control Fourier multiplier operator $\mathbb{M}_T$.

\subsection{Decay of the kernel operators}\label{sec : decay of the kernel operator}

In the computation of the bounds $\mathcal{Y}_0$ and $\mathcal{Z}_1$ presented in Lemmas \ref{lem : bound Y_0} and \ref{lem : computation of Z1} respectively, one needs to control explicitly the exponential decay of the following functions
 \begin{align}\label{def : functions fiT}
 \nonumber
 f_{\mathcal{Y}_0,T}& \bydef \mathcal{F}^{-1}\left(\frac{m_T(2\pi \cdot)}{l_{\nu}(2\pi \cdot)}\right)\\ \nonumber
      f_{0,T}& \bydef \mathcal{F}^{-1}\left(\frac{1}{l(2\pi \cdot)l_\nu(2\pi \cdot)}\right)\\ \nonumber
    f_{1,T}& \bydef \mathcal{F}^{-1}\left(\frac{2\pi \cdot}{l(2\pi \cdot)l_\nu(2\pi \cdot)}\right)\\ 
    f_{2,T}& \bydef \mathcal{F}^{-1}\left(\frac{1}{l(2\pi \cdot)}\right),
 \end{align}
 where we recall that $l_\nu(\xi) = 1+\nu\xi^2$ for all $\xi \in \R$ and $\nu$ is defined in \eqref{def : nu}.  {Note that estimating the exponential decay of the above functions have been achieved in \cite{BRUELL20174232} and \cite{EYCHENNE2023243} for instance. However, the aforementioned references do not provide explicit constants when establishing the estimates. In this section, we use rigorous numerics in order to resolve that problem.}
 Now, note that the above functions are related to some Fourier multiplier operators, which turn out to also be kernel operators. In particular, we have
 \begin{align}\label{eq : definition of fiT with operators}
 \nonumber
      \mathbb{M}_T {\mathbb{\Lambda}_{\nu}}^{-1} u &= f_{\mathcal{Y}_0,T}*u , ~~~~\mathbb{L}^{-1} {\mathbb{\Lambda}_{\nu}}^{-1}u = f_{0,T}*u \\
      \partial_x\mathbb{L}^{-1} {\mathbb{\Lambda}_{\nu}}^{-1}u &= f_{1,T}*u, ~~~~~~~~~~~~
     \mathbb{L}^{-1}u = f_{2,T}*u
 \end{align}
 for all $u \in L^2_e$ and $ {\mathbb{\Lambda}_{\nu}} = I_d - \nu \Delta$.   We first prove that the functions in \eqref{def : functions fiT} are analytic on some strip of the complex plane $S$. Moreover, we introduce some constants $\sigma_0, \sigma_1 >0$ that will be useful later on in Lemma \ref{lem : computation of f}.

\begin{prop}\label{prop : analyticity and value of a}
   Let $\nu$ be defined in \eqref{def : nu}. Then, there exists  $0<a < \min\{\frac{1}{\sqrt{\nu}}, \frac{\pi}{2}\}$  such that $|m_T(z)-c| >0$ for all $z \in S$ where $S \bydef \{z \in \mathbb{C},~ |\text{Im}(z)| \leq a\}$. Moreover, there exists $\sigma_0, \sigma_1 >0$ such that
\begin{align}\label{eq : alpha assumption}
\nonumber
    |l(\xi)|, |l(\xi+ia)| &\geq \sigma_0 \text{ for all } \xi \in \R,\\
   |l(\xi+ia)| &\geq \sigma_1\sqrt{T|\xi|} \text{ for all } |\xi| \geq 1.
 \end{align}
 In particular, $m_T, \frac{1}{l_\nu}$ and $\frac{1}{l}$ are analytic on $S$.
\end{prop}

\begin{proof}
    First, notice that $\xi \to \sqrt{\frac{\tanh(\xi)}{\xi}}$ is analytic on the strip $\{z \in \mathbb{C}, ~~ |\text{Im}(z)| < \frac{\pi}{2}\}$ (see \cite{EHRNSTROM2019on_whitham_conjecture}). Moreover, $\xi \to \frac{1}{(1+\nu\xi^2)}$ is analytic on the strip $S_0 \bydef \{z \in \mathbb{C}, ~~ |\text{Im}(z)| < \frac{1}{\sqrt{\nu}}\}$. Therefore, as $\nu \geq T$, we get that $\xi \to \frac{1}{(1+T\xi^2)}$ is also analytic on $S_0$. Therefore, if $a < \min\{\frac{1}{\sqrt{\nu}}, \frac{\pi}{2}\}$, then  ${m_T}$ and $\frac{1}{l_\nu}$ are analytic on the strip $S$. Moreover, as $|l(\xi)| >0$ for all $\xi \in \R$ under Assumption \ref{ass : value of c and T}, then there exists $0 < a < \min\{\frac{1}{\sqrt{\nu}}, \frac{\pi}{2}\}$ such that $|m_T(z)-c| >0$ for all $z \in S$. Defining such a constant $a$, it yields that $\frac{1}{l}$ is analytic on $S$. This implies that  $m_T, \frac{1}{l_\nu}$ and $\frac{1}{l}$ are analytic on $S$.

Now, we prove \eqref{eq : alpha assumption}. The existence of $\sigma_0$ is a direct consequence of the fact that $|m_T(z)-c| >0$ for all $z \in \{z \in \mathbb{C},~ |\text{Im}(z)| \leq a\}$. Then, if $T>0$, we have $|m_T(z) - c| = \mathcal{O}(\sqrt{T|z|})$ as $|z| \to \infty$. This provides the existence of $\sigma_1.$
\end{proof}

Using the previous proposition , we can prove that the functions defined in \eqref{def : functions fiT} are exponentially decaying to zero at infinity using Cauchy's integral theorem. In particular, we provide in the next lemma explicit constants controlling this exponential decay. The obtained constants are defined thanks to the existence of $a, \sigma_0$ and $\sigma_1$ in Proposition \ref{prop : analyticity and value of a}.

\begin{lemma}\label{lem : computation of f}
Let $\nu$ be defined in \eqref{def : nu} and let $a, \sigma_0, \sigma_1 >0$ be defined in Proposition \ref{prop : analyticity and value of a}. Moreover, if $T>0$, let $\xi_0>0$ be big enough so that 
\begin{align*}
     \frac{1}{2}\sqrt{\tanh(\xi_0)T\xi_0} \geq |c|, ~~~~
     \xi_0 \geq \max\{1, \frac{1}{\sqrt{T}}\}\\
      \frac{2\tanh(\xi_0)\xi_0}{3T} \geq 1, ~~~~
     \frac{1}{2}\left(\frac{1}{C_a}\right)^{\frac{1}{2}}\sqrt{T\xi_0} \geq |c|.
\end{align*}
If $T=0$, assume that $\xi_0 \geq 1$. Then, defining \begin{equation}
    a_0 \bydef \begin{cases}
    1 &\text{ if } T=0\\
    \min\left\{\frac{\pi}{2}, \frac{1}{\sqrt{T}}\right\} &\text{ if } T>0,
\end{cases}
\end{equation}
we have
\begin{align}\label{eq : exponential decay fi}
\nonumber
    |f_{\mathcal{Y}_0,T}(x)| &\leq C_{\mathcal{Y}_0,T}e^{-a_0 |x|}  \\
    |f_{0,T}(x)| &\leq C_{0,T} e^{-a|x|}, ~~
    |f_{1,T}(x)| \leq 
        C_{1,T}e^{-a|x|}\\ \nonumber
    |f_{2,T}(x)| &\leq  C_{2,T,c}\frac{e^{-a|x|}}{\sqrt{|x|}} \text{ if } T>0\\ 
    |f_{2,0}(x) + \frac{1}{c}\delta(x)| &\leq  C_{2,0,c}\frac{e^{-a|x|}}{\sqrt{|x|}} \text{ if } T=0
\end{align}
for all $x \neq 0$ where $\delta$ is the Dirac-delta distribution and 
\begin{align}
{C}_{\mathcal{Y}_0,T} &\bydef  \begin{cases}
    \displaystyle\max_{s >0}\min\left\{\sqrt{s}+\sqrt{2}, \frac{\sqrt{2}}{1-e^{-\pi s}}\right\} &\text{ if } T=0\\
    \displaystyle \displaystyle \max_{s >0}\min\left\{\sqrt{s}+\sqrt{2}, \frac{\sqrt{2}}{1-e^{-\pi s}}\right\}\frac{1}{\sqrt{\pi}T^{\frac{1}{4}}}\left(\frac{2}{\sqrt{1+{a_0}T^{\frac{1}{2}}}}+ \frac{1}{\sqrt{2}}\right) &\text{ if } T>0
\end{cases} \\
    C_{0,T} &\bydef \frac{1}{\pi\sigma_0(1-\nu a^2)} + \frac{1}{\pi\nu\sigma_0} \label{eq : constant C0T}\\
    C_{1,T} & \bydef \begin{cases}
        \frac{1}{2\pi}\left( \frac{2(1+a)}{\sigma_0(1-\nu a^2)}  + \frac{4(1+a)}{\sigma_1\sqrt{T} \nu}\right) & \text{ if } T>0\\
        \frac{1}{2c\nu} + \frac{(1+\sqrt{a})\left(C_a\right)^{\frac{1}{4}}}{\pi c\sigma_0}\left(\frac{1}{1-\nu a^2} + \frac{2}{\nu}\right) & \text{ if } T=0
    \end{cases}\label{eq : constant C1T}\\\nonumber
    K_{1,T,c} & \bydef \begin{cases}
        \frac{2\xi_0}{\pi\sigma_0}+ \frac{2\sqrt{\xi_0}(1+|c|)}{\pi \sigma_0\sqrt{T}} + \frac{2\left(\frac{1}{3T} + \frac{|c|}{4T^{\frac{3}{2}}} + \frac{2c^2}{T}\right)}{\pi\sqrt{\tanh(\xi_0)T}\sqrt{\xi_0}} + \frac{|c|}{T\pi}\left(2  + 3\ln(\xi_0)\right) + \frac{1}{\sqrt{2\pi T}} &\text{ if } T>0\\
       \frac{1}{\pi \min\{1,|c|^3\}\sqrt{\xi_0}\sigma_0}\left(2+4e^{-2\xi_0}\right) + \frac{1}{\pi \sigma_0 |c|} + \frac{2}{\pi c^2}  +  \frac{1}{\pi |c|^3}\left(2+3\ln(\xi_0)\right) + \frac{1}{c^2\sqrt{2\pi}} &\text{ if } T=0
    \end{cases} \\\nonumber
    K_{2,T} & \bydef \begin{cases}
        \frac{C_a\sqrt{\xi_0^2+a^2}}{2\pi\sigma_0^2(1-Ta^2)^2}\left( \frac{1+Ta^2+aC_a}{a^2} +C_aT\sqrt{\xi_0^2+a^2})\right) + \frac{2C_a^{2}}{\pi}\left( \frac{2(1 + T)}{(T|\xi_0|)^{\frac{1}{2}}} + \frac{(2+a)}{2}e^{-2\xi_0}\right) &\text{ if } T>0\\
        \frac{\left(C_a\right)^{\frac{1}{4}}}{\pi}\left( \frac{1}{2\sigma_0^2}(\frac{1}{a^{\frac{3}{2}}} + 2) + \frac{1}{4\sigma_0^2(1-|\cos(2a)|)^2\sqrt{a}}\right) &\text{ if } T=0
    \end{cases} \\
    C_{2,T,c} &\bydef \max\{K_{2,T}, K_{1,T,c}e^{a}\}\label{eq : constant C2T}
\end{align}
where $C_a \bydef \frac{1+|\cos(2a)|}{1-|\cos(2a)|}$.
\end{lemma}

\begin{proof}
    The proof is presented in the Appendix \ref{sec : proof of lemma 4.1 in appendix}.
\end{proof}

Notice, that the constants involved in Lemma \ref{lem : computation of f} depend on the values of $a, \sigma_0$ and $\sigma_1$, which we do not know explicitly. Sharp computation of such constants can be tedious and technically involving. We derive in the Appendix \ref{sec : computation of constants in appendix} a computer-assisted approach to have a pointwise control on the function $m_T$. In particular, the computation of $a, \sigma_0$ and $\sigma_1$ can be obtained thanks to rigorous numerics.   From now on, we consider that explicit values of $a$, $\sigma_0, \sigma_1$ have been obtained thanks to the strategy established in Appendix \ref{sec : computation of constants in appendix}. In particular, we are in a position to expose the computations of the bounds introduce in Theorem \ref{th: radii polynomial}.

\subsection{The bound \texorpdfstring{$\mathcal{Y}_0$}{Y0}}\label{sec : Y0 bound}

We present in this section the computation for the bound $\mathcal{Y}_0$. Since the operator $\mathbb{L}$ is not a linear differential operator, the computation of this bound naturally differs from the one exposed in  \cite{unbounded_domain_cadiot} in the PDE case. Indeed, given $U_0  \in \mathscr{h}_e$ such that $u_0 \bydef \gamma^\dagger\left(U_0\right) \in \mathcal{H}_{e,\om}$, if $\mathbb{L}$ is a linear differential operator with constant coefficients, then $\mathbb{L}u_0 = \gamma^\dagger\left(LU_0\right) \in L^2_e$. However, this is not necessarily the case  if $\mathbb{L}$ is nonlocal. Consequently, we need to estimate how close $\mathbb{L}u_0 \in L^2_e$ is to $\gamma^\dagger\left(LU_0\right)$. The next Lemma \ref{lem : bound Y_0} exposes the details for such a computation.
\begin{lemma}\label{lem : bound Y_0}
Let $a_0>0$ be defined in Lemma \ref{lem : computation of f}. Moreover, define $E_0 \in \ell^2_e$ as 
\begin{align}\label{eq : coefficients of E}
    \left(E_0\right)_n &\bydef \frac{e^{2a_0d}}{d} \frac{a_0(-1)^{n}(1-e^{-4a_0d})}{4a_0^2 + \tilde{n}^2}
\end{align}
for all $n \in \mathbb{Z}$. In particular, $E_0 = \gamma\left(\cha(x) \cosh(2a_0x)\right)$. Moreover, define $C(d) >0$ as 
\begin{align}\label{def : value for C(d)}
    C(d)  \bydef e^{-2a_0d}\left(4d + \frac{4e^{-a_0d}}{a(1-e^{-\frac{3a_0d}{2}})} + \frac{2}{a_0(1-e^{-2a_0d})}\right).
\end{align}
Now,  let $Y_0$  and $\mathcal{Y}_{u}$ be non-negative bounds satisfying 
\begin{align}\label{def : Y definition}
\nonumber
    Y_0 &\geq  \sqrt{2d}\left(\|B^N_T \Lambda_\nu F(U_0)\|_2^2 + \|\pi_N W_T*(\Lambda_\nu F(U_0))\|_2^2\right)^\frac{1}{2}\\ \nonumber
   \mathcal{Y}_{u} &\geq
       2d C_{\mathcal{Y}_0,T}  e^{-a_0d}\left(~\left( {\Lambda_{\nu}}^2 U_0,E_0*( {\Lambda_{\nu}}^2U_0)\right)_2\left(1 +  \|B_T\|_2^2(1 + C(d))\right)~\right)^{\frac{1}{2}}.
\end{align}
Then, defining $\mathcal{Y}_0 >0$ as  
\begin{equation}
    \mathcal{Y}_0 \bydef Y_0 + \mathcal{Y}_u,
\end{equation}
we obtain that $\|\mathbb{A}_T{\mathbb{F}}(u_0)\|_{\mathcal{H}} \leq \mathcal{Y}_0$. 
\end{lemma}


\begin{proof}
By definition of the norm on $\mathcal{H}_e$, we have that $\|u\|_{\mathcal{H}} = \|\mathbb{L}\mathbb{\Lambda}_\nu u\|_2$ for all $u \in \mathcal{H}_e$. This implies
\[
 {\|\mathbb{A}_T\mathbb{F}(u_0)\|_{\mathcal{H}} = \|\mathbb{L}\mln\mathbb{A}_T\mathbb{F}(u_0)\|_2 = \|\mathbb{B}_T \mln{\mathbb{F}}(u_0)\|_2}
\]
by definition of $\mathbb{B}_T$ in Section \ref{ssec : approximate inverse}. Now, since $u_0 = \gamma^\dagger(U_0) \in H^4_{e,\om}(\R)$ by construction in Section \ref{ssec : approximate solution}, we have that $\mathbb{G}(u_0) = \gamma^\dagger\left(G(U_0)\right)$. Moreover, we have 
\begin{align*}
   \|\mathbb{B}_T \mln\mathbb{F}(u_0)\|_2 &= \|\mathbb{B}_T(\mathbb{L} {\mathbb{\Lambda}_{\nu}} u_0 + \mln\mathbb{G}(u_0))\|_2\\
    &\leq \|\mathbb{B}_T\left( \mathbb{M}_T  {\mathbb{\Lambda}_{\nu}} - \Gamma^\dagger\left(M_T\right)  {\mathbb{\Lambda}_{\nu}} \right)u_0\|_2 + \|\mathbb{B}_T( (\Gamma^\dagger\left(M_T\right)-c) {\mathbb{\Lambda}_{\nu}} u_0 + \mln\mathbb{G}(u_0))\|_2.
\end{align*}
Now, using Lemma \ref{lem : gamma and Gamma properties} we get
\begin{align*}
    \|\mathbb{B}_T( (\Gamma^\dagger\left(M_T\right)-c) {\mathbb{\Lambda}_{\nu}} u_0 + \mathbb{G}(u_0))\|_2 &= \sqrt{|\om|} \|B_T \Lambda_\nu F(U_0)\|_2\\
    &=  \sqrt{2d} \left( \|B^N_T \Lambda_\nu F(U_0)\|_2^2 + \|\pi_N W_T* (\Lambda_\nu F(U_0))\|_2^2\right)^{\frac{1}{1}} \leq Y_0
\end{align*}
as $B_{T} = B^N_T + \pi_N W_T$ by construction in \eqref{def : operator BT periodic}.

Then, using that $\mathbb{B}_T = \out + \Gamma^\dagger\left(B_T\right)$ where $\Gamma^\dagger\left(B_T\right) = \cha \Gamma^\dagger\left(B_T\right) \cha $, we get
\begin{align}\label{ex : example in Y0}
\nonumber
     &\|\mathbb{B}_T\left( \mathbb{M}_T  {\mathbb{\Lambda}_{\nu}} - \Gamma^\dagger\left(M_T\right)  {\mathbb{\Lambda}_{\nu}} \right)u_0\|_2^2\\
     \leq ~& \|\out\left( \mathbb{M}_T  {\mathbb{\Lambda}_{\nu}} - \Gamma^\dagger\left(M_T\right)  {\mathbb{\Lambda}_{\nu}} \right)u_0\|_2^2 + \|\Gamma^\dagger\left({B}_{T}\right)\|_2^2\|\cha \left( \mathbb{M}_T  {\mathbb{\Lambda}_{\nu}} - \Gamma^\dagger\left(M_T\right)  {\mathbb{\Lambda}_{\nu}} \right)u_0\|_2^2\\
     = ~&\|\out\left( \mathbb{M}_T  {\mathbb{\Lambda}_{\nu}}^{-1} - \Gamma^\dagger\left(M_T\right)  {\mathbb{\Lambda}_{\nu}}^{-1} \right)h_0\|_2^2 + \|B_T\|_2^2\|\cha \left( \mathbb{M}_T  {\mathbb{\Lambda}_{\nu}}^{-1} - \Gamma^\dagger\left(M_T\right)  {\mathbb{\Lambda}_{\nu}}^{-1} \right)h_0\|_2^2
\end{align}
where we define $h_0 \bydef  {\mathbb{\Lambda}_{\nu}}^2 u_0$ and where we used Lemma \ref{lem : gamma and Gamma properties}. Notice that $h_0 \in L^2_{e,\om}$ as $u_0 \in H^4_{e,\om}(\R)$ by construction.  Then, recall that \begin{align*}
    \mathbb{M}_T  {\mathbb{\Lambda}_{\nu}}^{-1} h_0 = f_{\mathcal{Y}_0,T}*h_0
\end{align*}
by definition of $f_{\mathcal{Y}_0,T}$ in \eqref{eq : definition of fiT with operators}. Now, since $u_0 = \gamma^\dagger(U_0) \in H^4_{e,\om}(\R)$ by construction, then $$h_0~=~\gamma^\dagger( {\Lambda_{\nu}}^2U_0)~\in~L^2_{e,\om}(\R).$$ Therefore, letting $u = \cha$ and using Theorem 3.9 in \cite{unbounded_domain_cadiot}, we obtain that
\begin{align*}
   \|\out\left( \mathbb{M}_T  {\mathbb{\Lambda}_{\nu}}^{-1} - \Gamma^\dagger\left(M_T\right)  {\mathbb{\Lambda}_{\nu}}^{-1} \right)h_0\|_2^2 &=  \|\out \mathbb{M}_T  {\mathbb{\Lambda}_{\nu}}^{-1}h_0 u\|_2^2\\
   &\leq   C_{\mathcal{Y}_0,T}^2 |\om|  e^{-2a_0d}\left( {\Lambda_{\nu}}^2 U_0,E_0*( {\Lambda_{\nu}}^2U_0)\right)_2 \|\cha\|_2^2\\
   &= C_{\mathcal{Y}_0,T}^2 4d^2 e^{-2a_0d}\left( {\Lambda_{\nu}}^2 U_0,E_0*( {\Lambda_{\nu}}^2U_0)\right)_2 \bydef \left(\mathcal{Y}_{u,1}\right)^2,
\end{align*}
where we used that $\out \Gamma^\dagger(M_T) =0$ by definition of $\Gamma^\dagger$ in \eqref{def : Gamma and Gamma dagger}.
In particular, the computation of the Fourier coefficients $E_0$ of $\cosh(2a_0x)$ on $\om$ was also derived in Theorem 3.9 in \cite{unbounded_domain_cadiot}.
Now, let us focus on the term $ \|\cha\left( \mathbb{M}_T  {\mathbb{\Lambda}_{\nu}}^{-1} - \Gamma^\dagger\left(M_T\right)  {\mathbb{\Lambda}_{\nu}}^{-1} \right)h_0\|_2.$ Letting $u \bydef \cha$, then
the proof of Theorem 3.9 in \cite{unbounded_domain_cadiot} provides that
\begin{align*}
    &\|\cha\left( \mathbb{M}_T  {\mathbb{\Lambda}_{\nu}}^{-1} - \Gamma^\dagger\left(M_T\right)  {\mathbb{\Lambda}_{\nu}}^{-1} \right)h_0\|_2^2 \\
    \leq &\left(\mathcal{Y}_{u,1}\right)^2 + C_{\mathcal{Y}_0,T}^2 \sum_{ k \in \mathbb{Z}\setminus\{0\}}\int_{\om}\int_{\om} |h_0(x)h_0(z)|\left(\int_{\mathbb{R}\setminus (\om \cup (\om +2dk))} e^{-a_0|y-x|}e^{-a_0|y-2dk-z|}dy\right) dz dx. 
\end{align*}
Moreover, the above can be written as 
\begin{align}\label{eq : proof Yu2 first step}
    &\|\cha\left( \mathbb{M}_T  {\mathbb{\Lambda}_{\nu}}^{-1} - \Gamma^\dagger\left(M_T\right)  {\mathbb{\Lambda}_{\nu}}^{-1} \right)h_0\|_2^2 \\
    \leq &\left(\mathcal{Y}_{u,1}\right)^2 + 2C_{\mathcal{Y}_0,T}^2 \sum_{ k =1}^\infty e^{-2adk}\int_{\om}\int_{\om} |h_0(x)h_0(z)|\left(\int_{\mathbb{R}\setminus (\om \cup (\om +2dk))} \hspace{-1cm} e^{-a_0|y-x|}\cosh(a_0(y-z))dy\right) dz dx. 
\end{align}
Finally, using the proof of Lemma 6.5 in \cite{unbounded_domain_cadiot}, we readily obtain
\begin{align*}
     \|\cha (\mathbb{M}_T {\mathbb{\Lambda}_{\nu}}^{-1}h_0 - \mathbb{M}_{\om}  {\mathbb{\Lambda}_{\nu}}^{-1} h_0)\|_2 \leq 2d C_{\mathcal{Y}_0,T}  e^{-a_0d}\left(~\left( {\Lambda_{\nu}}^2 U_0,E_0*( {\Lambda_{\nu}}^2U_0)\right)_2  \|B_T\|_2^2(1 + C(d))~\right)^{\frac{1}{2}},
\end{align*}
which concludes the proof.
\end{proof}

\subsection{The bound \boldmath\texorpdfstring{$\mathcal{Z}_2$}{Z2}\unboldmath}

The computation of the bound $\mathcal{Z}_2$ is obtained thanks to  Lemma \ref{lem : banach algebra} (under which products in $\mathcal{H} \times \mathcal{H} \to L^2$ are well-defined). We present its computation in the next lemma
\begin{lemma}\label{lem : bound Z_2}
Let $r>0$ and let $\kappa_T$ satisfying \eqref{def : definition of kappa}. Moreover, let $\mathcal{Z}_2 >0$ be such that 
\begin{equation}
    \mathcal{Z}_2 \geq  2\kappa_T \max\left\{1,~  \max\{\|B^N_T\|_2, \|W_T\|_1 \} + \|\pi_N\mathbb{W}_T \pi^N\|_2\right\},
\end{equation}
then $ \|\mathbb{A}_T\left({D}\mathbb{F}(v) - D\mathbb{F}(u_0)\right)\|_{\mathcal{H}} \leq \mathcal{Z}_2r$ for all $v \in \overline{B_r(u_0)} \subset \mathcal{H}_e.$
\end{lemma}

\begin{proof}
Let $v \in \overline{B_r(u_0)}$, then observe that because $\|u\|_{\mathcal{H}} = \|\mathbb{L}u\|_2$ for all $u \in \mathcal{H}_e$, we have
\begin{align*}
     \|\mathbb{A}_T\left({D}\mathbb{F}(v) - D\mathbb{F}(u_0)\right)\|_{\mathcal{H}}  &= \|\mathbb{L}\mln\mathbb{A}_T\left({D}\mathbb{F}(v) - D\mathbb{F}(u_0)\right)\|_{\mathcal{H},L^2}\\
   & = \|\mathbb{B}_T \mln\left({D}\mathbb{G}(v) - D\mathbb{G}(u_0)\right)\|_{\mathcal{H},L^2},
\end{align*}
where we also used that  ${D}\mathbb{F}(v) - D\mathbb{F}(u_0) = D\mathbb{G}(v) -  D\mathbb{G}(u_0)$. 

Now let  $w \bydef v-u_0 \in \overline{B_r(0)} \subset \mathcal{H}_e$ (in particular $\|w\|_{\mathcal{H}} \leq r$). Then we have
\[
D\mathbb{G}(v) - D\mathbb{G}(u_0) = 2  (\mathbb{v}-\mathbb{u}_0) = 2   \mathbb{w}.
\]
Moreover, the above yields
\[
\|\mathbb{B}_T\mln\left({D}\mathbb{G}(v) - D\mathbb{G}(u_0)\right)\|_{l,2} \leq  2 \|\mathbb{B}_T\|_{2}\|  {\mathbb{\Lambda}_{\nu}}\mathbb{w}\|_{l,2}.
\]
Now, given $u \in \mathcal{H}_e$, we have 
\begin{align*}
    \|  {\mathbb{\Lambda}_{\nu}}\mathbb{w}u\|_{2} = \| {\mathbb{\Lambda}_{\nu}}(wu)\|_2 \leq \kappa_T \|w\|_{\mathcal{H}} \|u\|_{\mathcal{H}}
\end{align*}
using Lemma \ref{lem : banach algebra}. We conclude the proof using \eqref{eq : equality B operator and B fourier}.
\end{proof}

\subsection{The bound \boldmath\texorpdfstring{$\mathcal{Z}_1$}{Z1}\unboldmath}\label{sec : bound Z1}

The bound $\mathcal{Z}_1$ is essential in our computer-assisted approach as it controls the accuracy of the approximate inverse $\mathbb{A}_T$. Indeed, $\mathcal{Z}_1$ satisfies 
\begin{align*}
    \|I_d - \mathbb{A}_T D\mathbb{F}(u_0)\|_{\mathcal{H}} \leq \mathcal{Z}_1.
\end{align*}
In particular, as demonstrated in the following lemma, having $\mathcal{Z}_1 <1$ implies the invertibility of $\mathbb{A}_T$ (which is required in Theorem \ref{th: radii polynomial}). We recall some of the results from \cite{unbounded_domain_cadiot} (see Section 3) and provide an explicit computation for $\mathcal{Z}_1$ in the case of the  {cgWE }.
\begin{lemma}\label{lem : computation of Z1}
Let $Z_1$ and $\mathcal{Z}_{u,i}^k$ $(i \in \{1, 2\}$, $k \in \{0, 1, 2\})$ be satisfying the following
\begin{align}
Z_1  \geq  \|I_d - &A_TDF(U_0)\|_{\mathcal{H}} \label{def : Z1 periodic}\\ \nonumber
\mathcal{Z}_{u,1}^{(0)} \bydef    2\nu\left\|\mathbb{1}_{\mathbb{R}\setminus \om}\mathbb{L}^{-1} \mathbb{u}_0''\right\|_2  ~~ &\text{ and } ~~   \mathcal{Z}_{u,2}^{(0)} \bydef 2\nu\left\|\cha\left(\mathbb{L}^{-1}-\Gamma^\dagger\left(L^{-1}\right)\right) \mathbb{u}_0''\right\|_2\\  \nonumber
\mathcal{Z}_{u,1}^{(1)} \bydef   4\nu\left\|\mathbb{1}_{\mathbb{R}\setminus \om} \partial_x \mathbb{L}^{-1}\mathbb{u}_0'\right\|_2  ~~ &\text{ and } ~~  \mathcal{Z}_{u,2}^{(1)} \bydef 4\nu\left\| \cha\left(\partial_x \mathbb{L}^{-1}-\Gamma^\dagger\left(\partial_x L^{-1}\right)\right)\mathbb{u}_0'\right\|_2 \label{def : Zu different components} \\ 
\mathcal{Z}_{u,1}^{(2)} \bydef 2\left\|\mathbb{1}_{\mathbb{R}\setminus \om} \mathbb{L}^{-1} \mathbb{u}_0\right\|_2 ~~ &\text{ and } ~~ \mathcal{Z}_{u,2}^{(2)} \bydef 2\left\| \cha \left(\mathbb{L}^{-1}-\Gamma^\dagger\left(L^{-1}\right)\right) \mathbb{u}_0\right\|_2
    \end{align}
Moreover let $\mathcal{Z}_1$ and $\mathcal{Z}_u$ be bounds satisfying 
\begin{align}\label{def : Zu bound}
    \mathcal{Z}_u \geq \|\mathbb{B}_T\|_2\sum_{k=0}^2 \sqrt{(\mathcal{Z}^k_{u,1})^2 + (\mathcal{Z}^k_{u,2})^2} ~~ \text{ and } ~~
    \mathcal{Z}_1 \geq Z_1 + \mathcal{Z}_u,
\end{align}
then $ \|I_d - {\mathbb{A}_T}D\mathbb{F}(u_0)\|_{\mathcal{H}} \leq \mathcal{Z}_1.$ Moreover, if in addition $\mathcal{Z}_1<1$, then both $\mathbb{A}_T : H^2_e \to \mathcal{H}_e$ and $D \mathbb{F}(u_0): \mathcal{H}_e \to H^2_e$ have a bounded  inverse. In particular,
\begin{align}\label{eq : norm on the inverse}
    \|D\mathbb{F}(u_0)^{-1}\|_{H^2,\mathcal{H}} \leq \frac{\|\mathbb{A}_T\|_{H^2,\mathcal{H}}}{1-\mathcal{Z}_1}.
\end{align}
Finally, if $\tilde{u} \in \mathcal{H}_e$ is a solution of \eqref{eq : f(u)=0 on He} obtained thanks to Theorem \eqref{th: radii polynomial}, then $D\mathbb{F}(\tilde{u}) : \mathcal{H}_e \to H^2_e$ has a bounded inverse as well.
\end{lemma}

\begin{proof}
First, notice that 
\begin{align}\label{eq : dev of DG on unbounded}
     \mln D\mathbb{G}(u_0) \mln^{-1} \mathbb{L}^{-1}u =  2 u_0 \mathbb{L}^{-1} u - 4\nu \partial_x u_0 \left(\partial_x\mathbb{L}^{-1} \mln^{-1}u\right) -2\nu \partial_x^2u_0  \mathbb{L}^{-1} \mln^{-1}u
\end{align}
for all $u \in L^2_e.$ Then, the proof of $ \|I_d - {\mathbb{A}_T}D\mathbb{F}(u_0)\|_{\mathcal{H}} \leq \mathcal{Z}_1$ is obtained using a similar proof as the one of Theorem 3.5 in \cite{unbounded_domain_cadiot}. Now, we prove that $\mathbb{A}_T : H^2_e \to \mathcal{H}_e$ and $D \mathbb{F}(u_0): \mathcal{H}_e \to H^2_e$ have a bounded inverse.

First notice that if $\mathcal{Z}_1<1$, then $\mathbb{A}_TD\mathbb{F}(u_0) : \mathcal{H}_e \to \mathcal{H}_e$ has a bounded inverse (using a Neumann series argument on $I_d = I_d - {\mathbb{A}_T}D\mathbb{F}(u_0) + {\mathbb{A}_T}D\mathbb{F}(u_0)$. In particular, this implies that $\mathbb{A}_T$ is surjective. Now, recall that 
\[
\mathbb{A}_T = \mathbb{L}^{-1}\mathbb{B}_T = \mathbb{L}^{-1}\left(\Gamma^\dagger\left(B_T\right) + \mathbb{1}_{\R\setminus \om}\right).
\]
Therefore, using Lemma \ref{lem : gamma and Gamma properties}, $\mathbb{A}_T$ is invertible if an only if $\mathbb{B}_{T} : \ell^2_e \to \ell^2_e$ is invertible. Then, recall that  $B_{T} = B^N_T + \pi_N \mathbb{W}_T : \ell^2_e \to \ell^2_e$, which  is surjective as $\mathbb{A}_T : H^2_e \to \mathcal{H}_e$ is surjective. Therefore $B^N_T : \pi^N \ell^2_e \to \pi^N \ell^2_e$ is surjective, hence invertible as it is finite dimentional. This also implies that $\pi_N \mathbb{W}_T : \ell^2_e \to \pi_N \ell^2_e$ is surjective.  Let $U \in \ell^2_e$ such that $B_{T}U = 0$. Then $\pi^NU = 0$ as $B^N_T$ is invertible and $B_{T}U = \pi_N W_T*(\pi_N U) = \pi_N\mathbb{W}_T\pi_N U = 0.$ Now, $\pi_N \mathbb{W}_T \pi_N : \pi_N \ell^2_e \to \pi_N \ell^2_e$ is surjective and symmetric, therefore it is also injective. This implies that $U=0$.

Consequently $\mathbb{A}_T : H^2_e \to \mathcal{H}_e$ is invertible and thus has a bounded  inverse (as a continuous operator between Hilbert spaces). Since $\mathbb{A}_TD\mathbb{F}(u_0) : \mathcal{H}_e \to \mathcal{H}_e$ has a bounded  inverse, so does $D\mathbb{F}(u_0) : \mathcal{H}_e \to H^2_e.$ The proof of  \eqref{eq : norm on the inverse} is given in Theorem 3.5 in \cite{unbounded_domain_cadiot}. 

Now, to prove that $D\mathbb{F}(\tilde{u}) : \mathcal{H}_e \to H^2_e$ also has a bounded inverse, notice that 
\begin{align*}
    I_d = I_d - \mathbb{A}_TD\mathbb{F}(\tilde{u}) + \mathbb{A}_TD\mathbb{F}(\tilde{u}).
\end{align*}
But,
\begin{align*}
    \|I_d - \mathbb{A}_TD\mathbb{F}(\tilde{u})\|_{\mathcal{H}} \leq  \|\mathbb{A}_T\left(D\mathbb{F}(\tilde{u})-D\mathbb{F}(u_0)\right)\|_{\mathcal{H}} + \|I_d - \mathbb{A}_TD\mathbb{F}(u_0)\|_{\mathcal{H}} \leq \mathcal{Z}_{2} r + \mathcal{Z}_1 
\end{align*}
using \eqref{eq: definition Y0 Z1 Z2}. Therefore, using \eqref{condition radii polynomial}, we get
\begin{align*}
     \mathcal{Z}_{2} r + \mathcal{Z}_1 < 1 - \frac{\mathcal{Y}_0}{r} < 1.
\end{align*}
Using a Neumann series argument and the fact that $\mathbb{A}_T : H^2_e \to \mathcal{H}_e$ has a bounded inverse, we get that $D\mathbb{F}(\tilde{u}): \mathcal{H}_e \to H^2_e$ has a bounded inverse as well.
\end{proof}


Using the previous lemma, if $\mathcal{Z}_1<1$ then we ensure that the operator $\mathbb{A}_T$ is invertible, hence injective. Consequently, Theorem \ref{th: radii polynomial} can be applied using $\mathbb{A}_T$ if we manage to prove that $\mathcal{Z}_1<1$.
In order to obtain an explicit upper bound for $\mathcal{Z}_1$, we need to compute an upper bound for $Z_1$ and $\mathcal{Z}_u$ satisfying \eqref{def : Z1 periodic} and \eqref{def : Zu bound} respectively.
$\mathcal{Z}_{u}$  comes from the unboundedness part of the problem. We will see in Lemma \ref{lem : lemma Zu} that $\mathcal{Z}_{u}$ is exponentially decaying with the size of $\om$, given by $d.$
 $Z_1$ is the usual term one has to compute during the proof of a periodic solution using the Radii-Polynomial approach (see \cite{van2021spontaneous} for instance). The next Lemma \ref{lem : usual term periodic Z1} provides the details for such an analysis. In particular, it is fully determined by vector and matrix norm computations.

\subsubsection{Computation of \boldmath\texorpdfstring{$Z_1$}{Z1}\unboldmath}\label{ssec : first component Z1}

The bound $Z_1$ controls the accuracy of $A_T$ as an approximate inverse of $DF(U_0)$. Therefore, this bound depends on the quality of the numerics as well as the decay in the Fourier coefficients. We present the computation of $Z_1$ in the next lemma.

\begin{lemma}\label{lem : usual term periodic Z1}
 {Let $B^N_T, W_T$ be defined in Section \ref{ssec : approximate inverse} and  let the} bounds $Z_{1,i}$ ($i \in \{1,2,3,4\})$  satisfy
 {
\begin{align}\label{def : values of Z1i}
\nonumber
    Z_{1,1} &\geq  \bigg(\|\pi^N - B_T^N \Lambda_\nu DF(U_0)\Lambda_\nu^{-1}L^{-1}\pi^{N}\|_2^2 + \|(\pi^{3N}-\pi^N)\mathbb{W}_T \Lambda_\nu DF(U_0)\Lambda_\nu^{-1} L^{-1}\pi^{N}\|_2^2\bigg)^{\frac{1}{2}}\\ \nonumber
    Z_{1,2} &\geq \|\pi^NB_T^N \Lambda_\nu DG(U_0)\Lambda_\nu^{-1} L^{-1}(\pi^{2N}-\pi^N)\|_2\\ \nonumber
    Z_{1,3} &\geq \frac{2\nu}{l_{T,0}}\|W_T*(\partial_x^2 U_0)\|_1 + \frac{4\nu}{l_{T,1}}\|W_T*(\partial_xU_0)\|_{\ell^1} + \frac{2}{l_{T,2}}\|W_T*U_0\|_1\\ 
    Z_{1,4} &\geq \begin{cases}
        0 &\text{ if } T>0\\
        \|e_0 - W_0*(e_0 - \frac{1}{c}V_2)\|_1 &\text{ if } T=0.
    \end{cases}
\end{align}
}
where  
\begin{align*}
  l_{T,0} &\bydef \displaystyle\min_{|n| > N} |l(\tilde{n})|, ~~
  l_{T,1} \bydef \displaystyle\min_{|n| > N} \frac{|l(\tilde{n})|}{2\pi|\tilde{n}|}\\
 l_{T,2} &\bydef \begin{cases}
        \displaystyle\min_{|n| > N} |m_T(\tilde{n})-c| &\text{ if } T>0\\
        \displaystyle\min_{|n| > N} \frac{c|m_0(\tilde{n})-c|}{m_0(\tilde{n})} &\text{ if } T=0.
    \end{cases}
\end{align*}
Then defining $Z_1 \bydef \left( Z_{1,1}^2 + Z_{1,2}^2 + (Z_{1,3} + Z_{1,4})^2   \right)^{\frac{1}{2}}$, we have $Z_1 \geq \|I_d - A_TDF(U_0)\|_{\mathcal{H}}.$
\end{lemma}

\begin{proof}
First, notice that 
\begin{align}\label{eq : Z1 periodic first step}
\nonumber
    \|I_d - A_TDF(U_0)\|_{\mathcal{H}}^2 &= \|I_d - B_T \Lambda_\nu DF(U_0)\Lambda_\nu^{-1} L^{-1}\|_2^2\\
    \leq  \|\pi^N - &B_T\Lambda_\nu DF(U_0)\Lambda_\nu^{-1} L^{-1}\pi^N\|_2^2 + \|\pi_N - B_T\Lambda_\nu DF(U_0)\Lambda_\nu^{-1} L^{-1}\pi_N\|_2^2. 
\end{align}
Then, we have $DF(U_0)\pi^N = \pi^{2N}DF(U_0)\pi^N$ as $U_0 = \pi^N U_0$. Moreover, $B_T\pi^{2N} = \pi^{3N}B_T\pi^{2N}$ as $W_T = \pi^N W_T$ by construction. Therefore, we get $B_TDF(U_0)\pi^N = \pi^{3N}B_T DF(U_0)\pi^N$ and  $$ \|\pi^N - B_T\Lambda_\nu DF(U_0)\Lambda_\nu^{-1} L^{-1}\pi^N\|_2 =  \|\pi^N - \pi^{3N}\Lambda_\nu B_TDF(U_0)\Lambda_\nu^{-1} L^{-1}\pi^{N}\|_2.$$ 
Moreover, we have 
\begin{align*}
    &\|\pi^N - \pi^{3N}B_T\Lambda_\nu DF(U_0)\Lambda_\nu^{-1} L^{-1}\pi^{N}\|_2^2 \\
    &\leq 
    \|\pi^N - B_T^N\Lambda_\nu DF(U_0)\Lambda_\nu^{-1} L^{-1}\pi^{N}\|_2^2 + \|(\pi^{3N}-\pi^N)\mathbb{W}_T\Lambda_\nu DF(U_0)\Lambda_\nu^{-1} L^{-1}\pi^{N}\|_2^2 \leq Z_{1,1}^2
\end{align*}
using that $B_T = B_T^N + \pi_N \mathbb{W}_T.$ 
Let us now consider the term $\|\pi_N - B_TDF(U_0)L^{-1}\pi_N\|_2$ in \eqref{eq : Z1 periodic first step}. We have
\begin{align*}
     &\|\pi_N - B_T\Lambda_\nu DF(U_0)\Lambda_\nu^{-1}L^{-1}\pi_N\|_2^2 \\
     &\leq  \|\pi^NB_T^N(I_d + \Lambda_\nu DG(U_0)\Lambda_\nu ^{-1}L^{-1})\pi_N\|_2^2 + \|\pi_N - \pi_N\mathbb{W}_T(I_d + DG(U_0)L^{-1})\pi_N\|_2^2\\
    &= \|\pi^NB_T^NDG(U_0)L^{-1}\pi_N\|_2^2 + \|\pi_N - \pi_N\mathbb{W}_T(I_d + \Lambda_\nu DG(U_0)\Lambda_\nu^{-1} L^{-1})\pi_N\|_2^2\\
    &\leq Z_{1,2}^2 + \|\pi_N - \pi_N\mathbb{W}_T(I_d + \Lambda_\nu DG(U_0)\Lambda_\nu^{-1} L^{-1})\pi_N\|_2^2.
\end{align*}
Then, using \eqref{eq : dev of DG on unbounded}, we have 
\begin{align}\label{eq : development of DG(U0)}
   \Lambda_\nu DG(U_0)\Lambda_\nu^{-1} L^{-1}U = 2U_0*(L^{-1}U) - 4\nu(\partial_xU_0)*(\partial_x L^{-1}\Lambda_\nu^{-1} U) - 2(\partial_x^2U_0)*(L^{-1}\Lambda_\nu^{-1}U)
\end{align}
for all $U \in \ell^2_e.$ 
Suppose first that $T>0$, then
\begin{align*}
    \|\pi_N - \pi_N\mathbb{W}_T(I_d + \Lambda_\nu DG(U_0) \Lambda_\nu^{-1} L^{-1})\pi_N\|_2 = \|\pi_N\mathbb{W}_T\Lambda_\nu DG(U_0)\Lambda_\nu^{-1} L^{-1}\pi_N\|_2
\end{align*}
as $\mathbb{W}_T = I_d$ if $T>0$ (cf. \eqref{def : definition of WT}). 
Then, combining \eqref{eq : youngs inequality} with \eqref{eq : development of DG(U0)}, we get
\begin{align*}
    \|\pi_N \mathbb{W}_T DG(U_0)L^{-1}\pi_N\|_2  \leq  \frac{2\nu}{l_{T,0}}\|W_T*(\partial_x^2 U_0)\|_1 + \frac{4\nu}{l_{T,1}}\|W_T*(\partial_xU_0)\|_{\ell^1} + \frac{2}{l_{T,2}}\|W_T*U_0\|_1 \leq Z_{1,3}.
\end{align*}

Let us now focus on the case $T=0$.  Using that $DG(U_0) = 2 \mathbb{U}_0$, we have 
\begin{align*}
    &\|\pi_N - \pi_N \mathbb{W}_0(I_d + \Lambda_\nu DG(U_0)\Lambda_\nu^{-1}L^{-1})\pi_N\|_2 \\
    &\leq \|\pi_N - \pi_N\mathbb{W}_0(I_d - \Lambda_\nu\frac{2}{c}\mathbb{U}_0)\pi_N\|_2 + 2\|\pi_N \mathbb{W}_0\big( {\Lambda_{\nu}} \mathbb{U}_0 \Lambda_\nu^{-1}L^{-1} + \frac{1}{c}\mathbb{U}_0   \big)\pi_N\|_2.
\end{align*}
Now, using \eqref{eq : youngs inequality}, we have
\[
\|\pi_N - \pi_N\mathbb{W}_0(I_d - \frac{2}{c}\mathbb{U}_0)\pi_N\|_2 \leq \|I_d - \mathbb{W}_0(I_d - \frac{2}{c}\mathbb{U}_0)\|_2 \leq \|e_0 - W_0*(e_0 - \frac{2}{c}U_0)\|_1 \leq Z_{1,4}.
\]
Let $\widetilde{M} \bydef (M_0-cI_d)^{-1} + \frac{1}{c}I_d$. In particular, $\widetilde{M}$ is an infinite diagonal matrix with entries $\left(\frac{1}{m_0(\tilde{n})-c} + \frac{1}{c}\right)_n = \left(\frac{m_0(\tilde{n})}{c(m_0(\tilde{n})-c)}\right)_n$ on the diagonal. Then, it follows that 
\[
\|\pi_N\tilde{M}\|_2 \leq \frac{1}{l_{0,2}}.
\]
Moreover, using \eqref{eq : development of DG(U0)} we have
\[
2\left( {\Lambda_{\nu}} \mathbb{U}_0L^{-1} + \frac{1}{c}\mathbb{U}_0\right)U  = 2{U}_0*(\tilde{M}U) - 4 \nu  (\partial_x U_0)*(\partial_x L^{-1}\Lambda_\nu^{-1}U)  -2\nu (\partial_x^2U_0)*(L^{-1}\Lambda_\nu^{-1}U).
\]
Then, similarly as what was achieved in the case $T>0$, it implies that
\[
2\|\pi_N \mathbb{W}_0\big(  {\Lambda_{\nu}} \mathbb{U}_0L^{-1} + \frac{1}{c}\mathbb{U}_0   \big)\pi_N\|_2 \leq \frac{2\nu}{l_{0,0}}\|W_T*(\partial_x^2 U_0)\|_1 + \frac{4\nu}{l_{0,1}}\|W_T*(\partial_xU_0)\|_{\ell^1} + \frac{2}{l_{0,2}}\|W_T*U_0\|_1 \leq Z_{1,3}
\] 
\end{proof}

\begin{remark}
    Notice that in the case $T=0$, the term $Z_{1,4}$ in \eqref{def : values of Z1i} controls that $\mathbb{W}_0$ is a good approximate inverse for $ I_d - \frac{2}{c}\mathbb{U}_0.$ In particular, a Neumann series argument shows that both $I_d - \frac{2}{c}\mathbb{U}_0$ and  $\mathbb{W}_0$ have a bounded inverse (from $\ell^2_e$ to $\ell^2_e$)  if $Z_{1,4} <  1$. 
\end{remark}

\subsubsection{Computation of \boldmath\texorpdfstring{$\mathcal{Z}_{u}$}{Zu}\unboldmath}\label{ssec : computation Z11}

Similarly as for the bound $\mathcal{Y}_u$ presented in Lemma \ref{lem : bound Y_0}, the bound $\mathcal{Z}_u$ controls the approximation of convolution operators with exponential decay by their periodic counterparts. In particular, the computation of $\mathcal{Z}_u$ is based on the constants obtained in Lemma \ref{lem : computation of f}.

\begin{lemma}\label{lem : lemma Zu}
Let $C_{0,T}, C_{1,T}, C_{2,T,c}$ and $a>0$ be defined in Lemma \ref{lem : computation of f}. Moreover, let $C(d)$ be defined in \eqref{def : value for C(d)} and let $C_1(d)$ be defined as 
\begin{align}
    C_1(d) \bydef\left(\frac{2\sqrt{\pi}e^{-2ad}}{\sqrt{4ad}(1-e^{-2ad})} + \frac{4e^{-2ad}}{1-e^{-2ad}} \right).
\end{align}
Finally, define $E = (E_n)_{n \in \mathbb{Z}} \in \ell^2_e$  as 
\begin{align}\label{def : definition of sequence E}
    E_{n} \bydef \frac{e^{2ad}}{d} \frac{a(-1)^{n}(1-e^{-4ad})}{4a^2 + (2\pi\tilde{n})^2} 
\end{align}
for all $n \in \mathbb{Z}$.
In particular, $\gamma^\dagger(E)(x)  = \cha(x) \cosh(2ax)$ for all $x \in \R$. 
Then, letting $\mathcal{Z}_{u,j}^{(i)}$ be defined in Lemma \ref{lem : computation of Z1}, we have
\begin{align}
 \hspace{-0.5cm}\left(\mathcal{Z}_{u,1}^{(0)}\right)^2 \hspace{-0.1cm} + \hspace{-0.1cm} \left(\mathcal{Z}_{u,2}^{(0)}\right)^2 &\leq 4\nu^2|\om|C_{0,T}^2e^{-2ad}\left(\partial_x^2U_0,E*(\partial_x^2U_0)\right)_{2} \left(\frac{2}{a} + C(d)\right)  \label{eq : upper bound for Zu first}  \\
\hspace{-0.5cm}\left(\mathcal{Z}_{u,1}^{(1)}\right)^2 \hspace{-0.1cm} + \hspace{-0.1cm} \left(\mathcal{Z}_{u,2}^{(1)}\right)^2 &\leq 16\nu^2|\om|C_{1,T}^2e^{-2ad}\left(\partial_xU_0,E*(\partial_xU_0)\right)_{2} \left(\frac{2}{a} + C(d)\right)  \label{eq : upper bound for Zu second}\\
 \hspace{-0.5cm}\left(\mathcal{Z}_{u,1}^{(2)}\right)^2 \hspace{-0.1cm} + \hspace{-0.1cm} \left(\mathcal{Z}_{u,2}^{(2)}\right)^2 &\leq  {4|\om| C_{2,T,c}^2e^{-2ad}}\left(U_0,E*U_0\right)_2\left( \frac{2}{a} + C_1(d)\right) + 32C_{2,T,c}^2\ln(2)\int_{d-1}^d|u_0'|^2. \label{eq : upper bound for Zu third}
\end{align}
\end{lemma}

\begin{proof}
First, since $|u_0''|$ and $|u_0'|$ are both even function, notice that the proof of \eqref{eq : upper bound for Zu first} and \eqref{eq : upper bound for Zu second} can be found in \cite{unbounded_domain_cadiot} (cf. proof of Lemma 6.5). 
Therefore it remains to treat $\mathcal{Z}_{u,i}^{(2)}$ ($i \in \{1,2\}$).

Let us first suppose that $T>0$. Then, let $u \in L^2_e$ such that $\|u\|_2 = 1$ and let us denote $v\bydef 2u_0 u.$ By construction, $v \in L^2_e$ and supp$(v) \subset \overline{\om}$. We want to estimate
$
\|\mathbb{1}_{\mathbb{R}\setminus \om} \mathbb{L}^{-1} v\|_2.
$\\
Let us first suppose that $T>0$. Then, using Lemma \ref{lem : computation of f} we obtain
\begin{align*}
    \|\mathbb{1}_{\mathbb{R} \setminus \Omega_0}\mathbb{L}^{-1} v\|_2^2
   & = \int_{\mathbb{R} \setminus \om} \left(\int_\om f_{2,T}(y-x)v(x) dx \right)^2dy\\
   & \leq C_{2,T,c}^2 \int_{\mathbb{R} \setminus \om} \left(\int_\om \frac{e^{-a|y-x|}}{\sqrt{|y-x|}}|v(x)| dx \right)^2dy.
\end{align*}
Now using Cauchy-Schwarz inequality and $\|u\|_2=1$, we get
\begin{align}\label{eq : cauchy S for E2}
\nonumber
     \int_{\R \setminus \om} \left(\int_\om \frac{e^{-a(y-x)}}{\sqrt{y-x}}|v(x)| dx\right)^2dy
    & =   4\int_{\R \setminus \om} \left(\int_\om\frac{e^{-a(y-x)}}{\sqrt{y-x}}|u_0(x)u(x)| dx\right)^2dy\\ \nonumber
   & \leq  4\int_\om |u(x)|^2 \int_{\R \setminus \om} \int_\om \frac{e^{-2a|y-x|}}{|y-x|} u_0(x)^2dydx\\
   & \leq 4\int_\om u_0(x)^2\int_{\R \setminus \om} \frac{e^{-2a|y-x|}}{|y-x|} dydx.
\end{align}

But now notice that if $x \in \om$ and $|x| \leq d-1$, then
\begin{align}\label{eq : computation exponential decay 2}
     \int_{\R \setminus \om} \frac{e^{-2a|y-x|}}{|y-x|}dy
     = \int_{d}^\infty\frac{e^{-2a(y-x)}}{y-x}dy + \int_{-\infty}^{-d}\frac{e^{2a(y-x)}}{x-y}dy
     \leq  \frac{e^{-2ad}\cosh(2ax)}{a}.
\end{align}

In addition, if $x \in \om$ and $d- 1 < x \leq d$, then
\begin{align}\label{eq : computation exponential decay 3}
 \nonumber
      \int_{\R \setminus \om} \frac{e^{-2a|y-x|}}{|y-x|}dy
     \leq  & \int_{d+1}^\infty e^{-2a(y-x)}dy + \int_{d}^{d+1}\frac{1}{y-x}dy + \int_{-\infty}^{-d} \frac{e^{2a(y-x)}}{2d-1} \\
     \leq &  \frac{e^{-2ad}\cosh(2ax)}{a} + \ln(d+1-x)-\ln(d-x)
\end{align}
since $d \geq 1.$
Similarly, if $x \in \om$ and $-d+1 < x \leq -d$, then
\begin{align}\label{eq : computation exponential decay 4}
      \int_{\R \setminus \om} \frac{e^{-2a|y-x|}}{|y-x|}dy
     \leq &  \frac{e^{-2ad}\cosh(2ax)}{a} + \ln(d+1-|x|)-\ln(d-|x|).
\end{align}
Therefore, combining \eqref{eq : cauchy S for E2}, \eqref{eq : computation exponential decay 2}, \eqref{eq : computation exponential decay 3} and \eqref{eq : computation exponential decay 4}, we get
\begin{align*}
   \int_{\R \setminus \om} \left(\int_\om\frac{e^{-a(y-x)}}{\sqrt{y-x}}|v(x)| dx\right)^2dy &\leq  \frac{4e^{-ad}}{a}\int_{-d}^{d} u_0(x)^2\cosh(2ax)dx\\
    & ~~ + 4\int_{d-1\leq |x| \leq d} u_0(x)^2\left(\ln(d+1-|x|)-\ln(d-|x|)\right)dx.
\end{align*}
Then, notice that
\begin{align*}
   \int_{d-1}^{d} u_0(x)^2\left( \ln(d+1-x)-\ln(d-x)\right)dx
   &\leq 2\ln(2)\sup_{x \in (d-1,d)}|u_0(x)^2|.
\end{align*}
Let $x \in (d-1,d)$, then since $u_0(d) =0$ as supp$(u_0) \subset \om$ and $u_0$ is smooth, we have
\begin{align*}
    |u_0(x)| \leq \int_{x}^d|u_0'(t)|dt \leq \sqrt{d-x}\left(\int_{d-1}^d|u_0'(t)|^2dt\right)^{\frac{1}{2}} \leq \left(\int_{d-1}^d|u_0'(t)|^2dt\right)^{\frac{1}{2}}.
\end{align*}
Therefore, using that $u_0$ is even,  we get
\begin{align*}
    \int_{d-1 \leq |x| \leq d} u_0(x)^2\left(\ln(d+1-x)-\ln(d-x)\right)dx 
    \leq 4\ln(2)\int_{d-1}^d|u_0'(t)|^2dt.
\end{align*}
Now, using Parseval's identity, we have
\begin{align*}
   \int_{\om} u_0(x)^2\cosh(2ax)dx &= (u_0,u_0\cosh(2ax))_2 = |\om|(U_0,E*U_0)_2.
\end{align*}
This implies that 
\begin{align}\label{eq : Zu2 in lemma}
     \left(\mathcal{Z}_{u,1}^{(2)}\right)^2  &\leq 4C_{2,T,c}^2\left(\frac{|\om|e^{-2ad}}{a}(U_0,E*U_0)_2 + 4\ln(2)\int_{d-1}^d|u_0'|^2\right).
\end{align}
 Let us now focus on $\mathcal{Z}_{u,2}^{(2)}$. Recall that $u \in L^2_e$ such that $\|u\|_2 = 1$ and $v \bydef  v_2  u$, then using the proof of Theorem 3.9 in \cite{unbounded_domain_cadiot} and the parity of $v$, we have
\begin{align}\label{eq : Z1_dirac_comb}
\nonumber
   &~~~~\left\|\cha  \left(\mathbb{L}^{-1}-\Gamma\left((M_T-cI_d)^{-1}\right)\right) v \right\|_2^2 \\ \nonumber
    &= \sum_{n \in \mathbb{Z}} \int_{\mathbb{R}\setminus (\om \cup (\om +2dn))} \mathbb{L}^{-1} v(y) \mathbb{L}^{-1}  v(y-2dn)   dy\\
    & \leq (\mathcal{Z}_{u,1}^{(2)})^2 + 2\sum_{n =1}^\infty \int_{\mathbb{R}\setminus (\om \cup (\om +2dn))} \left|\mathbb{L}^{-1} v(y) \mathbb{L}^{-1}  v(y-2dn)\right|   dy.
\end{align}
Then, using Lemma \ref{lem : computation of f} we get
\begin{align}\label{eq : proof Zu22}
       & ~~~~\|\cha  (\mathbb{L}^{-1}-\Gamma^\dagger\left(L^{-1}\right)) v \|_2^2 \\ \nonumber
       &\leq (\mathcal{Z}_{u,1}^{(2)})^2 + 2C_{2,T,c}^2\sum_{k =1}^{\infty} \int_{\mathbb{R}\setminus \left(\om\bigcup (\om+2dk)\right)} \int_{-d}^d \int_{-d}^d \frac{e^{-a|y-x|}}{\sqrt{|y-x|}}|v(x)| \frac{e^{-a|2kd+z-y|}}{\sqrt{|2kd+z-y|}} |v(z)| dx dz dy.
  \end{align}
   Let $k \in \mathbb{N}$ and let $x, z \in \om$, then denote $$I_k(x,z) \bydef  \displaystyle\int_{\mathbb{R}\setminus \left(\om\bigcup (\om+2dk)\right)}\frac{e^{-a|y-x|}}{\sqrt{|y-x|}}\frac{e^{-a|2kd+z-y|}}{\sqrt{|2kd+z-y|}}dy.$$ By definition, we have
  \begin{align*}
      I_k(x,z) 
      &= I_{k,1}(x,z) + I_{k,2}(x,z) + I_{k,3}(x,z)
  \end{align*}
  where
  \begin{align*}
      I_{k,1}(x,z) &\bydef \int_{-\infty}^{-d}\frac{e^{-a|y-x|}}{\sqrt{|y-x|}}\frac{e^{-a|2kd+z-y|}}{\sqrt{|2kd+z-y|}}dy \\
      I_{k,2}(x,z) &\bydef \int_{d}^{(2k-1)d}\frac{e^{-a|y-x|}}{\sqrt{|y-x|}}\frac{e^{-a|2kd+z-y|}}{\sqrt{|2kd+z-y|}}dy\\
      I_{k,3}(x,z) &\bydef \int_{(2k+1)d}^{\infty}\frac{e^{-a|y-x|}}{\sqrt{|y-x|}}\frac{e^{-a|2kd+z-y|}}{\sqrt{|2kd+z-y|}}dy.
  \end{align*}
  Now, an upper bound for each term $I_{k,i}$ can easily be computed. In particular, straightfoward computations lead to
  \begin{align}\label{computation of I1}
  \nonumber
      I_{k,1}(x,z)  =  \int_{-\infty}^{-d}\frac{e^{-a(x-y)}}{\sqrt{x-y}}\frac{e^{-a(2kd+z-y)}}{\sqrt{2kd+z-y}}dy 
      &\leq \frac{e^{-a(2kd +x+z)}}{\sqrt{2kd}}\int_{-\infty}^{-d}\frac{e^{2ay}}{\sqrt{-d-y}}dy\\
      & \leq \frac{\sqrt{\pi}e^{-a(2(k+1)d +x+z)}}{\sqrt{4ad}}
  \end{align}
  as $k\geq 1$ and $x, z \in (-d,d).$ Similarly,
   \begin{align}\label{computation of I3}
   \nonumber
      I_{k,3}(x,z) & =  \int_{(2k+1)d}^{\infty}\frac{e^{-a|x-y|}}{\sqrt{|x-y|}}\frac{e^{-a|2kd+z-y|}}{\sqrt{|2kd+z-y|}}dy  \\\nonumber
      & =  \int_{-\infty}^{-d}\frac{e^{-a|x+y-2kd|}}{\sqrt{|x+y-2kd|}}\frac{e^{-a|z+y|}}{\sqrt{|z+y|}}dy \\
      &\leq \frac{\sqrt{\pi}e^{-a(2(k+1)d -x-z)}}{\sqrt{4ad}}
  \end{align}
  using the change of variable $y \to 2kd-y$ and using \eqref{computation of I1}. Finally, notice that $I_{k,2}=0$ if $k=1$. If $k >1$, then
  \begin{align}\label{computation of I2}
  \nonumber
       I_{k,2}(x,z) & =  \int_{d}^{(2k-1)d}\frac{e^{-a|x-y|}}{\sqrt{|x-y|}}\frac{e^{-a|2kd+z-y|}}{\sqrt{|2kd+z-y|}}dy  \\\nonumber
      & =  \int_{d}^{(2k-1)d}\frac{e^{-a(y-x)}}{\sqrt{y-x}}\frac{e^{-a(2kd+z-y)}}{\sqrt{2kd+z-y}}dy \\\nonumber
      & =  e^{-a(2kd+z-x)}\int_{d}^{(2k-1)d}\frac{1}{\sqrt{y-x}}\frac{1}{\sqrt{2kd+z-y}}dy \\
       =   ~e^{-a(2kd+z-x)}&\left(\int_{d}^{kd}\frac{1}{\sqrt{y-x}}\frac{1}{\sqrt{2kd+z-y}}dy  +  \int_{kd}^{(2k-1)d}\frac{1}{\sqrt{y-x}}\frac{1}{\sqrt{2kd+z-y}}dy \right).
  \end{align}
 Moreover, notice that 
  \begin{align}\label{eq : beta1}
  \nonumber
      \int_{d}^{kd}\frac{1}{\sqrt{y-x}}\frac{1}{\sqrt{2kd+z-y}}dy &\leq \frac{1}{\sqrt{(2k-1)d - kd}}\int_{d}^{kd}\frac{1}{\sqrt{y-d}}dy\\
      &= \frac{2\sqrt{kd-d}}{\sqrt{(2k-1)d - kd}} = 2.
  \end{align}
  Similarly,
    \begin{align}\label{eq : beta2}
      \int_{kd}^{(2k-1)d}\frac{1}{\sqrt{y-x}}\frac{1}{\sqrt{2kd+z-y}}dy \leq  \frac{2\sqrt{(2k-1)d - kd}}{\sqrt{kd -d}}=2.
  \end{align}
  Therefore, combining \eqref{computation of I2}, \eqref{eq : beta1} and \eqref{eq : beta2}, we get
  \begin{align}\label{eq : equality I2}
      I_{k,2}(x,y) \leq 4e^{-a(2kd+z-x)}
  \end{align}
  for all $k >1$. Furthermore, combining \eqref{computation of I1}, \eqref{eq : equality I2} and \eqref{computation of I3}, it yields
  \begin{align}\label{eq : computation Ik last step}
      \sum_{k=1}^\infty I_k(x,y) &=  \sum_{k=1}^\infty I_{k,1}(x,y) +  \sum_{k=2}^\infty I_{k,2}(x,y)+  \sum_{k=1}^\infty I_{k,3}(x,y)\\
      & = \frac{\sqrt{\pi}e^{-4ad}}{\sqrt{4ad}(1-e^{-2ad})}e^{-a(x+z)} + \frac{4e^{-4ad}}{1-e^{-2ad}}e^{-a(z-x)} + \frac{\sqrt{\pi}e^{-4ad}}{\sqrt{4ad}(1-e^{-2ad})}e^{a(x+z)}.
  \end{align}
  Consequently, combining \eqref{eq : proof Zu22} and \eqref{eq : computation Ik last step}, we obtain
  \begin{align*}
      \quad &\|\cha  (\mathbb{L}^{-1}-\Gamma^\dagger\left(L^{-1}\right)) v \|_2^2  \\
      &\leq (\mathcal{Z}_{u,1}^{(2)})^2 + 2C_{2,T,c}^2\left(\frac{2\sqrt{\pi}e^{-4ad}}{\sqrt{4ad}(1-e^{-2ad})} + \frac{4e^{-4ad}}{1-e^{-2ad}} \right)\int_{-d}^d\int_{-d}^d |v(x)|e^{ax}|v(z)|e^{az}dxdz\\
      & = (\mathcal{Z}_{u,1}^{(2)})^2 + C_{2,T,c}^2C_1(d)e^{-2ad}\left(\int_{-d}^d |v(x)|e^{ax}dx\right)^2.
  \end{align*}
Now, recall that $v = v_2u$. Then, using Cauchy-Schwarz inequality and Parseval's identity we get
  \begin{align*}
      \|\cha  (\mathbb{L}^{-1}-\Gamma^\dagger\left(L^{-1}\right)) v \|_2^2 
       &\leq (\mathcal{Z}_{u,1}^{(2)})^2 +4 C_{2,T,c}^2{C_1(d)}e^{-2ad}\int_{-d}^d |u_0(x)|^2e^{2ax}dx\\
      & = (\mathcal{Z}_{u,1}^{(2)})^2 + 4|\om|C_{2,T,c}^2C_1(d)e^{-2ad}(U_0,E*U_0)_2.
  \end{align*}
  
This concludes the proof for the case $T>0$. Now, assume that $T=0$. Then, notice that, given $u \in L^2_{e,\om}$, we have
\begin{align*}
    \Gamma^\dagger\left(L^{-1}\right) u &=   \frac{1}{c}u - \left(\frac{1}{c}I_d +\Gamma^\dagger\left(L^{-1}\right)\right)u\\
&= \frac{1}{c}u - \left(\Gamma^\dagger\left(L^{-1} + \frac{1}{c}I_d\right)\right)u
\end{align*}
by definition of $\Gamma$ in \eqref{def : Gamma and Gamma dagger}.
Therefore, since $v_2 \in \mathcal{H}_{e,\om}$, this implies that 
{
\begin{align*}
    \left\|\left(\mathbb{L}^{-1}-\Gamma^\dagger\left(L^{-1}\right)\right) \mathbb{u}_0 \right\|_2 = \left\|\left(\mathbb{L}^{-1} + \frac{1}{c}I_d -\Gamma^\dagger\left(L^{-1} + \frac{1}{c}I_d\right)\right) \mathbb{u}_0 \right\|_2.
\end{align*}}
Moreover, $\left(\mathbb{M}_0-cI_d\right)^{-1}u + \frac{1}{c}u = (f_{2,0} + \frac{1}{c}\delta)*u$ for all $u \in L^2$ by definition of $f_{2,0}$ in \eqref{eq : definition of fiT with operators}. Therefore, using Lemma \ref{lem : computation of f}, the proof in the case $T=0$ can be  derived similarly as the case $T>0$ presented above.
\end{proof}

\begin{remark}
    We derived in Section \ref{sec : computation of the bounds} explicit computations of the required bounds $\mathcal{Y}_0$, $\mathcal{Z}_1$ and $\mathcal{Z}_2$. Notice that the only part that specifically depends on the  {cgWE} itself is the computation of the constants in Lemma \ref{lem : computation of f}. The rest of the analysis can easily be generalized to a large class of nonlocal equations. We further discuss this generalization in the conclusion \ref{conclusion}. 
\end{remark}

\subsection{Proof of existence of solitary waves}

We present four examples of computer-assisted proofs of even  {solitary waves}, one in the case $T = T_1  \bydef 0$ (cf. Theorem \ref{th : proof whitham}), which we denote $\tilde{u}_1 \in H^\infty(\R)$, and three in the case $T>0$ (cf. Theorem \ref{th: : proof capillary whitham} with $T_2 \bydef 0.25$, $T_3 \bydef 0.5$, $T_4 \bydef 3$), denoted $\tilde{u}_2, \tilde{u}_3, \tilde{u}_4 \in H^\infty(\R)$. \\
Using the strategy derived in Section \ref{ssec : approximate solution}, we start by computing an approximate solution $u_{0,i} \in H^4_{e,\om}(\R)$ given by its Fourier series representation $U_{0,i} \in X^4_e$ ($i \in \{1,2,3,4\}$). Then, using the results of Section \ref{ssec : approximate inverse}, we construct an approximate inverse $\mathbb{A}_{T_i} : H^2_e \to \mathcal{H}_e$  for $D{\mathbb{F}}(u_{0,i})$. In the case of the  {WE} $(T=0)$, we verify that Assumption \ref{ass : u0 is smaller than} is satisfied.
Finally, Section \ref{sec : computation of the bounds} allows to compute explicit bounds $\mathcal{Y}_0$, $\mathcal{Z}_1$ and $\mathcal{Z}_2$ using rigorous numerics. In particular, we are able to verify \eqref{condition radii polynomial} in Theorem \ref{th: radii polynomial} and obtain a proof of existence.   {Furthermore, using Theorem 3.17 in \cite{sh_cadiot}, we are able to prove that the solitary wave is the limit of a branch of periodic solutions, letting the period tends to infinity. As mentioned in the introduction, this phenomenon has been deeply studied in the literature and underlines the strong relationship that exists between solitary waves and their periodic counterparts. Note that Theorem 3.17 in \cite{sh_cadiot} provides a constructive proof of the branch.}   The algorithmic details are presented in \cite{julia_cadiot}. 

\subsubsection{Case \texorpdfstring{$T=0$}{T=0}}\label{sec : proof in the case T=0}

We present in this section a (constructive) proof of existence of a solitary wave in the WE , that is for $T =  0$.  {In particular, the obtained wave is part of a branch for which we provide the existence in Section \ref{sec : continuation}.} We fix $c = 1.1$ and, using the strategy established in Section \ref{ssec : approximate solution}, we build an approximate solution $u_{0,1} \in H^4_{e,\om}(\R)$ via its Fourier series $U_{0,1} \in X^4_e$ on $\om = (-50,50)$ such that $U_{0,1} = \pi^N U_{0,1}$ where $N = 800.$  The approximate solution is represented in Figure \ref{fig : whitham} below. 
In particular, choosing $0<\epsilon < \frac{c}{2} - \|U_{0,1}\|_1$, we have
\begin{align*}
    u_0(x) + \epsilon \leq \|U_{0,1}\|_1 + \epsilon < \frac{c}{2}
\end{align*}
for all $x \in \R$. Moreover, using rigorous numerics we can prove that we can choose $\epsilon = 0.39$. This implies that Assumption \ref{ass : u0 is smaller than} is satisfied and the analysis derived in Section \ref{ssec : approximate inverse} and Section \ref{sec : computation of the bounds} is applicable. We apply Theorem \ref{th: radii polynomial} and  obtain the following result.
 \begin{figure}[H]
\centering
 \begin{minipage}[H]{0.9\linewidth}
 
  \centering\epsfig{figure=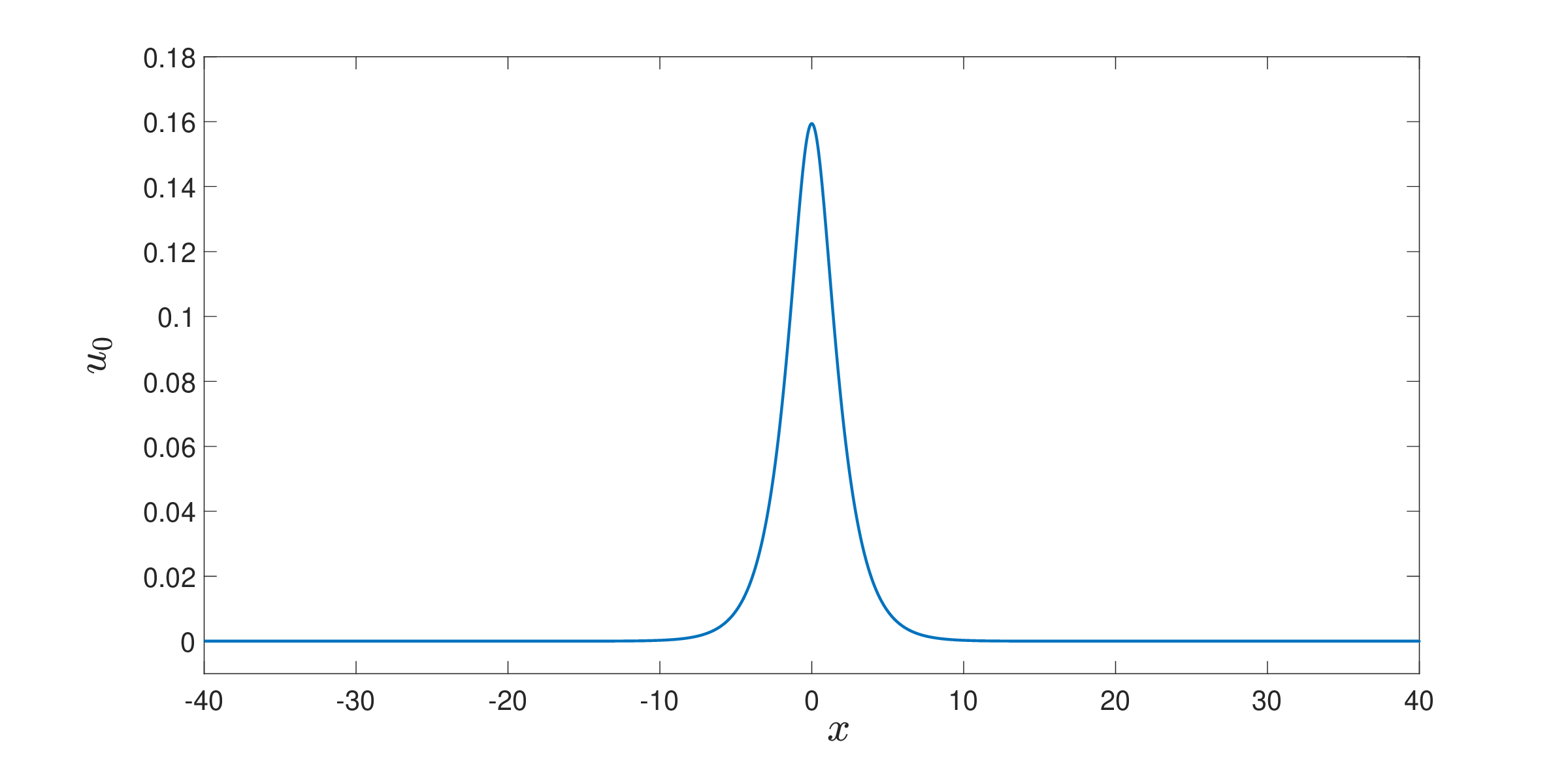,width=
  \linewidth}
  \caption{Numerical approximation $u_{0,1}$ for the Whitham equation.}
  \label{fig : whitham}
 \end{minipage} 
 \end{figure}
\begin{theorem}\label{th : proof whitham}(Proof of a  {solitary wave} in the Whitham equation)\\
Let $r_{0,1} \bydef 5.72\times 10^{-9}$, then there exists a unique even solution $\tilde{u}_1$ to \eqref{eq : whitham stationary} in $\overline{B_{r_{0,1}}(u_{0,1})} \subset \mathcal{H}_e$ for $T=0$ and $c = 1.1$.  { In addition, there exists a smooth curve 
\[
\left\{\tilde{u}_1(q) : q \in [d,\infty]\right\} \subset C^\infty(\R)
\]
such that $\tilde{u}_1(q)$ is a periodic solution to \eqref{eq : whitham stationary} with period $2q$. In particular, $\tilde{u}_1(\infty) = \tilde{u}_1$.}
\end{theorem}
\begin{proof}
    The proof is a direct application of Theorem \ref{th: radii polynomial}. In particular, we obtain that $\mathcal{Y}_0 \bydef 5.24 \times 10^{-9}$, $\mathcal{Z}_1 \bydef 0.078$ and $\mathcal{Z}_2 \bydef 1990$ satisfy \eqref{eq: definition Y0 Z1 Z2}. Then, one can prove that $r_{0,1}$ satisfies \eqref{condition radii polynomial}, leading to the proof of existence of $\tilde{u}_1$.
    Moreover, using that $\epsilon = 0.39$ and that $\min_{\xi \in \R} |m_0(\xi)-c| = c-1 = 0.1$, we prove that $\epsilon$ and $r_{0,1}$ satisfy \eqref{eq : condition epsilon and r}.  This provides the regularity of $\tilde{u}_1$ (cf. Section \ref{ssec : regularity of the solution}).  {The branch of periodic solutions is obtained thanks to Theorem 3.17 in \cite{sh_cadiot}.}
\end{proof}

\subsubsection{Case \texorpdfstring{$T>0$}{T>0}}

Similarly as the previous section, we fix $c = 0.8$, a value for $T>0$, and we construct an approximate solution $u_{0}$ using Section \ref{ssec : approximate solution}.  {Specifically, we choose $T_2 = 0.25$, $T_3 = 0.5$ and $T_4 = 3$. The value $T = \frac{1}{3}$ is known to be a critical value in the dynamics of the cgWE, leading, for instance, to the existence of so-called ``generalized solitary waves" (cf. \cite{bif_diagram_ehrnstrom, johnson_generalized,  REMONATO201751}). In particular, the regime $0 < T < \frac{1}{3}$ does not allow to readily use a KdV approximation for the supcritical solitary waves (as detailed in \cite{johnson_generalized}). Using the analysis derived in Section \ref{sec : computation of the bounds}, we provide in  Theorem \ref{th: : proof capillary whitham}  existence proofs
 on both sides of the critical value ($0.25< \frac{1}{3} < 0.5)$.  Furthermore, we obtain a proof for a large Bond number ($T=3$), underlying a strong dispersive effect of the surface tension.} The approximate solutions are represented in Figure \ref{fig : capillary} below.  
 \begin{figure}[h!]
  \centering
  \begin{minipage}{.52\textwidth}
   \centering
  \includegraphics[clip,width=1\textwidth]{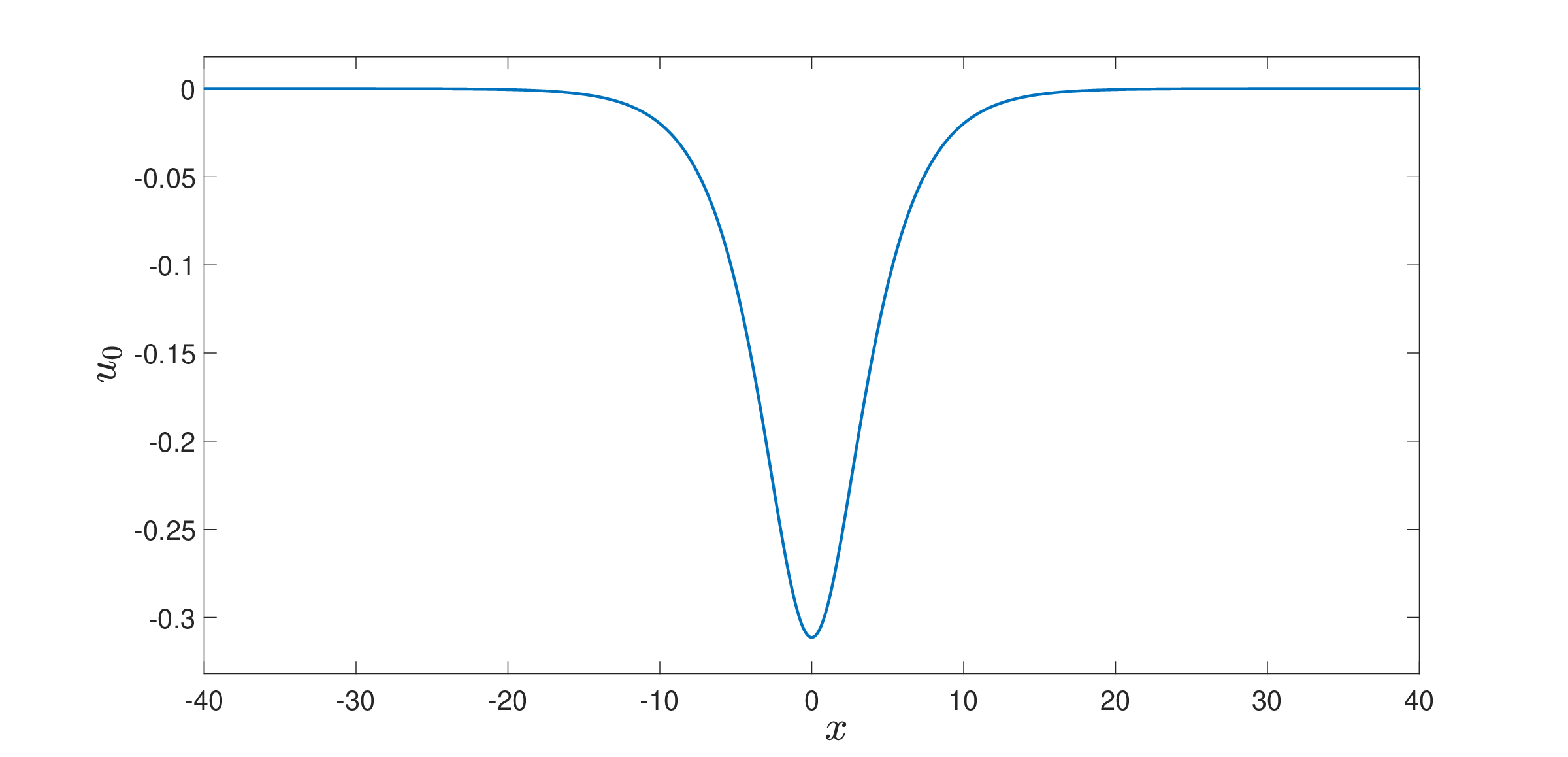}
  \end{minipage}%
  \begin{minipage}{.52\textwidth}
    \centering
   \includegraphics[clip,width=1\textwidth]{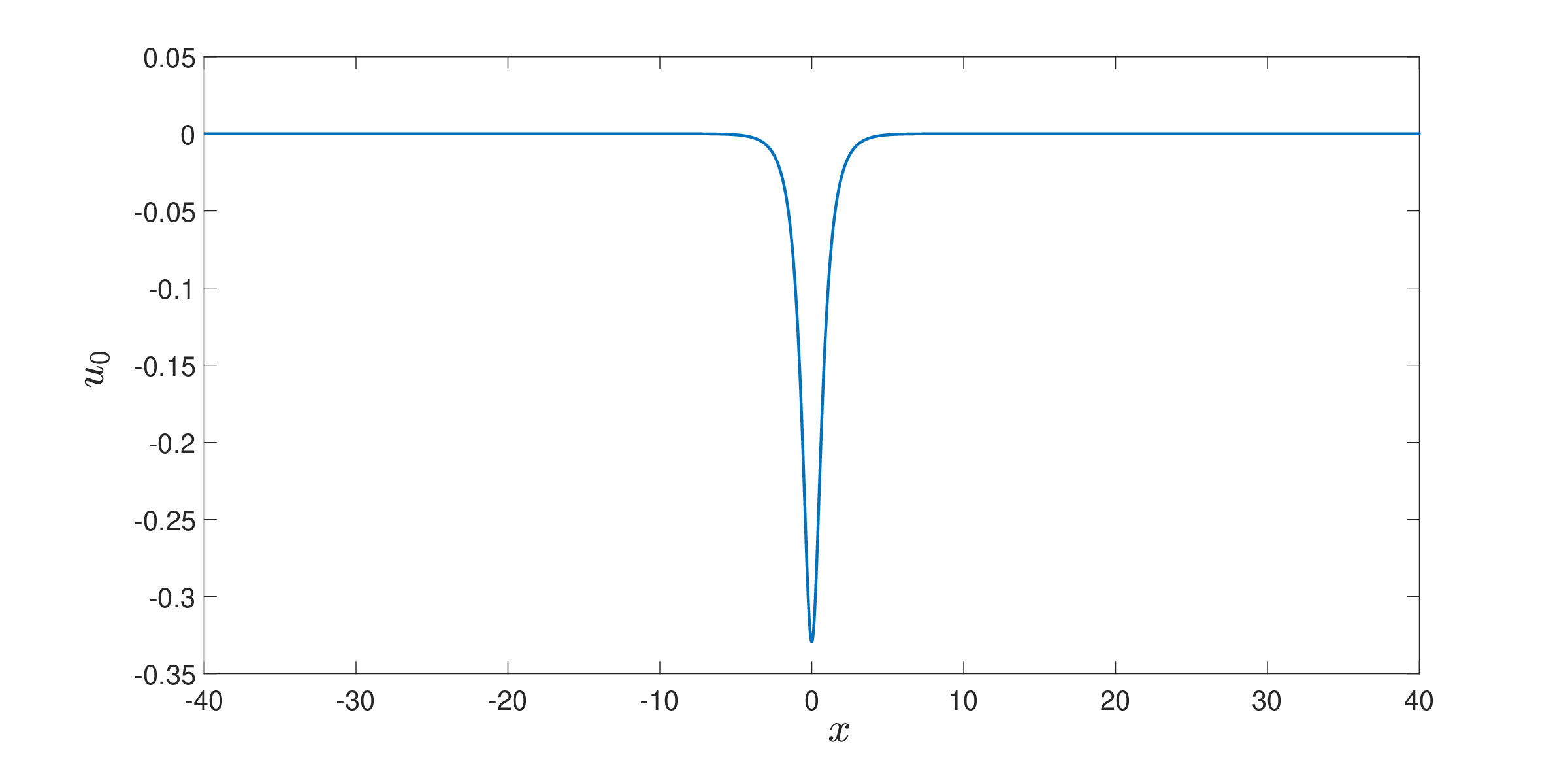}
  \end{minipage}
  \begin{minipage}{.9\textwidth}
    \centering
   \includegraphics[clip,width=1\textwidth]{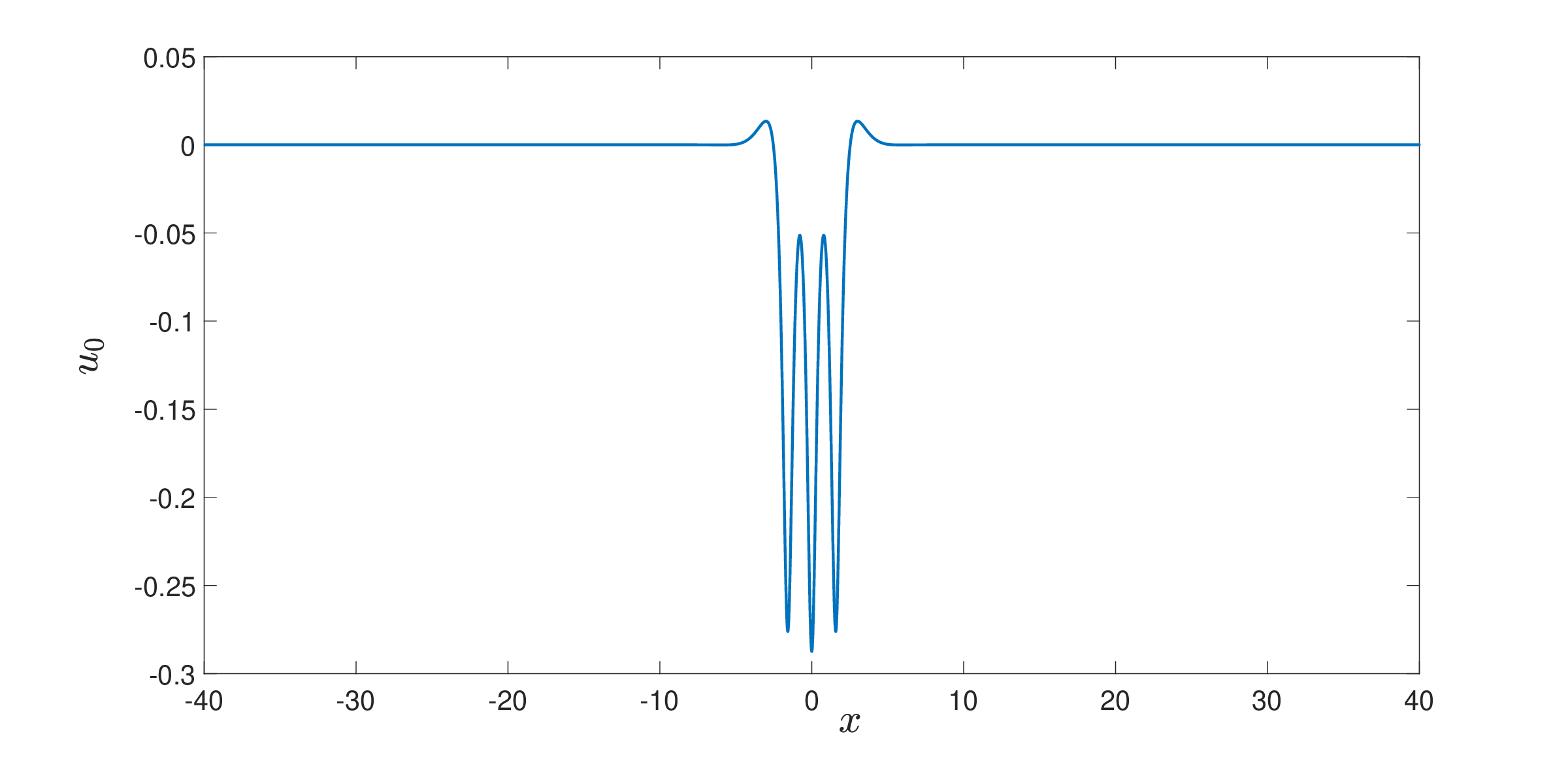}
  \end{minipage}
  \caption{Numerical approximations for the cgWE for the case $T=3$ (top-left), $T=0.5$ (top-right) and $T=0.25$ (bottom).}
    \label{fig : capillary}
  \end{figure}





\begin{theorem}\label{th: : proof capillary whitham}(Proof of   {solitary waves} in the capillary-gravity Whitham equation)\\
 {Let $r_{0,2} \bydef 5.3 \times 10^{-6}$, $r_{0,3} \bydef 8.7\times 10^{-9}$, $r_{0,4} \bydef 9.7\times 10^{-7}$, then there exists a unique even solution $\tilde{u}_i$ to \eqref{eq : whitham stationary} in $\overline{B_{r_{0,i}}(u_{0,i})} \subset \mathcal{H}_e$ for $T_i$ ($i \in {2,3,4}$) and $c=0.8$.}  { In addition, there exists a smooth curve 
\[
\left\{\tilde{u}_i(q) : q \in [d,\infty]\right\} \subset C^\infty(\R)
\]
such that $\tilde{u}_i(q)$ is a periodic solution to \eqref{eq : whitham stationary} with period $2q$. In particular, $\tilde{u}_i(\infty) = \tilde{u}_i$.}
\end{theorem}
\begin{proof}
 {For each $i \in \{2,3,4\}$, similarly as for the proof of Theorem \ref{th : proof whitham}, we are able to compute bounds  $\mathcal{Y}_0 $, $\mathcal{Z}_1$ and $\mathcal{Z}_2$   satisfying  \eqref{eq: definition Y0 Z1 Z2}.  Moreover, $r_{0,i}$ satisfies \eqref{condition radii polynomial}, leading to the proof of existence of $\tilde{u}_i$.} The regularity is obtained thanks to Proposition \ref{prop : regularity of the solution}  {and the branch of periodic solutions is established thanks to Theorem 3.17 in \cite{sh_cadiot}.}
\end{proof}

 {

\subsection{Proof of existence of a branch of solitary waves in the Whitham equation}\label{sec : continuation}

In this section we present a computer-assisted method for performing a rigorous continuation in the parameter $c$ in the Whitham equation. 
Indeed, the bounds obtained in Section \ref{sec : computation of the bounds} allows to set-up a Newton-Kantorovish approach to verify the existence of a branch of solutions thanks to the uniform contraction mapping theorem. Such a branch of solitary waves has been deeply studied and various existence results have recently been obtained (cf. \cite{arnesen_maximization_technique, CHEN_global_bifurcation, ehrnstrom_direct, truong_global_whitham} for instance). In particular, along the branch, the solution $u$ satisfies the inequality $u \leq \frac{c}{2}$ and at the end of the branch, there exists a velocity for which $u(0) = \frac{c}{2}$ (considering that $u$ is even). The later is a cusped solution of the Whitham equation and has been proven to have a $|x|^\frac{1}{2}$ singularity at the origin (cf. \cite{cusped_behavior_ehsnstrom}). In this section, we present a  computer-assisted proof for a part of such a  branch between $c = 1.05$ to $c = 1.21$. In particular, we obtain a uniform and explicit control the solutions with high accuracy.

More specifically, we base our approach on the setting developed in \cite{breden_polynomial_chaos,henot_marchal}. We use a finite expansion in the velocity $c$ of our approximate solution and approximate inverse in Chebyshev polynomials. Then, using the uniform contraction mapping theorem, we prove the existence of a branch of solitary waves, parametrized by $c$. In this section, since our zero finding problem \eqref{eq : whitham stationary} depends on $c$, we write $\mathbb{F}_c$ instead of $\mathbb{F}$ to emphasize the dependency on $c$.

Given an initial velocity $c_0 > 1$, a final velocity $c_1 > c_0$ and  $N_{cheb} \in \mathbb{N}$, we first compute a numerical approximate branch of solutions between $c_0$ and $c_1$ given by 
\begin{align*}
    v(c) \bydef  v_0 + 2 \sum_{n=1}^{N_{cheb}} v_n T_n\left(\frac{2(c-c_1)}{c_1-c_0} +1 \right) 
\end{align*}
where $T_n : [-1,1] \to \R$ is the $n$-th Chebyshev polynomial of the first kind, where $v_n \in H^4_{\om,e}$ for all $n \in \{0, \dots, N_{cheb}\}$. In particular, each $v_n$ follows the construction of Section \ref{ssec : approximate solution} and $v_n = \gamma^\dagger(V_n)$, where $V_n = \pi^N V_n$. That is $v_n$ has a representation as a vector of Fourier coefficients.

Now, let $\mathbb{L}_0 \bydef \mathbb{M}_0- \frac{1.05+1.21}{2} I_d$ and let $\mathcal{H}_{con}$ be the associated Hilbert space, as defined in \eqref{def : definition Hl}. Specifically,  $\mathcal{H}_{con}$ is associated to the following norm
\begin{align*}
    \|u\|_{\mathcal{H}_{con}} \bydef \|\mathbb{\Lambda}_\nu\mathbb{L}_0u\|_2
\end{align*}
for all $u \in \mathcal{H}_{con}$.
Now, we compute an approximate inverse $\mathbb{A}(c) : H^2_e \to \mathcal{H}_{con}$ for $D\mathbb{F}(v(c))$ as
\begin{align}\label{eq : approx inverse continuation}
  \mathbb{A}(c)  \bydef  (\mathbb{M}_0- cI_d)^{-1} \mathbb{\Lambda}_\nu^{-1}  \mathbb{B}(c) \mathbb{\Lambda}_\nu, \text{ where } ~~ \mathbb{B}(c) \bydef  \mathbb{B}_0 + 2 \sum_{n=1}^{N_{cheb}} \mathbb{B}_n T_n\left(\frac{2(c-c_1)}{c_1-c_0} +1 \right) 
\end{align}
and $\mathbb{B}_n : L^2_e \to L^2_e$ is a bounded linear operator. In particular, each $\mathbb{B}_n$ is constructed similarly as the operator $\mathbb{B}_T$ for $T=0$ in Section \ref{subsec : AT in the case T=0}. Specifically,
\[
\mathbb{B}_n = \mathbb{1}_{\R \setminus \om} + \gamma^\dagger(B_n), \text{ where } B_n = \pi_N \mathbb{W}_n + B^N_n
\]
for some sequence $W_n \in \ell^1_e$. Furthermore, $W_n$ is chosen so that 
\begin{align}
    \left(\gamma^\dagger(W_0) + 2 \sum_{n=1}^{N_{cheb}} \gamma^\dagger(W_n) T_n\left(\frac{2(c-c_1)}{c_1-c_0} +1 \right) \right)*(e_0 - \frac{2}{c}v(c)) \approx \cha  
\end{align}
for all $c \in [c_0,c_1]$. In other terms, the function $\gamma^\dagger(W_0) + 2 \sum_{n=1}^{N_{cheb}} \gamma^\dagger(W_n) T_n\left(\frac{2(c-c_1)}{c_1-c_0} +1 \right)$ is an approximate inverse of the function $e_0 - \frac{2}{c}v(c)$ on $\om.$ In practice we verify  that $v(c)$ satisfies Assumption \ref{ass : value of c and T} for each $c \in [c_0,c_1]$ in order to make sense of the construction of the coefficients $(W_n)$.

 Then, the following theorem, which is based on the uniform contraction mapping theorem, provides a sufficient condition for the existence of a branch of solitary waves between $c_0$ and $c_1$.
\begin{theorem}\label{th : radii polynomial continuation}
   Let $\mathbb{A}(c) : H^2_e \to \mathcal{H}_{con}$ be an injective bounded linear operator and let $\mathcal{Y}_0, \mathcal{Z}_1$ and $\mathcal{Z}_2$ be non-negative constants such that
  \begin{align}\label{eq: definition Y0 Z1 Z2 continuation}
  \nonumber
   \sup_{c \in [c_0,c_1]} \|\mathbb{A}(c){\mathbb{F}}_c(v(c))\|_{\mathcal{H}_{con}} &\leq \mathcal{Y}_0\\
    \sup_{c \in [c_0,c_1]} \|I_d - \mathbb{A}(c)D{\mathbb{F}_c}(v(c))\|_{\mathcal{H}_{con}} &\leq \mathcal{Z}_1\\\nonumber
    \sup_{c \in [c_0,c_1]} \|\mathbb{A}(c)\left({D}{\mathbb{F}_c}(v(c)) - D{\mathbb{F}_c}(v(c) + w)\right)\|_{\mathcal{H}_{con}} &\leq \mathcal{Z}_2r, ~~ \text{for all } w \in \overline{B_r(0)}.
\end{align}  
If there exists $r>0$ such that
\begin{equation}\label{eq : condition contraction continuation}
    \frac{1}{2}\mathcal{Z}_2r^2 - (1-\mathcal{Z}_1)r + \mathcal{Y}_0 < 0 \text{ and } \mathcal{Z}_1 + \mathcal{Z}_2 r < 1
 \end{equation}
then for every $c \in [c_0,c_1]$, there exists a unique $\tilde{u}(c) \in \overline{B_r(v(c))} \subset \mathcal{H}_{con}$ solving \eqref{eq : f(u)=0 on He}. Moreover the function $c \mapsto \tilde{u}(c)$ is of class $C^\infty.$
\end{theorem}

\begin{proof}
     The proof can be found in \cite{cont_global_bif_diag, cont_equilibria_pde,  cont_suspension_bridge} for instance. In particular, the regularity of the branch of solutions is provided by the smoothness of the fixed point operator $\mathbb{T}_c(u) \bydef u - \mathbb{A}(c) \mathbb{F}_c(u)$, which is smooth in $c$ since $\mathbb{A}(c) \mathbb{F}_c$ have a finite expansion in Chebyshev polynomials.
\end{proof}

The bounds in the previous theorem can be computed explicitly thanks to the analysis developed in Section \ref{sec : computation of the bounds}. In particular, we use that if $v(c) = v_0 +  2\sum_{n=1}^\infty v_n T_n\left(\frac{2(c-c_1)}{c_1-c_0} +1 \right) \in \mathcal{H}_{con}$, then
\begin{align*}
    \sup_{c \in [c_0,c_1]} \|v(c)\|_{\mathcal{H}_{con}} \leq \|v_0\|_{\mathcal{H}_{con}} +  2\sum_{n=1}^\infty \|v_n\|_{\mathcal{H}_{con}}. 
\end{align*}
Similarly, given a bounded linear operator $\mathbb{B}(c) \bydef  \mathbb{B}_0 + 2 \sum_{n=1}^{\infty} \mathbb{B}_n T_n\left(\frac{2(c-c_1)}{c_1-c_0} +1 \right) $, we have 
\begin{align*}
    \sup_{c \in [c_0,c_1]} \|\mathbb{B}(c)\|_{\mathcal{H}_{con}} \leq \|\mathbb{B}_0\|_{\mathcal{H}_{con}} +  2\sum_{n=1}^\infty \|\mathbb{B}_n\|_{\mathcal{H}_{con}}. 
\end{align*}
Consequently, since $v(c)$, $\mathbb{A}(c)$ and $\mathbb{F}(c)$ have a finite expansion in Chebyshev polynomials (for the variable $c$),  the above inequalities combined with the analysis of Section \ref{sec : computation of the bounds} allow to compute the bounds of Theorem \ref{th : radii polynomial continuation}.

Numerically, we start at the approximate solution in Fourier coefficients obtained in Section \ref{sec : proof in the case T=0} and we use a parameter continuation to construct a finite number of approximate solutions $V_0(c_k) \in \mathscr{h}_e$ at the Chebyshev nodes 
\[
c_k \bydef \frac{c_0+c_1}{2} + \frac{c_1-c_0}{2}\cos\left(\frac{(2k+1)\pi}{2N_{cheb}}\right)
\]
for $k = 0, \dots, N_{cheb}-1.$  Then, we use an FFT to obtain the coefficients $V_n$ such that $V(c) = V_0 + 2 \sum_{n=1}^{N_{cheb}} V_n T_n\left(\frac{2(c-c_1)}{c_1-c_0} +1 \right)$ is our approximate branch in the Fourier coefficients. In particular, each $V_n$ has a function representation with a zero trace, using the projection defined in \eqref{eq : projection in h^k_0}. Then, our approximate branch $v(c) \in H^4_{\om,e}$ is defined as $v(c) \bydef \gamma^\dagger(V(c)).$ Similarly, computing an approximate inverse $\mathbb{A}(c_k)$ at each Chebychev node allows to obtain the continuum of approximate inverses $\mathbb{A}(c)$ for all $c \in [c_0,c_1]$ as in \eqref{eq : approx inverse continuation}. 

The rigorous FFT and inverse FFT functions are implemented in the package RadiiPolynomial.jl \cite{julia_olivier} and are based on the IntervalArithmetic.jl package \cite{julia_interval}.
Using rigorous computation for the bounds from Theorem \ref{th : radii polynomial continuation} in \cite{julia_cadiot}, we obtain a proof for the following theorem.

\begin{theorem}\label{th : proof of a branch}
    For every $c \in [1.05, 1.21]$, there exists a smooth even solution $\tilde{u}(c) \in H^\infty(\R)$ to \eqref{eq : whitham stationary} and the function $c \mapsto \tilde{u}(c)$ is continuous. Furthermore, $\underset{c \in [1.05,1.21]}{\sup}\|\tilde{u}(c) - v(c)\|_{\mathcal{H}_{con}} \leq 3.2\times 10^{-4} $.
\end{theorem}

\begin{proof}
    The rigorous computation of the bounds of Theorem \ref{th : radii polynomial continuation} is presented in the code \cite{julia_cadiot}.
In practice, we cut the interval $[1.05, 1.21]$ between five subintervals $[c_{i},c_{i+1}]$ ($i = 0, \dots, 4$), with $1.05 = c_0 < c_1 < \dots < c_4 < c_5 = 1.21$, and apply Theorem \ref{th : radii polynomial continuation} on each of them. In particular, we prove that there exists $0< r_{min}^{(i)}< r_{max}^{(i)} $ such that, given approximate branches $v_i(c)$, we obtain the existence of five branches of solutions $\tilde{u}_i(c)$ ($i = 0, \dots, 4$) defined on $[c_i,c_{i+1}]$ respectively, with $\tilde{u}_i(c)$ being the unique solution to \eqref{eq : whitham stationary} in $\overline{B_r(v_i(c))}$ for all $r \in [r_{min}^{(i)}, r_{max}^{(i)}]$.

To ensure continuity of the branch on $[1.05, 1.21]$, we need to prove that $\tilde{u}_i(c_{i+1}) = \tilde{u}_{i+1}(c_{i+1})$ for all $i \in \{0, \dots, 3\}$.  This is achieved using the uniqueness of $\tilde{u}_i(c)$ in $\overline{B_r(v_i(c))}$ for all $r \in [r_{min}^{(i)}, r_{max}^{(i)}]$. Indeed, we verify using rigorous numerics that 
\[
\overline{B_{r_{min}^{(i)}}(v_i(c_{i+1}))} \subset \overline{B_{r_{max}^{(i+1)}}(v_{i+1}(c_{i+1}))}
\]
for all $i = 0, \dots, 3.$ The uniqueness of each $\tilde{u}_i$ provides the desired proof.  

Now, choosing $r = \max_{i= 0, \dots, 4} r_{min}^{(i)}$, we  demonstrate that $r \leq 3.2 \times 10^{-4}$ and obtain a uniform control of the branch of solutions on $[1.05, 1.21].$
    We prove the smoothness of each function $\tilde{u}(c)$ using Proposition \ref{prop : regularity of the solution} and  by verifying that 
    \begin{align*}
        \|\tilde{u}(c)\|_\infty < \frac{c}{2}
    \end{align*}
    for each $c \in [1.05, 1.21]$ thanks to the analysis developed in Section \ref{ssec : regularity of the solution}.
\end{proof}

  \begin{figure}[h!]
  \centering
  \begin{minipage}{.52\textwidth}
   \centering
  \includegraphics[clip,width=1\textwidth]{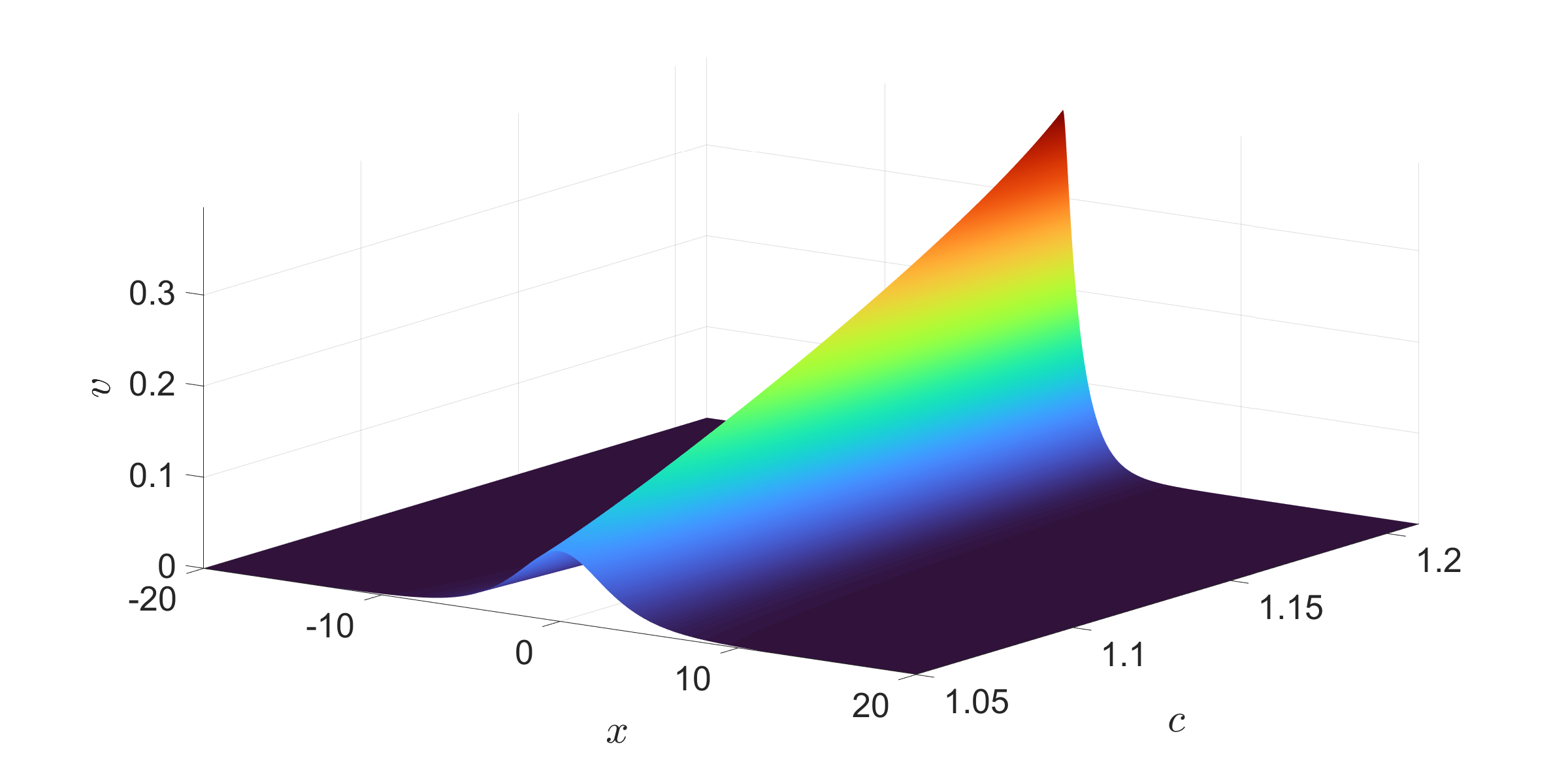}
  \end{minipage}%
  \begin{minipage}{.52\textwidth}
    \centering
   \includegraphics[clip,width=1\textwidth]{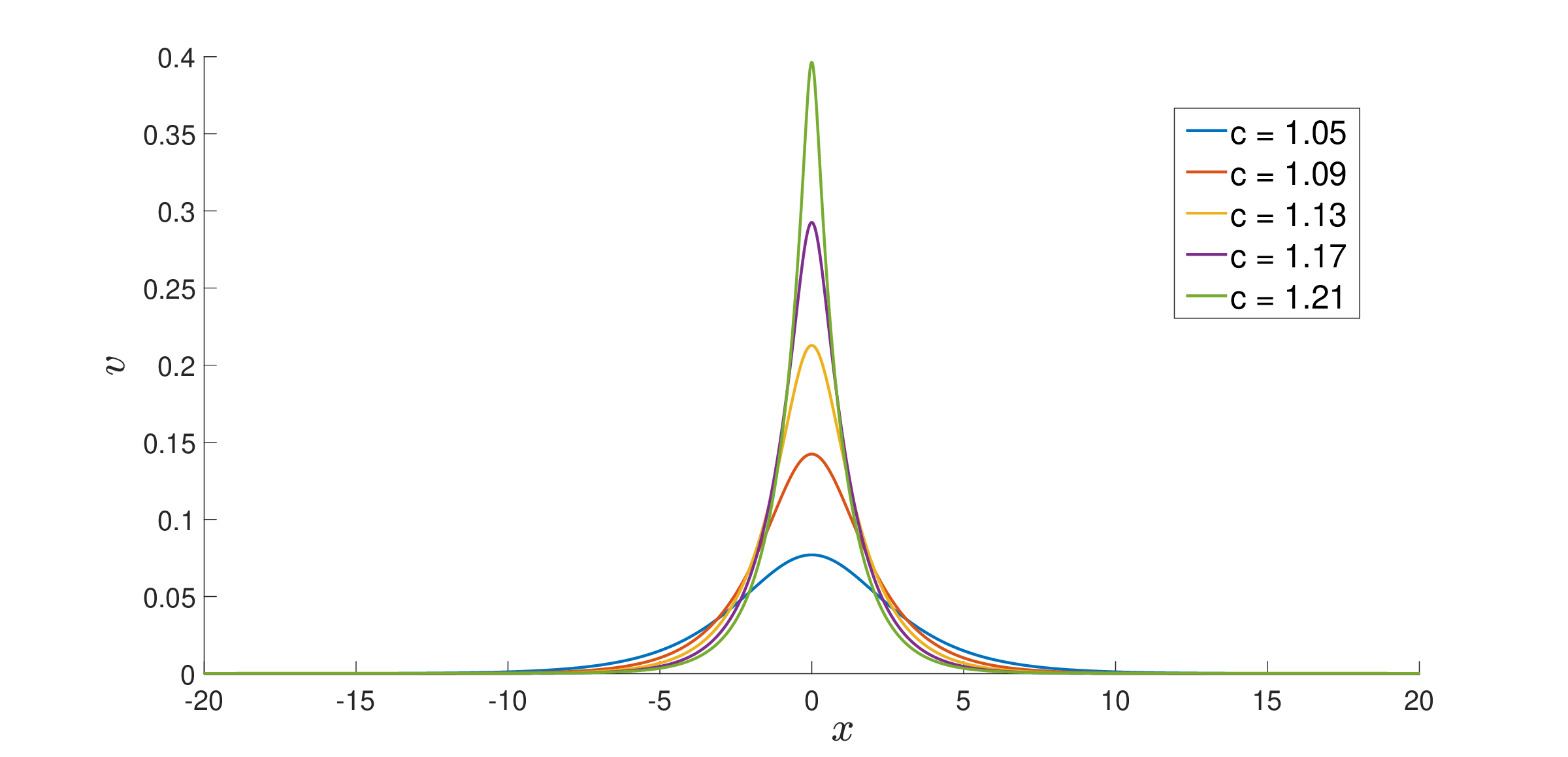}
  \end{minipage}
  \caption{Visualization of the branch of solitary waves corresponding to Theorem~\ref{th : proof of a branch} (left) and  sample of solitary waves for specific values of the velocity $c$ (right).}
    \label{fig : continuation}
  \end{figure}

As mentioned earlier, the branch of solitary waves proven in the above theorem  has been intensively studied. In particular, the branch is known to display a cusped wave for a specific value of velocity $c^*$ for which $u(0) = \frac{c^*}{2}$. For the part of the branch presented above, we consider velocities strictly smaller than the critical one $c^*$ and avoid the singular behavior of the cusp. Numerically, the critical value seems to be in between $1.22$ and $1.23$. Approximating solitary waves with Fourier series becomes more and more difficult as one approaches $c^*$.  In order to get closer to the cusping behavior and actually provide an existence proof of the singular wave, one would have to use a different basis of approximation than Fourier series. For instance, following the framework established in \cite{javier_convexity_highest}, one could use the Clausen functions to improve the approximation.  Finally, note that since the proof of Theorem \ref{th : proof of a branch} is obtain thanks to the contraction mapping theorem, we obtain that the branch is locally unique (in the class of even smooth functions). Consequently, there is no possible branching out between $c = 1.05$ and $c = 1.21$. One could then think about gluing the bifurcation analysis at $c \approx 1$ and the cusping phenomenon at $c \approx 1.23$ to the above branch in order to obtain a full understanding of the (whole) branch.

}

\section{Spectral stability}\label{sec : stability}

In this section, using the analysis derived in Section 5 in \cite{unbounded_domain_cadiot}, we establish a method to prove the spectral stability of solitary wave solutions to \eqref{eq : original Whitham}.  The approach is again computer-assisted and relies on the analysis of Sections \ref{sec : computer-assisted approach} and \ref{sec : computation of the bounds}.

We first derive sufficient conditions under which the spectral stability of solitary wave solutions is achieved. In particular, these conditions involve  requirements on the spectrum of the linearization around the solitary wave. Having such conditions in mind,  we derive a general computer-assisted approach to control the spetrum of the linearization. Combining these results, we are able to conclude about the stability of the solutions obtained in Theorems \ref{th : proof whitham} and \ref{th: : proof capillary whitham}. 

In this section, we do not restrict the operators to even functions anymore but instead use a subscript ``e", if necessary, to establish a restriction to the even symmetry.

\subsection{Conditions for stability}\label{sec : condition for stability}

 Let us fix $c$ and $T$ satisfying Assumption \ref{ass : value of c and T} and let $\tilde{u} \in H^\infty(\R)$ be a solution to \eqref{eq : whitham stationary} obtained thanks to Theorem \ref{th: radii polynomial}. In particular, assume that 
 \begin{align}\label{def : solution stability}
     \tilde{u} \in \overline{B_{r_0}(u_0)} \subset \mathcal{H}_e
 \end{align}
 for some $u_0 \in H^4_e$ constructed as in Section \ref{ssec : approximate solution} and $r_0>0.$ Moreover, $u_0$ satisfies Assumption \ref{ass : u0 is smaller than} and it has a representation as a sequence of Fourier coefficients given by $U_0 \in X^4_e$ and $U_0 = \pi^N U_0$. If $T=0$,  assume that 
 \begin{align}
     \|U_0\|_1 + \frac{r_0}{4\sqrt{\nu} \sigma_0} < \frac{c}{2}.
 \end{align}
 Using the reasoning of Section \ref{ssec : regularity of the solution}, this implies that
 \begin{align}\label{condition solution smooth}
     \|\tilde{u}\|_\infty \leq \|U_0\|_1 + \frac{r_0}{4\sqrt{\nu} \sigma_0} < \frac{c}{2} ~~ \text{ and } ~~  \|\tilde{u} - u_0\|_\infty \leq \frac{r_0}{4\sqrt{\nu}\sigma_0}
 \end{align}
 which justifies the fact that $\tilde{u} \in H^\infty(\R)$ (cf. Proposition \ref{prop : regularity of the solution}). The inequalities in \eqref{condition solution smooth} were actually proven in Theorem \ref{th : proof whitham} for the solution $\tilde{u}_1$.
 Given this construction of $\tilde{u}$, we derive conditions under which $\tilde{u}$ is spectrally stable. 
 
 {
In this section, we first redefine $\mathbb{F}$, $\mathbb{L}$ and $\mathbb{G}$ as follows
\begin{align*}
    \mathbb{F}(u) \bydef \begin{cases}
      -\mathbb{M}_0u + cu - u^2  &\text{ if } T=0\\
        \mathbb{M}_Tu - cu + u^2 &\text{ if } T>0
    \end{cases}, ~  \mathbb{L} \bydef \begin{cases}
         -\mathbb{M}_0 + cI_d  &\text{ if } T=0\\
        \mathbb{M}_T - cI_d &\text{ if } T>0
    \end{cases}
    ~ \text{ and } ~ \mathbb{G}(u) \bydef \begin{cases}
      - u^2  &\text{ if } T=0\\
        u^2 &\text{ if } T>0.
    \end{cases}
\end{align*}
Note that the above redefinitions allow to obtain a positive linear part in $\mathbb{F}$. Consequently, in both the cases $T=0$ and $T>0$, we can focus our attention on the negative spectrum (as detailed in Section \eqref{sec : kdv eigen} below). Similarly, the associated operators $F$, $L$ and $G$ on Fourier coefficients are redefined correspondingly.

Moreover, we redefine the  Hilbert space $\mathcal{H}$ choosing a different regularity $\mathbb{\Lambda}_0 =I_d$ instead of $\mln$, that is the norm on $\mathcal{H}$ becomes
\begin{align*}
   \|u\|_{\mathcal{H}} = \|\mathbb{L}u\|_2.
\end{align*}
In particular, because $\tilde{u} \in L^\infty(\R)$, we obtain  that 
\[
D\mathbb{F}(\tilde{u}) : \mathcal{H} \to L^2
\]
is a bounded linear operator. Then, we use a subscript $``e"$ to specify the restriction of $\mathbb{F}$ to even functions. For instance
\begin{align}\label{eq : even restriction}
    D{\mathbb{F}}_e(\tilde{u}) : \mathcal{H}_e \to L^2_e
\end{align}
is the restriction of $D{\mathbb{F}}(\tilde{u})$ to even functions. 
  In fact, using Lemma \ref{lem : computation of Z1} and the fact that $\tilde{u}$ has been obtained thanks to Theorem \ref{th: radii polynomial}, 
  we obtain that $D\mathbb{F}_e(\tilde{u}) : \mathcal{H}_e \to L^2$ has a 
  bounded inverse.  Moreover, the analysis derived in Section 3 of 
  \cite{Stefanov2018SmallAT} is applicable and we can use it to study the spectral stability of $\tilde{u}$. In particular, the reasoning of Sections 
 3.2 and 3.3 of \cite{Stefanov2018SmallAT} is readily applicable to $\partial_x D\mathbb{F}(\tilde{u})$ and we resume the obtained results in
 the following lemma. Note that a more general approach is available at 
 \cite{index_theorem_Lin_Chongchun}.  }

  {
\begin{lemma}\label{lem : result on eigenvalues}
Assume that 

{\em (P1)} $D{\mathbb{F}}(\tilde{u})$ has a simple negative eigenvalue $\lambda^{-}$\\
{\em (P2)} $D{\mathbb{F}}(\tilde{u})$ has no negative eigenvalue other than $\lambda^{-}$\\
{\em (P3)} 0 is a simple eigenvalue of $D{\mathbb{F}}(\tilde{u})$.

If 
\begin{align}\label{eq : condition for stability}
         \left(\tilde{u}, D\mathbb{F}_e(\tilde{u})^{-1}\tilde{u}\right)_2 < 0, 
\end{align}
then $\tilde{u}$ is (spectrally) stable.
\end{lemma}
}
 {
Using the previous lemma, the spectral stability of $\tilde{u}$ requires controlling the negative part of the spectrum of $D\mathbb{F}(\tilde{u})$. Consequently, we recall in the next section some tools from \cite{unbounded_domain_cadiot} in order to study eigenvalue problems. Then, condition \eqref{eq : condition for stability} involve the negativity of the so-called Vakhitov-Kolokolov quantity. Controlling such a quantity is a complex task when one does not possess an explicit control on $\tilde{u}$ and $D\mathbb{F}_e(\tilde{u})^{-1}$. Using the analysis derived in Section \ref{sec : computation of the bounds} combined with Proposition 6.14 in \cite{unbounded_domain_cadiot}, we provide a computer-assisted approach to verify that condition \eqref{eq : condition for stability} is satisfied. We adapt Proposition 6.14 in \cite{unbounded_domain_cadiot} to \eqref{eq : whitham stationary} and obtain the following result.

\begin{prop}\label{prop:stable}
Let $\tilde{u} \in \overline{B_{r}(u_0)}$ be a solution to \eqref{eq : whitham stationary} obtained thanks to Theorem \ref{th: radii polynomial}. Then, let $V_0 \bydef A_TU_0,$ where $A_T$ is defined in Section \ref{ssec : approximate inverse}, and let $V = (V_n)_n \in X^4_0$ be the projection of $V_0$ in the kernel of $\mathcal{T}$ (using the construction of Section \ref{ssec : approximate solution}). Finally, let $\epsilon$ be a constant satisfying
\[
\epsilon \geq  \|D\mathbb{F}_e(\tilde{u})^{-1}\|_{2}\| u_0 - D\mathbb{F}(u_0) \gamma^\dagger(V)\|_2 + \frac{r}{2\sqrt{\nu} \sigma_0}|\Omega_0|^{\frac{1}{2}}\|D\mathbb{F}_e(\tilde{u})^{-1}\|_{2}  \|V\|_{2}.
\]
If there exists $\tau<0$ such that
\[
    |\Omega_0| \sum_{n \in \mathbb{Z}}(U_0)_n\overline{V_n} + \epsilon |\Omega_0|^{\frac{1}{2}}\|U_0\|_2 + 2\|D\mathbb{F}_e(\tilde{u})^{-1}\|_{2}(|\Omega_0|^{\frac{1}{2}}\|U_0\|_2 + r)r \leq \tau
\]
then
\[
    \int_{\mathbb{R}} \tilde{u} D\mathbb{F}_e(\tilde{u})^{-1}\tilde{u} < \tau.
\]
\end{prop}

\begin{proof}
    The proof is obtained using the proof of Proposition 6.14 in \cite{unbounded_domain_cadiot} combined with the fact that 
    \[
    \|D\mathbb{G}(\tilde{u})\gamma^\dagger(V) - D\mathbb{G}(u_0)\gamma^\dagger(V)\|_2 \leq 2 \|\tilde{u}-u_0\|_\infty \|\gamma^\dagger(V)\|_2 \leq \frac{r}{2\sqrt{\nu}\sigma_0} |\Omega_0|^\frac{1}{2} \|\tilde{u}-u_0\|_\infty \|V\|_2,
    \]
    where we used \eqref{condition solution smooth} for the last step.
\end{proof}

 {\begin{remark}
   A value for $\epsilon$ in the previous proposition can be obtained thanks to rigorous numerics. Indeed,  the quantity $\| u_0 - D\mathbb{F}(u_0) \gamma^\dagger(V)\|_2$ can be bounded following similar steps as the ones used for the computation of the bound $\mathcal{Y}_0$ in Lemma \ref{lem : bound Y_0}. Moreover,  an upper bound for $\|D\mathbb{F}_e(\tilde{u})^{-1}\|_{2}$  can be  obtained combining \eqref{condition solution smooth} and Lemma \ref{lem : computation of Z1}. Such computations are implemented in the code \cite{julia_cadiot}.
\end{remark}}

}

\subsection{Proof of eigencouples of \boldmath\texorpdfstring{$D\mathbb{F}(\tilde{u})$}{DF(u)}\unboldmath}\label{sec : kdv eigen}

First, notice that $D\mathbb{F}(\tilde{u})$ only possesses real eigenvalues as it is  self-adjoint on $L^2$. 
Then, similarly as what was achieved in Section 5 in \cite{unbounded_domain_cadiot}, the goal is to set-up a zero finding problem to prove eigencouples of $D\mathbb{F}(\tilde{u}).$ Following Lemma \ref{lem : result on eigenvalues}, it is enough to study the non-positive part of the spectrum of $D\mathbb{F}(\tilde{u})$ in order to conclude about stability. 
In particular, we prove that the non-positive part of the spectrum of $D\mathbb{F}(\tilde{u})$ only contains eigenvalues with finite multiplicity.  {Before presenting the proof, notice that, given $\lambda < \sigma_0$, we have
\begin{align}\label{eq : min of m}
     |m_{T}(\xi)-c| - \lambda &\geq  \sigma_0-\lambda >0
\end{align}
for all $\xi \in \R$. Therefore, for all $\lambda < \sigma_0$, we can define the positive linear operator $\mathbb{L}_{\lambda}$ as follows
\begin{align}\label{def : L lambda unbounded}
    \mathbb{L}_{\lambda} \bydef  \mathbb{L} - \lambda I_d 
\end{align}
and define the associated Hilbert space $\mathcal{H}_\lambda$ as in \eqref{def : definition Hl} with  norm $\|u\|_{\mathcal{H}_\lambda} \bydef \|\mathbb{L}_\lambda u\|_2$ for all $u \in \mathcal{H}_\lambda$. Using \eqref{eq : min of m}, we obtain that $\mathbb{L}_\lambda : \mathcal{H}_\lambda \to L^2$ is an isometric isomorphism. }

\begin{lemma}\label{lem : spectrum is eigenvalues}
Let $\sigma_0$ be defined in Lemma \ref{prop : analyticity and value of a} and let $\lambda_{\max}>0$ be defined as 
\begin{align}\label{def : lambda_max}
    \lambda_{\max} \bydef  \begin{cases}
         \min\left\{\sigma_0, ~c -2\left(\|U_0\|_1 + \frac{r_0}{4\sqrt{\nu}\sigma_0}\right)\right\} &\text{ if } T=0\\
          \sigma_0  &\text{ if } T>0.
    \end{cases}
\end{align}
    Let $\lambda < \lambda_{\max}$ be such that $D\mathbb{F}(\tilde{u}) - \lambda I_d$ is non-invertible. Then $\lambda$ is an eigenvalue of $D\mathbb{F}(\tilde{u})$ with finite multiplicity.
\end{lemma}
\begin{proof}
First, since $\lambda < \lambda_{\max} \leq \sigma_0$, we obtain that $\mathbb{L}_{\lambda} = {\mathbb{L}} - \lambda I_d$ is a positive definite linear operator and is invertible. Suppose that $T>0$. Then 
    \begin{align*}
        D\mathbb{F}(\tilde{u}) - \lambda I_d = \mathbb{L}_\lambda + 2 \tilde{\mathbb{u}} = \mathbb{L}_\lambda\left(I_d + 2\mathbb{L}_\lambda^{-1} \tilde{\mathbb{u}}\right).
    \end{align*}
Moreover, using the proof of Lemma 3.1 in \cite{unbounded_domain_cadiot}, we obtain that $2\mathbb{L}_\lambda^{-1} \tilde{\mathbb{u}} : L^2 \to L^2$ is compact and therefore $D\mathbb{F}(\tilde{u}) - \lambda I_d$ is a Fredholm operator. We conclude the  proof for the case $T>0$ using the Fredholm alternative. Now, if $T=0$, then 
 \begin{align*}
        D\mathbb{F}(\tilde{u}) - \lambda I_d = \mathbb{L}_\lambda\left(I_d - 2\mathbb{L}_\lambda^{-1} \tilde{\mathbb{u}}\right) = \mathbb{L}_\lambda(I_d - \frac{2}{c-\lambda}\tilde{\mathbb{u}}) \left(I_d - 2(I_d - \frac{2}{c-\lambda}\tilde{\mathbb{u}})^{-1}\left(\mathbb{L}_\lambda^{-1} - \frac{1}{c-\lambda} I_d \right)\tilde{\mathbb{u}}\right).
    \end{align*}
    Now, using \eqref{condition solution smooth}, we have that 
    \begin{align*}
        \tilde{u} \leq \|U_0\|_1 + \frac{r_0}{4\sqrt{\nu}\sigma_0} \leq \frac{c-\lambda_{\max}}{2} < \frac{c-\lambda}{2}.
    \end{align*}
    as $\lambda < \lambda_{\max}$. Therefore, since $\tilde{u}$ is smooth, we have that $(I_d - \frac{2}{c-\lambda}\tilde{\mathbb{u}})^{-1} : L^2 \to L^2$ is bounded. Moreover, using again the proof of Lemma 3.1 in \cite{unbounded_domain_cadiot}, we obtain that $\left(\mathbb{L}_\lambda^{-1} - \frac{1}{c-\lambda} I_d \right)\tilde{\mathbb{u}} : L^2 \to L^2$ is compact, which implies that 
    $2(I_d - \frac{2}{c-\lambda}\tilde{\mathbb{u}})^{-1}\left(\mathbb{L}_\lambda^{-1} - \frac{1}{c-\lambda} I_d \right)\tilde{\mathbb{u}} : L^2 \to L^2$ is compact as well. We obtain that $D\mathbb{F}(\tilde{u}) - \lambda I_d$ is a Fredholm operator, which concludes the proof.
\end{proof}
Using the above lemma combined with Lemma \ref{lem : result on eigenvalues}, we obtain that we can study the stability of $\tilde{u}$ by controlling the non-positive eigenvalues of $D\mathbb{F}(\tilde{u})$. Consequently, we  expose a computer-assisted strategy to prove the existence of an eigencouple $(\tilde{\lambda},\tilde{\psi})$, given an approximation $(\lambda_0, \psi_0)$. 

 Let $\psi_0 \in \mathcal{H}_\om$ be an approximate eigenvector of $D\mathbb{F}(\tilde{u})$
associated to an approximate eigenvalue $\lambda_0 \leq 0$. Moreover, using the construction presented in Section \ref{ssec : approximate solution}, we assume that $\psi_0$ is determined numerically and that there exists $\Psi_0 \in X^4$ such that
\begin{align*}
    \psi_0 = \gamma^\dagger(\Psi_0) \in H^4_{\om} \text{ with }~~ \Psi_0 = \pi^N \Psi_0,
\end{align*}
where  $\pi^N$ is defined in \eqref{def : projection truncation}.
 Moreover, denote by $H_1 \bydef \mathbb{R} \times \mathcal{H}_{\lambda_0}$ and $H_2 \bydef \mathbb{R} \times L^2$ the Hilbert spaces endowed respectively with the norms
\[
    \|(\nu,u)\|_{H_1} \bydef \left( |\nu|^2 + \|u\|_{\mathcal{H}_{\lambda_0}}^2\right)^{\frac{1}{2}} ~~ \text{ and } ~~
    \|(\nu,u)\|_{H_2} \bydef \left( |\nu|^2 + \|u\|_2^2\right)^{\frac{1}{2}}.
\]
In particular, we have
\begin{align}\label{def : x0}
    x_0 \bydef (\lambda_0,\psi_0) \in H_1
\end{align}
by construction.
Denote $x = (\nu,u) \in H_1$ and define the augmented zero finding problem $\overline{\mathbb{F}} : H_1 \to H_2$ as 
\begin{equation}\label{eq : f(u) = 0 on H1}
    \overline{\mathbb{F}}(x) \bydef \begin{pmatrix}(\psi_0- u,\psi_0)_2\\
    D\mathbb{F}(\tilde{u}) u - \nu u\end{pmatrix} =0
\end{equation}
for all $x = (\nu,u) \in H_1$. Zeros of $\overline{\mathbb{F}}$ are equivalently eigencouples of $D\mathbb{F}(\tilde{u}).$ Moreover, similarly as in \cite{unbounded_domain_cadiot}, we define $\oF_0 : H_1 \to H_2$ as 
\begin{align*}
    \overline{\mathbb{F}}_0(x) \bydef \begin{pmatrix}(\psi_0-u,\psi_0)_2\\
    D\mathbb{F}({u}_0) u - \nu u \end{pmatrix}.
\end{align*}
Note that the analysis derived in \cite{unbounded_domain_cadiot} does not readily allow to investigate $D\oF (x_0)$ as $\tilde{u}$ does not necessarily have a support contained on $\om$. However, the theory applies to $D\oF_0 (x_0)$. Moreover, using that $\|\tilde{u}- u_{0}\|_{\mathcal{H}} \leq r_0$ by assumption, we can build an approximate inverse for  $D\oF_0 (x_0)$ and use the control on $\|\tilde{u}- u_{0}\|_{\mathcal{H}}$ to obtain rigorous estimates for $D\oF (x_0)$ (cf. Section 5 in \cite{unbounded_domain_cadiot}).

Now, the map $\oF_0$ has a periodic correspondence on $\om$. Indeed, given $\lambda \leq 0 $, define $L_{\lambda}$ as 
\begin{align}\label{def : L lambda periodic}
    L_\lambda \bydef {L} - \lambda I_d.
\end{align}
Then, define $\mathscr{h}_{\lambda_0}$ to be the following Hilbert space defined  as
\begin{align}
    \mathscr{h}_{\lambda} \bydef \left\{U \in \ell^2(\mathbb{Z}), ~~ \|U\|_{\mathscr{h}_{\lambda}} \bydef \|L_{\lambda}U\|_2 < \infty\right\}.
\end{align}
Moreover, define the Hilbert spaces $X_1 \bydef \mathbb{R}\times \mathscr{h}_{\lambda_0}$ and $X_2 \bydef \mathbb{R} \times \ell^2$  associated to their norms
\[
    \|(\nu,U)\|_{X_1} \bydef \left( |\nu|^2 + |\om|\|U\|_{\mathscr{h}_{\lambda_0}}^2\right)^{\frac{1}{2}}  ~~ \text{and} ~~
    \|(\nu,U)\|_{X_2} \bydef \left( |\nu|^2 + |\om|\|U\|_2^2\right)^{\frac{1}{2}}.
\]
 Finally, we  define a corresponding operator $\overline{F} : X_1 \to X_2$ as 
 \begin{align*}
    \overline{F}(\nu,U) \bydef \begin{pmatrix}|\om|(\Psi_0- U,\Psi_0)_2\\
   D{F}(U_0)U - \nu  U\end{pmatrix}.   
\end{align*}

  At this point, we require the construction of an approximate inverse $D\oF_0(x_0) : H_1 \to H_2$. Let $\oL : H_1 \to H_2$ be defined as follows
 \begin{align*}
     \oL \bydef \begin{pmatrix}
         1&0\\
         0& \mathbb{L}_{\lambda_0}
     \end{pmatrix}.
 \end{align*}
 Then, by construction $\oL$ is an isometric isomorphism between $H_1$ and $H_2$. Moreover, denote $\overline{\mathbb{1}}_\om : H_2 \to H_2$, $\overline{\pi}^N : X_2 \to X_2$ and $\overline{\pi}_N : X_2 \to X_2$ the projection operators defined as 
 \begin{align*}
     \overline{\mathbb{1}}_\om \bydef \begin{pmatrix}
         1&0\\
         0 & \cha 
     \end{pmatrix}, ~~ \overline{\pi}^N \bydef \begin{pmatrix}
         1&0\\
         0 & \pi^N
     \end{pmatrix} \text{ and }~ \overline{\pi}_N \bydef \begin{pmatrix}
         0&0\\
         0 & \pi_N 
     \end{pmatrix}.
 \end{align*}
 Now, we build an approximate inverse for $D\oF_0(x_0) : H_1 \to H_2$ following the construction of Section \ref{ssec : approximate inverse}. Indeed, we build a linear operator $\overline{B}_T \bydef \overline{B}^N_T +  \overline{\pi}_N \begin{pmatrix}
     0&0\\
     0 & \mathbb{W}_T
 \end{pmatrix} : X_2 \to X_2$ such that $\overline{B}^N_T = \overline{\pi}^N \overline{B}^N_T\overline{\pi}^N$ and  define $\overline{A}_T : X_1 \to X_2$ as
 \begin{align}\label{operator AT extra equation}
     \overline{A}_T \bydef \begin{pmatrix}
     1&0\\
     0& L_{\lambda_0}^{-1}
 \end{pmatrix}\overline{B}_T.
 \end{align}
Intuitively, $\overline{A}_T$  is an approximation of the inverse of $ D\overline{F}(\lambda_0,\Psi_0)$. Moreover, define $W_T = e_0$ (where $e_0$ is given in \eqref{def : definition of e0}) for all $T>0$ and $W_0 = \pi^N W_0 \in \ell^1$. $W_0$ is a sequence chosen such that $W_0 = \pi^N W_0$ and $$W_0*(e_0 - \frac{2}{c - \lambda_0}U_{0} ) \approx e_0.$$
Moreover, denote $\overline{B}_{T}^N \bydef \begin{pmatrix}
    b_{1,1}& B_{1,2}^*\\
    B_{2,1}& B_{2,2}
\end{pmatrix}$ 
where $b_{1,1} \in \R$, $B_{2,2} = \pi^N B_{2,2}\pi^N$ and $B_{1,2}, B_{2,1} \in \ell^2$ are such that $B_{1,2} = \pi^N B_{1,2}$ and $B_{2,1} = \pi^N B_{2,1}$.

Let us denote $\mathcal{B}_{\om}(H_2)$ the following set
\begin{align*}
    \mathcal{B}_{\om}(H_2) \bydef \{\oB \in \mathcal{B}(H_2), ~ \oB = \overline{\mathbb{1}}_\om \oB \overline{\mathbb{1}}_\om \}
\end{align*}
where $\mathcal{B}(H_2)$ is the set of linear bounded operators on $H_2$. Moreover, define $H_{2,\om}$ as 
\[
H_{2,\om} \bydef \{x \in H_2, ~ x = \overline{\mathbb{1}}_\om x\}.
\]
Recall the definition of the following maps introduced in \cite{unbounded_domain_cadiot} $\overline{\gamma} : H_2 \to X_2$, $\overline{\gamma}^\dagger : X_2 \to H_2$, $\overline{\Gamma} : \mathcal{B}(H_2) \to \mathcal{B}(X_2)$ and $\overline{\Gamma}^\dagger : \mathcal{B}(X_2) \to \mathcal{B}(H_2)$
\begin{align*}
     \og \bydef \begin{pmatrix}
        1 & 0\\
        0 & \gamma
    \end{pmatrix}
    ~~ &\text{ and } ~~ \ogd \bydef \begin{pmatrix}
        1 & 0\\
        0 & \gamma^\dagger
    \end{pmatrix}\\
     \overline{\Gamma}\left(\overline{\mathbb{B}}\right) \bydef \og \overline{\mathbb{B}} \ogd
    ~~ &\text{ and } ~~ \overline{\Gamma}^\dagger\left(\overline{B}\right) \bydef \ogd \overline{B} \og
\end{align*}
for all $\overline{\mathbb{B}} \in \mathcal{B}(H_2)$ and $\overline{B} \in X_2$. 
Then, we recall Lemma 5.1 in \cite{unbounded_domain_cadiot}, which provides a one to one relationship between bounded linear operators on $X_2$ and the ones in $\mathcal{B}_{\om}(H_2)$.
 \begin{lemma}\label{prop : operator link L^2 extra equation}
 The map $\og : H_{2,\om} \to X_2$ (respectively $\overline{\Gamma} : \mathcal{B}_\om(H_2) \to \mathcal{B}(X_2)$) is an isometric 
    isomorphism whose inverse is $\ogd : X_2 \to H_{2,\om}$ 
    (respectively $\overline{\Gamma}^\dagger : \mathcal{B}(X_2) \to \mathcal{B}_\om(H_2)$). In particular,
    \begin{align*}
        \|(\nu,u)\|_{H_2} = \|(\nu,U)\|_{X_2} ~~ \text{ and } ~~ \|\overline{\B}_\om\|_{H_2} = \|\overline{B}\|_{X_2} 
    \end{align*}
    for all $(\nu,u) \in H_{2,\om}$ and all $\overline{\B}_\om \in \mathcal{B}_\om(H_2)$, where $(\nu,U) \bydef \og\left(\nu,u\right)$ and $\overline{B} \bydef \overline{\Gamma}\left(\overline{\B}_\om\right)$.
\end{lemma}
 Using the previous lemma, we define $\oB_{T,\om} \bydef \overline{\Gamma}^\dagger(\overline{B}_T) \in \mathcal{B}_\om(H_2)$  and we have
 \begin{align*}
    \oB_{T,\om} = \begin{pmatrix}
            b_{1,1}  & {b}_{1,2}^*\\
             {b}_{2,1} & \mathbb{B}_{2,2}
        \end{pmatrix},
 \end{align*}
 where $b_{1,2} = \gamma^\dagger(B_{1,2})$, $b_{2,1} = \gamma^\dagger(B_{2,1})$ and $\mathbb{B}_{2,2} = \Gamma^\dagger(B_{2,2}).$
 Finally, we define
  \begin{align}\label{def : def of B extra}
    \oB_T \bydef \begin{pmatrix}
     0&0\\
     0&\out
 \end{pmatrix} + \oB_{T,\om} : H_2 \to H_2
 \end{align}
 and $\oA_T \bydef \oL^{-1} \oB_{T} : H_2 \to H_1$, which approximates the inverse of $D\oF(x_0).$ 
Using the analysis derived in Section \ref{sec : computation of the bounds}, we have all the necessary results to apply Theorem 4.6 from \cite{unbounded_domain_cadiot} to \eqref{eq : f(u) = 0 on H1}.
\begin{theorem}\label{th : radii eigenvalue}
 Let $\sigma_0>0$ be defined in \eqref{eq : alpha assumption} and let $\mathbb{\Psi}_0 : \ell^2 \to \ell^2$ be the discrete convolution operator associated to $\Psi_0$ (cf. \eqref{def : discrete conv operator}). Moreover, let us define
    \begin{align*}
        C_{\lambda,T} \bydef \begin{cases}
            C_{2,T,c-\lambda} &\text{ if } T=0\\
            C_{2,T,c+\lambda} &\text{ if } T>0\\
        \end{cases}
    \end{align*}
    where $C_{2,T,c}$ is given in \eqref{eq : constant C2T}. Recall the sequence $E \in \ell^2_e$ defined in Lemma \ref{lem : lemma Zu}. Then, given $v \in H^4_{\om}(\R)$ and $\lambda \leq 0$, define $\mathcal{Z}_{u,\lambda}(v)$ as 
\begin{align}\label{def : Zulambda in lemma}
   \mathcal{Z}_{u,\lambda}(v) \bydef  \left(2dC_{\lambda,T}^2 {e^{-2ad}}\left(\gamma(v),E_{full}*\gamma(v)\right)_{\ell^2}\left( \frac{2}{a} + C_1(d)\right) + 8C_{\lambda,T}^2\ln(2)\int_{d-1}^d|v'|^2\right)^{\frac{1}{2}}.
\end{align}
    Moreover, let $\mathcal{Y}_u, \mathcal{Y}_0, \mathcal{Z}_u, {\mathcal{Z}}_1$ and ${\mathcal{Z}}_2$ be non-negative constants such that
  \begin{align}\label{eq: definition Y0 Z1 Z2 eigen}
  \nonumber
  \mathcal{Y}_{u} &\geq
       2d C_{\mathcal{Y}_0,T}  e^{-a_0d}\left(~\left( {\Lambda_{\nu}} \Psi_0,E_{0,full}*( {\Lambda_{\nu}} \Psi_0)\right)_{\ell^2} \left( 1 + (1+C(d))\|\overline{{B}}_T\|_{X_2}^2 \right)~\right)^{\frac{1}{2}}\\
    {\mathcal{Y}}_0 &\geq \sqrt{2d}  \left\|\overline{B}_T \overline{F}(\lambda_0, \Psi_0)\right\|_{X_2} + \mathcal{Y}_u + \frac{2r_{0}}{\sigma_0}\max\left\{\frac{\|B_{1,2}*\Psi_0\|_2}{\sqrt{2d}}, ~ \|B_{2,2}\|_2\|{\Psi}_0\|_1\right\}\\
       \mathcal{Z}_1 &\geq \left\|I_d - \overline{A}_T D\overline{F}({\lambda_0},\Psi_0)\right\|_{X_1} + \max\{2\mathcal{Z}_{u,\lambda_0}(u_0), \sqrt{2d}\mathcal{Z}_{u,\lambda_0}(\psi_0) \} \left\|\oB_T\right\|_{H_2}  +  \frac{r_{0} \left\|\oB_T\right\|_{H_2}}{4\sqrt{\nu}\sigma_0 (\sigma_0-\lambda_0)}\\\nonumber
    \mathcal{Z}_2 &\geq \left\|\oB_T\right\|_{H_2}
\end{align}  
If there exists $r>0$ such that
\begin{equation}\label{condition radii polynomial eig}
    \mathcal{Z}_2r^2 - (1-\mathcal{Z}_1)r + \mathcal{Y}_0 < 0,
 \end{equation}
then there exists an eigenpair $(\tilde{\lambda},\tilde{\psi})$ of $D\mathbb{F}(\tilde{u})$ in $ \overline{B_r(x_0)} \subset H_1$, where $x_0$ is defined in \eqref{def : x0}.   In particular,
if $1 - \mathcal{Z}_1 - \frac{r \|\overline{\mathbb{B}}_T\|_2}{(\sigma_0-\lambda_0)^2}>0$ and if $\lambda_0 + r < \lambda_{\max}$, where $\lambda_{\max}$ is defined in \eqref{def : lambda_max}, then defining $R>0$ as
\begin{align}
    R \bydef  \frac{(\sigma_0-\lambda_0)\left(1 - \mathcal{Z}_1 - \frac{r \|\overline{\mathbb{B}}_T\|_2}{(\sigma_0-\lambda_0)^2}\right)}{\|\overline{\mathbb{B}}_T\|_2},
\end{align}
we obtain that $\tilde{\lambda}$ is simple and it is the only eigenvalue of $D\mathbb{F}(\tilde{u})$ in $( \lambda_0-R,{\lambda}_0+R) \cap (\lambda_0-R,\lambda_{max})$.
\end{theorem}

\begin{proof}
    Our goal is to apply Theorem 4.6 from \cite{unbounded_domain_cadiot} to \eqref{eq : f(u) = 0 on H1}. Recall that $\oA_T \bydef \oL^{-1} \oB_T : H_2 \to H_1$, which is a bounded linear operator. In particular, we need to compute upper bounds $\mathcal{Y}_0$, $\mathcal{Z}_1$ and $\mathcal{Z}_2$ such that 
    \begin{align*}
        \|\oA \oF(x_0)\|_{H_1} &\leq  \mathcal{Y}_0\\
        \|I_d - \oA D\oF(x_0)\|_{H_1} &\leq  \mathcal{Z}_1\\
        \|\oA \left(D\oF(x_0) - D\oF(x)\right)\|_{H_1}& \leq  \mathcal{Z}_2 \|x_0-x\|_{H_1}
    \end{align*}
    for all $x \in H_1.$ First, using \eqref{eq : min of m}, notice that $\|u\|_2 \leq \frac{1}{\sigma_0}\|u\|_{\mathcal{H}}$ for all $u \in \mathcal{H}.$ Then, using Lemma 5.6 in \cite{unbounded_domain_cadiot} and $\oB_T = \oL \oA_T$, we have
    \begin{align*}
       \|\oA_T \oF(x_0)\|_{H_1}  &\leq \|\oA_T \oF_0(x_0)\|_{H_1} + \left\|\oA_T\begin{pmatrix}
       0\\
           D\mathbb{G}(u_0)\psi_0 - D\mathbb{G}(\tilde{u})\psi_0
       \end{pmatrix}\right\|_{H_1}\\
       &\leq \|\oB_T \oF_0(x_0)\|_{H_2} + 2\left\|\begin{pmatrix}
           0 & (b_{1,2}\psi_0)^*\\
           0 & \mathbb{B}_{2,2}\mathbb{\psi}_0 
       \end{pmatrix}\right\|_{H_2}
       \|u_0-\tilde{u}\|_2\\
       & \leq \|\oB_T \oF_0(x_0)\|_{H_2} + \frac{2r_{0}}{\sigma_0}\left\|\begin{pmatrix}
           0 & (b_{1,2}\psi_0)^*\\
           0 & \mathbb{B}_{2,2}\mathbb{\psi}_0 
       \end{pmatrix}\right\|_{H_2}
    \end{align*}
    where we used that $\|u\|_2 \leq \frac{1}{\sigma_0}\|\mathbb{L}u\|_2$ and where we abuse notation in the above by considering $\psi_0$ as its associated multiplication operator (cf. \eqref{eq : multiplication operator L2}).
    Now, recall that $b_{1,2} = \gamma^\dagger(B_{12})$ and $\psi_0 = \gamma^{\dagger}(\Psi_0)$. This implies that $B_{1,2}*\Psi_0 = \gamma(b_{12}\psi_0)$ and, using Lemma \ref{prop : operator link L^2 extra equation}, we get 
    \begin{align*}
        \left\|\begin{pmatrix}
           0 & (b_{1,2}\psi_0)^*\\
           0 & \mathbb{B}_{2,2}\mathbb{\psi}_0 
       \end{pmatrix}\right\|_{H_2} = \left\|\begin{pmatrix}
           0 & ({B}_{1,2}*\Psi_0)^*\\
           0 & B_{2,2}\mathbb{\Psi}_0 
       \end{pmatrix}\right\|_{X_2}.
    \end{align*}
    Now, given $U \in \ell^2$, we have 
    \begin{align*}
        |(B_{1,2}*\Psi_0,U)_2| \leq \|B_{1,2}*\Psi_0\|_2\|U\|_2
    \end{align*}
    by Cauchy-Schwarz inequality. Moreover,
    \begin{align}
        \|B_{2,2}\mathbb{\Psi}_0U\|_2 \leq \|B_{2,2}\|_2\|{\Psi}_0*U\|_2  \leq \|B_{2,2}\|_2\|{\Psi}_0\|_1\|U\|_2 
    \end{align}
    where we used \eqref{eq : youngs inequality} for the last step. This implies that 
    \begin{align*}
         \left\|\begin{pmatrix}
           0 & ({B}_{1,2}*\Psi_0)^*\\
           0 & B_{2,2}\mathbb{\Psi}_0 
       \end{pmatrix}\right\|_{X_2} \leq \max\left\{\frac{\|B_{1,2}*\Psi_0\|_2}{\sqrt{2d}}, ~ \|B_{2,2}\|_2\|{\Psi}_0\|_1\right\}.
    \end{align*}
    Then, notice that 
    \begin{align*}
        \oF_0(x_0) = \begin{pmatrix}
            0\\
            {\mathbb{L}} \psi_0 +  D{\mathbb{G}}(u_0) \psi_0 - \lambda_0 \psi_0
        \end{pmatrix} = \begin{pmatrix}
            0\\
            \Gamma^\dagger(\tilde{L}) \psi_0 +  D{\mathbb{G}}(u_0) \psi_0 - \lambda_0 \psi_0
        \end{pmatrix} + \begin{pmatrix}
            0\\
           {\mathbb{L}}\psi_0 - \Gamma^\dagger(\tilde{L}) \psi_0
        \end{pmatrix}.
    \end{align*}
    Now, observing that 
    \begin{align*}
        \left\| {\mathbb{L}}\psi_0 - \Gamma^\dagger(\tilde{L}) \psi_0\right\|_2 = \left\| \left({\mathbb{L}} {\mathbb{\Lambda}_{\nu}}^{-1} - \Gamma^\dagger(\tilde{L}) {\mathbb{\Lambda}_{\nu}}^{-1}\right)  {\mathbb{\Lambda}_{\nu}}\psi_0\right\|_2
    \end{align*}
    and using the proof of Lemma \ref{sec : Y0 bound}, we obtain 
    \begin{align*}
        \left\|\overline{\mathbb{B}_T} \begin{pmatrix}
            0\\
           {\mathbb{L}}\psi_0 - \Gamma^\dagger(\tilde{L}) \psi_0
        \end{pmatrix}\right\|_2 \leq \mathcal{Y}_u.
    \end{align*}
    Moreover, since $\psi_0 = \gamma^\dagger\left(\Psi_0\right) \in H^4_{\om}(\R)$, we obtain
    \[
    \Gamma^\dagger(\tilde{L}) \psi_0 +  D{\mathbb{G}}(u_0) \psi_0 - \lambda_0 \psi_0 = \gamma^\dagger\left(\tilde{L} \psi_0 D{\mathbb{G}}(U_0) \Psi_0 - \lambda_0 \Psi_0 \right).  
    \]
    Consequently, using Lemma \ref{lem : gamma and Gamma properties}, we get
    \begin{align*}
        \|\oB_T \oF_0(x_0)\|_{H_2} 
            \leq \|\overline{B} \overline{F}(\lambda_0, \Psi_0)\|_{X_2} + \mathcal{Y}_u.
    \end{align*}
    This proves that $\mathcal{Y}_0 \geq \|\oA_T \oF(x_0)\|_{H_1}$. Now, let $x = (\nu,u) \in H_1$, we have
    \begin{align*}
        \left\|\oA_T \left(D\oF(x_0) - D\oF(x)\right)\right\|_{H_1} = \left\|\oA_T \begin{pmatrix}
            0 &0\\
            u-\psi_0  & (\nu-\lambda_0)I_d 
        \end{pmatrix} \right\|_{H_1} &\leq \|\oA_T\|_{H_2,H_1} \|x-x_0\|_{H_1} \\
        &= \|\oB_T\|_{H_2} \|x-x_0\|_{H_1}.
    \end{align*}
This proves that $\mathcal{Z}_2 \|x-x_0\|_{H_1} \geq  \|\oA_T \left(D\oF(x_0) - D\oF(x)\right)\|_{H_2}$ for all $x \in H_1.$ Finally, we consider the bound $\mathcal{Z}_1$. Using Lemma 5.6 in \cite{unbounded_domain_cadiot}, we have
\begin{align*}
    \|I_d - \oA_T D\oF(x_0)\|_{H_1}  &\leq \|I_d - \oA_T D\oF_0(x_0)\|_{H_1} + \|\oB_T\|_{H_2} \left\|\left(D{\mathbb{G}}(u_0) - D{\mathbb{G}}(\tilde{u})\right)\mathbb{L}_{\lambda_0}^{-1}\right\|_{2} \\
    & \leq \|I_d - \oA_T D\oF_0(x_0)\|_{H_1} + \frac{1}{\sigma_0-\lambda_0}\|\oB_T\|_{H_2} \|u_0 - \tilde{u}\|_\infty\\
    &\leq \|I_d - \oA_T D\oF_0(x_0)\|_{H_1} +  \frac{r_{0}}{4\sqrt{\nu}\sigma_0 (\sigma_0-\lambda_0)}\|\oB_T\|_{H_2},
\end{align*}
where we used Proposition \ref{prop : equivalence infinite norm}  for the last step. Now, notice that
\begin{align*}
    D\overline{\mathbb{F}}_0(x_0) = \begin{pmatrix}
        0 & - \psi_0^*\\
        -\psi_0 & D\mathbb{F}(u_0) - \lambda_0 I_d
    \end{pmatrix}.
\end{align*}
Using a similar reasoning as the one used in Theorem 5.2 in \cite{unbounded_domain_cadiot}, we get
\begin{align*}
    \|I_d - \oA_T D\oF_0(x_0)\|_{H_1} \leq  \|I_d - \overline{A} D\overline{F}(\lambda_0,\Psi_0)\|_{X_1} + \left\|\begin{pmatrix}
        \left({\mathbb{L}}_{\lambda_0}^{-1}\psi_0 - \Gamma^\dagger\left({L}_{\lambda_0}^{-1}\right)\psi_0\right)^*\\
        2\left(\mathbb{L}_{\lambda_0}^{-1} - \Gamma^\dagger\left({L}_{\lambda_0}^{-1}\right)\right)\mathbb{u}_0
    \end{pmatrix}\right\|_2 \|\oB_T\|_{H_2} .
\end{align*}
Then, focusing on the second term of the right-hand side of the above inequality we get
\begin{align*}
    \left\|\begin{pmatrix}
        \left({\mathbb{L}}_{\lambda_0}^{-1}\psi_0 - \Gamma^\dagger\left({L}_{\lambda_0}^{-1}\right)\psi_0\right)^*\\
        2\left(\mathbb{L}_{\lambda_0}^{-1} - \Gamma^\dagger\left({L}_{\lambda_0}^{-1}\right)\right)\mathbb{u}_0
    \end{pmatrix}\right\|_2 \leq \max\left\{\|{\mathbb{L}}_{\lambda_0}^{-1}\psi_0 - \Gamma^\dagger\left({L}_{\lambda_0}^{-1}\right)\psi_0\|_2, ~ 2\|\left(\mathbb{L}_{\lambda_0}^{-1} - \Gamma^\dagger\left({L}_{\lambda_0}^{-1}\right)\right)\mathbb{u}_0\|_2\right\}.
\end{align*}
But now, notice that $2\left\|\left({\mathbb{L}}_{\lambda_0}^{-1} - \Gamma^\dagger\left({L}_{\lambda_0}^{-1}\right)\right)\mathbb{u}_0\right\|_2$ has already been investigated in Lemmas \ref{lem : computation of Z1} and \ref{lem : lemma Zu} in the case $\lambda_0 = 0$ and corresponds to the quantity $\left(\left(\mathcal{Z}_{u,1}^{(2)}\right)^2 + \left(\mathcal{Z}_{u,2}^{(2)}\right)^2\right)^{\frac{1}{2}}$. If $T>0$, we also have $\mathbb{L}_{\lambda_0} = \mathbb{M}_T - c I_d - \lambda_0 I_d = \mathbb{M}_T - (c+\lambda_0) I_d$. Moreover, since $\lambda_0 \leq 0$, $ \xi \to \frac{1}{m_T(\xi) - (c+\lambda_0)}$ has a bigger domain of analyticity than $\frac{1}{l}$. Consequently, Lemmas \ref{prop : analyticity and value of a},  \ref{lem : computation of Z1} and \ref{lem : lemma Zu} are applicable if we replace $c$ by $c+\lambda_0$. The same reasoning applies in the case $T=0$ when replacing $c$ by $c-\lambda_0$.\\
 Recalling that $v_2 = 2u_0$ by definition, we obtain that
\[
2\left\|\left({\mathbb{L}}^{-1}_{\lambda_0} - \Gamma^\dagger\left({L}^{-1}_{\lambda_0}\right)\right)\mathbb{u}_0\right\|_2 \leq 2\mathcal{Z}_{u,\lambda_0}(u_0).
\]
Now, we study the term $\|{\mathbb{L}}_{\lambda_0}^{-1}\psi_0 - \Gamma^\dagger\left({L}_{\lambda_0}^{-1}\right)\psi_0\|_2$. Denoting $u = \cha$, notice that we have 
\begin{align*}
    \|{\mathbb{L}}_{\lambda_0}^{-1}\psi_0 - \Gamma^\dagger\left({L}_{\lambda_0}^{-1}\right)\psi_0\|_2 = \|\left({\mathbb{L}}_{\lambda_0}^{-1} - \Gamma^\dagger\left({L}_{\lambda_0}^{-1}\right)\right)\psi_0 u \|_2 &\leq \|\left({\mathbb{L}}_{\lambda_0}^{-1} - \Gamma^\dagger\left({L}_{\lambda_0}^{-1}\right)\right)\psi_{0,op}\|_2 \|\cha\|_2\\
    &=  \sqrt{2d}\|\left({\mathbb{L}}_{\lambda_0}^{-1} - \Gamma^\dagger\left({L}_{\lambda_0}^{-1}\right)\right)\psi_{0,op}\|_2 
\end{align*}
where $\psi_{0,op}$ is the multiplication operator associated to $\psi_0.$ Therefore, using a similar reasoning as for the term $\left\|\left({\mathbb{L}}^{-1}_{\lambda_0} - \Gamma^\dagger\left({L}^{-1}_{\lambda_0}\right)\right)\mathbb{u}_0\right\|_2$,  this implies that 
\begin{align*}
    \|{\mathbb{L}}_{\lambda_0}^{-1}\psi_0 - \Gamma^\dagger\left({L}_{\lambda_0}^{-1}\right)\psi_0\|_2  \leq \sqrt{2d} \mathcal{Z}_{u,\lambda_0}(\psi_0).
\end{align*}
Consequently, if \eqref{condition radii polynomial eig} is satisfied for some $r>0$, then Theorem 4.6 from \cite{unbounded_domain_cadiot} is applicable to \eqref{eq : f(u) = 0 on H1}.
 In particular, there exists an eigenpair $(\tilde{\lambda},\tilde{\psi})$ of $D\mathbb{F}(\tilde{u})$ in $ \overline{B_r(x_0)} \subset H_1$.

 Now, assume that $1 - \mathcal{Z}_1 - \frac{r \|\overline{\mathbb{B}}_T\|_2}{(\sigma_0-\lambda_0)^2}>0$ and $\lambda_0 +r < \lambda_{\max}$. In particular, this implies that 
 \[
\tilde{\lambda} \leq \lambda_0 + r < \lambda_{\max}
 \]
 since $(\tilde{\lambda},\tilde{\psi}) \in \overline{B_r(x_0)} \subset H_1$.
 We prove that $\tilde{\lambda}$ is simple and that it is the only eigenvalue of $D\mathbb{F}(\tilde{u})$ in $( \lambda_0-R,{\lambda}_0+R) \cap (\lambda_0-R,\lambda_{\max})$. Using Lemma \ref{lem : spectrum is eigenvalues}, we know that  spectrum of $D\mathbb{F}(\tilde{u})$ below $\lambda_{\max}$ consists of eigenvalues of finite-multiplicity. Moreover, since  $D\mathbb{F}(\tilde{u})$ is self-adjoint in $L^2$, if $\tilde{\lambda}$ is not simple or if there exists another eigenvalue in $( \lambda_0-R,{\lambda}_0+R) \cap (\lambda_0-R,\lambda_{\max})$, then it implies that there exists $\mu \in ( \lambda_0-R,{\lambda}_0+R) \cap (\lambda_0-R,\lambda_{\max})$ and $w \in \mathcal{H}_{\lambda_0}, w \neq 0$ such that 
 \[
D\mathbb{F}(\tilde{u}) w = \mu w.
 \]
 In particular, 
 \begin{equation}\label{eq : orthogonality cdtt}
     (w,\tilde{\psi})_2 = 0
 \end{equation}
 since $D\mathbb{F}(\tilde{u})$ is self-adjoint in $L^2$.
 Let $\mathcal{A}$ be given by
 \begin{align}
     \mathcal{A} \bydef \begin{pmatrix}
         0 & -\tilde{\psi}^*\\
        -\tilde{\psi} &  D\mathbb{F}(\tilde{u}) - {\lambda_0}I_d
     \end{pmatrix}.
 \end{align}
 Then,
 \begin{align*}
    \left\| \mathcal{A}\begin{pmatrix}
         \nu\\
         u
     \end{pmatrix} - D\overline{\mathbb{F}}(\tilde{\lambda},\tilde{\psi})\begin{pmatrix}
         \nu\\
         u
     \end{pmatrix}\right\|_{H_2} = \left\|\begin{pmatrix}
         (u,\psi_0-\tilde{\psi})_2\\
        (\tilde{\lambda}-\lambda_0)u
     \end{pmatrix}\right\|_{H_2} \leq \|u\|_2 \max\left\{\|\psi_0 - \tilde{\psi}\|_2, ~ |\lambda_0-\tilde{\lambda}|\right\}
 \end{align*}
for all $(\nu,u) \in H_1$. Now, using that $\|u\|_2 \leq \frac{1}{\sigma_0 - \lambda_0}\|u\|_{\mathcal{H}_{\lambda_0}} \leq \frac{1}{\sigma_0 - \lambda_0}\|(\nu,u)\|_{H_1}$, we obtain that
\begin{align*}
     \left\| \mathcal{A} - D\overline{\mathbb{F}}(\tilde{\lambda},\tilde{\psi})\right\|_{H_1,H_2} \leq \frac{r}{(\sigma_0-\lambda_0)^2}
\end{align*}
since $\|(\lambda_0,\psi_0)-(\tilde{\lambda},\tilde{\psi})\|_{H_1} \leq r$. In particular, using a Neumann series argument on $I_d - \overline{\mathbb{A}} \mathcal{A}$, since $\frac{r\|\overline{\mathbb{B}}_T\|_2}{(\sigma_0-\lambda_0)^2} + \mathcal{Z}_1 <1$, we obtain that $\mathcal{A} : H_1 \to H_2$ has a bounded inverse and
\begin{align}
    \|\mathcal{A}^{-1}\|_{H_2,H_1}\leq \frac{\|\overline{\mathbb{B}}_T\|_2}{1 - \mathcal{Z}_1 - \frac{r \|\overline{\mathbb{B}}_T\|_2}{(\sigma_0-\lambda_0)^2}} \bydef \frac{\sigma_0-\lambda_0}{R}.
\end{align}
 Now, notice that this implies that 
 \begin{align}
     \left\|\mathcal{A}\begin{pmatrix}
         \nu\\
         u
     \end{pmatrix}\right\|_{H_2} \geq \frac{R}{\sigma_0-\lambda_0} \|(\nu,u)\|_{H_1}
 \end{align}
 for all $(\nu,u) \in H_1.$ In particular, using \eqref{eq : orthogonality cdtt} and the above inequality, we get
 \begin{align*}
     \left\|\mathcal{A}\begin{pmatrix}
         0\\
         w
     \end{pmatrix}\right\|_{H_2} = \left\|\begin{pmatrix}
         0\\
         (\mu-{\lambda}_0)w
     \end{pmatrix}\right\|_{H_2} = |\mu-{\lambda}_0| \|w\|_2  \geq  \frac{R}{\sigma_0-\lambda_0} \|(0,w)\|_{H_1} \geq R \|(0,w)\|_{H_2}. 
 \end{align*}
 This implies that $\tilde{\lambda}$ is simple and it is the unique eigenvalue of $D\mathbb{F}(\tilde{u})$ in $( \lambda_0-R,{\lambda}_0+R) \cap (\lambda_0-R,\lambda_{max}).$
\end{proof}

\begin{remark}\label{rem : eigenvalue zero}
    It is well-known that zero is an eigenvalue of $D\mathbb{F}(\tilde{u})$ associated to the eigenvector $\tilde{u}'$. Now, if using Theorem \ref{th : radii eigenvalue}, one can prove that there exists a simple eigenvalue $\tilde{\lambda}_0$ of $D\mathbb{F}(\tilde{u})$ such that $\tilde{\lambda}_0$ is the only eigenvalue  in $(-R_0,0]$ for some $R_0>0$, then this implies that $\tilde{\lambda}_0 = 0$ by uniqueness. Therefore, Theorem \ref{th : radii eigenvalue} allows to prove that zero is a simple eigenvalue of $D\mathbb{F}(\tilde{u})$.
\end{remark}

In practice, Theorem \ref{th : radii eigenvalue} allows to verify that (P1) and (P3) from Lemma \ref{lem : result on eigenvalues} are satisfied. In other terms, we can prove that $D\mathbb{F}(\tilde{u})$ has a simple negative eigenvalue $\lambda^{-}$ and that zero is a simple eigenvalue of $D\mathbb{F}(\tilde{u})$. 

Consequently, we now assume that (using Theorem \ref{th : radii eigenvalue}) we were able to prove that there exists $\lambda_0^- <0$ and $R_1>0$ such that 
\begin{align*}
    \lambda^{-} \in (\lambda_0^- - R_1, \lambda_0^- + R_1)
\end{align*}
and $\lambda^{-}$ is the only eigenvalue of $D\mathbb{F}(\tilde{u})$ in $(\lambda_0^- - R_1, \lambda_0^- + R_1)$. Moreover, $\lambda^-$ is simple. Similarly, using Remark \ref{rem : eigenvalue zero}, suppose that we were able to prove the existence of $R_0 >0$ such that zero is the only eigenvalue of $D\mathbb{F}(\tilde{u})$ in $(-R_0,0]$. 

Our goal is to set-up a strategy to prove that $D\mathbb{F}(\tilde{u})$ has no negative eigenvalue other than $\lambda^{-}$, that is, prove the remaining condition (P2) of Lemma \ref{lem : result on eigenvalues}. 
We first state Lemma \ref{lem : minimal value eig}, which provides a lower bound for the spectrum of $D\mathbb{F}(\tilde{u}).$

 \begin{lemma}\label{lem : minimal value eig}
    Let $\sigma_0$ be given in \eqref{eq : alpha assumption} and let $\lambda_{min}$ be defined as
     \begin{align}\label{def : lambda_min}
         \lambda_{min} \bydef \sigma_0  - 2\|u_0\|_1 - \frac{r_{0}}{4\sqrt{\nu}\sigma_0}.
     \end{align}
If $D\mathbb{F}(\tilde{u}) - \lambda I_d$ is injective for all $\lambda \in [\lambda_{min}, \lambda_0^--R_1]\bigcup [\lambda_0^-+R_1,-R_{0}]$, then $\lambda^{-}$ is the only negative eigenvalue of $D\mathbb{F}(\tilde{u})$.
 \end{lemma}
 \begin{proof}
     Suppose that there exists $\lambda<0$ and $u \in L^2$ such that  $D\mathbb{F}(\tilde{u})u = \lambda u$, then 
\begin{align*}
    \lambda \|u\|_2 =(D\mathbb{F}(\tilde{u})u,u)_2 &\geq \min_{\xi \in \R}|m_{T}(\xi)-c| \|u\|_2 - 2\|\tilde{u}u\|_2 \\
    &\geq  \sigma_0 \|u\|_2 - \left(2\|u_0\|_\infty + \frac{r_{0}}{4\sqrt{\nu}\min_{\xi \in \R}|m_{T}(\xi)-c|}\right)\|u\|_2\\
    &\geq  \sigma_0 \|u\|_2 - \left(2\|U_0\|_1 + \frac{r_{0}}{4\sqrt{\nu}\sigma_0}\right)\|u\|_2
\end{align*}
using  Proposition \ref{prop : equivalence infinite norm}. 
Therefore, $\lambda \geq \displaystyle \lambda_{min}$. Moreover, we conclude the proof using that $0$ and $\lambda^-$ are the only eigenvalues in $(-R_0,0]$ and $(\lambda_0^--R_1,\lambda_0^-+R_1)$ by assumption. 
 \end{proof}

Using the previous lemma, we need to prove that $D\mathbb{F}(\tilde{u}) - \lambda I_d$ is injective for all $\lambda \in [\lambda_{min}, \lambda_0^--R_1]\bigcup [\lambda_0^-+R_1,-R_{0}]$, where $\lambda_{min}$ is given in \eqref{def : lambda_min}. Now, the next lemma provides the injectivity of $D\mathbb{F}(\tilde{u}) - \lambda I_d$ for a range of values of $\lambda$, provided that $D\mathbb{F}(\tilde{u}) - \lambda^* I_d$ is injective for some fixed $\lambda^*.$

\begin{lemma}\label{lem : neumann series}
    Let $\lambda^*  \leq 0$ and suppose that there exists $\mathcal{C}>0$ such that  $\|D\mathbb{F}(\tilde{u})u - \lambda^* u\|_{2}  \geq \mathcal{C}\|u\|_2$ for all $u \in \mathcal{H}$. Then $D\mathbb{F}(\tilde{u}) - \lambda I_d$  is injective for all $\lambda \in (\lambda^*-\mathcal{C},\lambda^*+\mathcal{C})$.
\end{lemma}
\begin{proof}
    Let $u \in \mathcal{H}$, then
    \begin{align*}
        \|D\mathbb{F}(\tilde{u})u - \lambda u\|_2 \geq \|D\mathbb{F}(\tilde{u})u - \lambda^* u\|_2 - |\lambda-\lambda^*|\|u\|_2 \geq \left(\mathcal{C} - |\lambda-\lambda^*|\right)\|u\|_2.
    \end{align*}
    This proves the lemma. 
\end{proof}

Using the previous lemma combined with Lemma \ref{lem : minimal value eig}, we can prove the injectivity of $D\mathbb{F}(\tilde{u})-\lambda I_d$ for all $\lambda \in [\lambda_{min}, \lambda_0^--R_1]\bigcup [\lambda_0^-+R_1,-R_{0}]$ by proving that $D\mathbb{F}(\tilde{u})-\lambda I_d$ is invertible for a finite number of values $\lambda^*$. Moreover, having access to an upper bound for the norm of the inverse of $D\mathbb{F}(\tilde{u})-\lambda^* I_d$, we obtain a value for $\mathcal{C}$ is the previous lemma. To do so, we want to use Lemma \ref{lem : computation of Z1}. Indeed,  Lemma \ref{lem : computation of Z1} provides the invertibility of $D\mathbb{F}(\tilde{u})-\lambda^* I_d$, by constructing an approximate inverse, as well as an upper bound for the norm of the inverse. 

\begin{lemma}\label{lem : Z bound for only eig}
    Let $\lambda \leq 0$, then recall that
    \begin{align*}
        C_{\lambda,T} \bydef \begin{cases}
            C_{2,T,c-\lambda} &\text{ if } T=0\\
            C_{2,T,c+\lambda} &\text{ if } T>0\\
        \end{cases}
    \end{align*}
    where $C_{2,T,c}$ is given in \eqref{eq : constant C2T}.
Now, recalling $E$ defined in Lemma \ref{lem : lemma Zu}, we let $\mathcal{Z}_{\lambda,u}>0$ be a bound satisfying
    \begin{align*}
     (\mathcal{Z}_{\lambda,u})^2 &\geq  4C_{\lambda,T}^2 {|\om|e^{-2ad}}\left(U_0,E*U_0\right)_2\left( \frac{2}{a} + C_1(d)\right) + 32C_{\lambda,T}^2\ln(2)\int_{d-1}^d|u_0'|^2.
\end{align*}
  Then,
    \begin{align}\label{eq : definition of Zlambda}
    2\left\|\left(\Gamma^\dagger\left(L_\lambda^{-1}\right) - \mathbb{L}_\lambda^{-1}\right)\mathbb{u}_0\right\|_2 \leq \mathcal{Z}_{\lambda,u},
    \end{align}
    where $\mathbb{L}_\lambda$ and $L_\lambda$ are defined in \eqref{def : L lambda unbounded} and \eqref{def : L lambda periodic} respectively. Now let $\mathbb{B}_T : L^2 \to L^2$ be a bounded linear operator defined as 
\begin{align*}
    \mathbb{B}_T \bydef \Gamma^\dagger\left(B_T^N + \pi_N \mathbb{W}_T\right),
\end{align*}
where $B^N_T : \ell^2 \to \ell^2$ is a bounded linear operator such that  $B^N_T = \pi^N B^N_T \pi^N$ and where $W_T \in \ell^1$ such that $W_T = \pi^N W_T$. In particular, we choose $\mathbb{W}_T = I_d$ if $T>0$.  Moreover, define the following bounded linear operators \[\mathbb{A}_T \bydef \mathbb{L}_\lambda^{-1} \mathbb{B}_T : L^2 \to \mathcal{H}_\lambda ~~ \text{ and } ~~ A_T \bydef {L}_\lambda^{-1}(B^N_T + \pi_N \mathbb{W}_T) : \ell^2 \to \mathscr{h}_\lambda.\]
Then, \begin{align*}
    \|D\mathbb{F}(\tilde{u})u - \lambda u\|_2 \geq \left(\frac{1- \|\mathbb{B}_T\|_2\mathcal{Z}_{\lambda,u}-Z_{\lambda,1}}{\|\mathbb{B}_T\|_2} (\sigma_0 - \lambda) -  \frac{r_{0}}{4\sqrt{\nu}\sigma_0}\right)\|u\|_2
\end{align*}
for all $u \in \mathcal{H}_\lambda$ where  $Z_{\lambda,1}$ satisfies
\begin{align*}
    Z_{\lambda,1} &\geq \left\|I_d - A_T \left(D{F}(U_0)-\lambda I_d\right) \right\|_{\mathscr{h}_\lambda}.
\end{align*}
\end{lemma}

\begin{proof}
Let $\lambda \leq 0$, recall that $\mathbb{L}_\lambda$ is invertible and we can define the Hilbert space $\mathcal{H}_\lambda$ associated to its norm $\|u\|_{\mathcal{H}} = \|\mathbb{L}_\lambda u \|_2$ for all $u \in \mathcal{H}_\lambda$. 
    First, using Proposition \ref{prop : equivalence infinite norm} and \eqref{eq : min of m}, notice that 
    \begin{align}\label{eq : step 1 estimation norm}
        \|D\mathbb{F}(\tilde{u})u - \lambda u\|_2 &\geq   \|D\mathbb{F}(u_0)u - \lambda u\|_2 - 2\|(\tilde{u}-u_0)u\|_2\\
       &\geq \|D\mathbb{F}(u_0)u - \lambda u\|_2-  \frac{r_{0}}{4\sqrt{\nu}\sigma_0}\|u\|_2
    \end{align}
    for all $u \in L^2.$ Let $u \in \mathcal{H}_\lambda$, then 
    \begin{align*}
        \|u\|_{l} &= \left\|u - \mathbb{A}_T \left(D\mathbb{F}(u_0)-\lambda I_d \right)u + \mathbb{A}_T\left(D\mathbb{F}(u_0)-\lambda I_d \right)u\right\|_2\\
        &\leq \left\|I_d - \mathbb{A}_T \left(D\mathbb{F}(u_0)-\lambda I_d \right)\right\|_{\mathcal{H}} \|u\|_{\mathcal{H}}  + \|\mathbb{A}_T\|_{2,l} \left\|\left(D\mathbb{F}(u_0)-\lambda I_d \right)u\right\|_2.
    \end{align*}
    In particular, this implies that 
    \begin{align*}
        \left\|\left(D\mathbb{F}(u_0)-\lambda I_d \right)u\right\|_2 &\geq \frac{1-\left\|I_d - \mathbb{A}_T \left(D\mathbb{F}(u_0)-\lambda I_d \right)\right\|_{\mathcal{H}} }{\|\mathbb{A}_T\|_{2,l}}\|u\|_{\mathcal{H}}\\
        &\geq \frac{1- \|\mathbb{B}\|_2 \|\left(\Gamma^\dagger(L_\lambda^{-1}) - \mathbb{L}_\lambda^{-1}\right)\mathbb{u}_0\|_2 - \left\|I_d - A_T\left(D\tilde{F}(U_0)-\lambda I_d\right)\right\|_{\mathcal{H}}}{\|\mathbb{B}_T\|_2} \|u\|_{\mathcal{H}}
    \end{align*}
    using Lemma \ref{lem : computation of Z1}. Now, using \eqref{eq : min of m}, notice that
    \begin{align*}
        \|u\|_{\mathcal{H}_\lambda} \geq (\sigma_0-\lambda)\|u\|_2
    \end{align*}
    for all $u \in \mathcal{H}_\lambda.$ Finally, the proof of \eqref{eq : definition of Zlambda} is a direct consequence of Lemma \ref{lem : lemma Zu} where $c$ is replaced by $c-\lambda$ if $T=0$ and by $c+\lambda$ if $T>0$.  
\end{proof}

\begin{remark}
    In practice, the bound $Z_{\lambda,1} $ introduced in the previous lemma can be computed in the same manner as the bound $Z_1$ in Lemma \ref{lem : usual term periodic Z1}. We expose its numerical computation at \cite{julia_cadiot}.
\end{remark}

Now that we can compute an upper bound for the norm of the inverse of $D\mathbb{F}(\tilde{u})-\lambda^* I_d$ (for a given $\lambda^*$), we possess all the computer-assisted tools to control the negative spectrum of $D\mathbb{F}(\tilde{u})$. This control allows, potentially, to conclude about the spectral stability of a proven solitary wave.

\subsection{Computer-assisted proof of stability}

Combining the results of the previous section with Lemma \ref{lem : result on eigenvalues} and  {Proposition \ref{prop:stable}}, we obtain a computer-assisted method to prove the spectral stability of solutions to \eqref{eq : original Whitham}. We apply this approach to the obtained solutions $\tilde{u}_1, \tilde{u}_2, \tilde{u}_3$ and $\tilde{u}_4$ (cf. Theorems \ref{th : proof whitham} and \ref{th: : proof capillary whitham}) to prove their stability. The numerical details are available at \cite{julia_cadiot}.

  \begin{theorem}\label{th : existence eigenvalues}
     Let $\tilde{u}_i$ ($i \in \{1,2,3,4\}$) be a solution to \eqref{eq : whitham stationary} obtained in either Theorem \ref{th : proof whitham} or Theorem \ref{th: : proof capillary whitham}. Then, $D\mathbb{F}(\tilde{u}_i)$ has a simple negative eigenvalue $\lambda^{-}_i$, which is its only negative eigenvalue. Moreover, $D\mathbb{F}(\tilde{u}_i)$ has a zero eigenvalue which is also simple. 
 \end{theorem}

 \begin{proof}
     The proof of the simple eigenvalue $\lambda^{-}_i$ is obtained thanks to Theorem \ref{th : radii eigenvalue}. Similarly, the simplicity of the zero-eigenvalue is obtained thanks to Remark \ref{rem : eigenvalue zero}. 

      The rest of the proof is obtained via rigorous numerics in \cite{julia_cadiot}. Indeed, combining Lemmas \ref{lem : minimal value eig}, \ref{lem : neumann series} and \ref{lem : Z bound for only eig}, we can prove that $D\mathbb{F}(\tilde{u}_i) - \lambda I_d$ is injective for all $\lambda \in [\lambda_{min}, \lambda_i-R_i]\bigcup [\lambda_i+R_i,-R_{i,0}]$ by rigorously computing a constant $\mathcal{C}$ for a finite number of $\lambda^*\leq 0$ (using the notations of Lemma \ref{lem : neumann series}).
 \end{proof}

  {Now, for each $\tilde{u}_i$, we prove that condition \eqref{eq : condition for stability} is satisfied, that is the Vakhitov-Kolokolov quantity is negative. This is achieved rigorously on the computer thanks to  Proposition\ref{prop:stable}. Then, using Lemma \ref{lem : result on eigenvalues}, we obtain the spectral stability of each $\tilde{u}_i$.} 
\begin{theorem}\label{th : stability of solution}
     The solutions $\tilde{u}_1, \tilde{u}_2, \tilde{u}_3$ and $\tilde{u}_4$ obtained in Theorems \ref{th : proof whitham} and \ref{th: : proof capillary whitham} are spectrally stable. 
 \end{theorem}

\begin{remark}
    If one could prove the well-posedness of initial value problems  {in an energy space (that is in a Sobolev space matching the regularity of the Hamiltonian)}  with initial data in a neighborhood of $\tilde{u}_i$ ($i \in \{1,2\}$), then one could conclude about orbital stability (see discussions in \cite{Ehrnstrom_existence_stability_solitary} or \cite{Stefanov2018SmallAT} for instance). 
\end{remark}

\section{Conclusion}\label{conclusion}
In this article we presented a new computer-assisted method to study solitary waves in the Whitham and capillary-gravity Whitham equations. Moreover, the approach is general enough to handle proofs of existence of solitary waves as well as eigenvalue problems. In particular, we were able to prove constructively, with high accuracy, the existence of a solitary wave and its spectral stability in both the cases $T=0$ and $T>0$.

Similarly as what is presented in \cite{unbounded_domain_cadiot}, the method established in this paper can be generalized to a large class of nonlocal equations defined on $\R^n$ ($n \in \mathbb{N}$). Indeed, we can consider a nonlocal equation of the form
\begin{equation}\label{zero finding conclusion}
   \mathbb{L}u + \mathbb{G}(u) =  f 
\end{equation}
where $f : \R^n \to \R$ is a function in $L^2(\R^n)$. Then, $\mathbb{L}$ has to be a Fourier multiplier operator associated to its symbol $l : \R^n \to \mathbb{C}$. In particular, we require that there exists $\sigma_0>0$ such that
\begin{align*}
    |l(\xi)| > \sigma_0 >0
\end{align*}
for  all $\xi \in \R^n$ and that there exists $a>0$ such that $\frac{1}{l}$ is analytic on $S^n$ where $S \bydef \{z \in \mathbb{C}, ~ |Im(z)| < a\}$. This assumption allows one to define $\mathbb{L}^{-1}$ as a bounded linear operator, as well as the Hilbert space $\mathcal{H}$. The analyticity of $\frac{1}{l}$ allows one to derive an exponential decay for its inverse Fourier transform (as illustrated in Lemma \ref{lem : computation of f}). Moreover, using the notations of \cite{unbounded_domain_cadiot}, 
the non-linear operator $\mathbb{G}$ has to be of the form 
\begin{equation}\label{def: G and j}
     \mathbb{G}(u) \bydef \displaystyle\sum_{i = 2}^{N_{\mathbb{G}}}\mathbb{G}_i(u) 
\end{equation}
where $N_\mathbb{G} \in \mathbb{N}$ and  each $\mathbb{G}_i(u)$  can be decomposed as follows
\[
\mathbb{G}_i(u) \bydef \sum_{k \in J_i} (\mathbb{G}^1_{i,k}u) \cdots (\mathbb{G}^i_{i,k}u)
\]
where $J_i \subset \mathbb{N}$ is a finite set of indices, and where $\mathbb{G}^p_{i,k}$ is a Fourier multiplier operator for all $1\leq p \leq i$ and $k \in J_i$. In the case of the  {cgWE }, $\mathbb{G}(u) = \mathbb{G}_2(u) = (\mathbb{G}^1_{2,1}u)(\mathbb{G}^2_{2,1}u)$, where $\mathbb{G}^1_{2,1} = \mathbb{G}^2_{2,1} = I_d$.  In particular, each  Fourier multiplier operator $ \mathbb{G}^p_{i,k}$ has a symbol that we denote $g^p_{i,k}: \R^n \to \mathbb{C}.$  Then, we require each $g^p_{i,k}$ to be analytic on $S^n$ and
\begin{align*}
    \frac{|g^p_{i,k}(\xi)|}{|l(\xi)|} \to C_{i,k,p}
\end{align*}
as $|\xi| \to \infty$, where $C_{i,k,p}$ is a non-negative constant.  {Note that this set-up has recently been investigated in \cite{marstrander} on the real line ($n=1$) and existence proofs of solitary waves were obtained. One could use the above set-up to study existence of solutions in higher dimensional problems, such as the Kadomtsev-Petviashvili equation (cf. \cite{Ehrnström_2018}).}

Under these assumptions, the  analysis presented in Sections \ref{sec : computer-assisted approach}, \ref{sec : computation of the bounds} and \ref{sec : stability} is applicable to \eqref{zero finding conclusion}. In the case to case scenario, one has to compute the exponential decay associated to each $\frac{g^p_{i,k}}{l}$ (cf. Lemma \ref{lem : computation of f}) and the rest of the analysis of the present paper can easily be re-used. Note that if $C_{i,k,p} =0$ for each $1\leq i\leq N_{\mathbb{G}}$, $1\leq p \leq i$ and each $k \in J_i$, then the required analysis for the construction of the approximate inverse $\mathbb{A}$ is similar to the one required for a  semi-linear PDE. This is illustrated by the case $T>0$ in the  {cgWE }. However, if there exists a constant $C_{i,k,p} \neq 0$, then one has to follow the analysis derived in Section \ref{subsec : AT in the case T=0}. In particular, assumptions on the approximate solution might be needed (cf. Assumption \ref{ass : u0 is smaller than} for instance). This point is illustrated in   the case $T=0$ in this paper.

\section{Appendix}\label{sec : appendix}

\subsection{Proof of Lemma \ref{lem : computation of f}}\label{sec : proof of lemma 4.1 in appendix}

We present in this section the proof of Lemma \ref{lem : computation of f}. In particular, we provide the explicit computations of the constants defined in \eqref{eq : constant C0T}. First recall that we need to study the following functions
 \begin{align*}
 f_{\mathcal{Y}_0,T}& \bydef \mathcal{F}^{-1}\left(\frac{m_T(2\pi \cdot)}{l_\nu(2\pi \cdot)}\right)\\ 
      f_{0,T}& \bydef \mathcal{F}^{-1}\left(\frac{1}{l(2\pi \cdot)l_\nu(2\pi \cdot)}\right)\\ 
    f_{1,T}& \bydef \mathcal{F}^{-1}\left(\frac{2\pi \cdot}{l(2\pi \cdot)l_\nu(2\pi \cdot)}\right)\\ 
    f_{2,T}& \bydef \mathcal{F}^{-1}\left(\frac{1}{l(2\pi \cdot)}\right).
 \end{align*}
Then, using Proposition \ref{prop : analyticity and value of a}, we know that there exists  $0<a < \min\{\frac{1}{\sqrt{\nu}}, \frac{\pi}{2}\}$  such that $|m_T(z)-c| >0$ for all $z \in S$ where $S \bydef \{z \in \mathbb{C},~ |\text{Im}(z)| \leq a\}$. Moreover, there exists $\sigma_0, \sigma_1 >0$ such that
\begin{align*}
    |l(\xi)|, |l(\xi+ia)| &\geq \sigma_0 \text{ for all } \xi \in \R,\\
   |l(\xi+ia)| &\geq \sigma_1\sqrt{T|\xi|} \text{ for all } |\xi| \geq 1.
 \end{align*}
 In particular, $m_T, \frac{1}{l_\nu}$ and $\frac{1}{l}$ are analytic on $S$. Having these results in mind, we present the proof of the lemma.
\begin{proof}
Suppose that $T>0$, then using that $m_T(\xi) = m_0(\xi)\sqrt{1+T\xi^2}$ and that $\nu = T$ (cf. \eqref{def : nu}), we have
\begin{align*}
    f_{\mathcal{Y}_0,T} &=  \mathcal{F}^{-1}\left(\frac{m_T(2\pi \cdot)}{l_\nu(2\pi \cdot)}\right)\\
    &=   \mathcal{F}^{-1}\left(m_0(2\pi \cdot)\right)*\mathcal{F}^{-1}\left(\frac{1}{\sqrt{1+\nu(2\pi \cdot)^2}}\right).
\end{align*}
Let $x > 0$, then using \cite{EHRNSTROM2019on_whitham_conjecture}, we have
\begin{align*}
    g_0(x) \bydef \mathcal{F}^{-1}\left(m_0(2\pi \cdot)\right)(x) = \frac{1}{\pi} \sum_{n=1}^\infty \int_{n\pi - \frac{\pi}{2}}^{n\pi}e^{-xs}\sqrt{\frac{|\tan(s)|}{s}}ds.
\end{align*}
Moreover, using \cite{watson1995treatise} we have 
\begin{align*}
    g_1(x) \bydef \mathcal{F}^{-1}\left(\frac{1}{\sqrt{1+T(2\pi \cdot)^2}}\right)(x) = \frac{1}{\pi \sqrt{T}}K_0\left(\frac{x}{\sqrt{T}}\right) = \frac{1}{\pi \sqrt{T}}\int_{1}^\infty \frac{e^{-\frac{x}{\sqrt{T}}s}}{\sqrt{s^2-1}}ds
\end{align*}
where $K_0$ is the modified Bessel function of the second kind.

Now, given $n \in \mathbb{N}$, notice that $s \to \sqrt{\frac{|\tan(s)|(s-n\pi + \frac{\pi}{2})}{s}}$ is decreasing on $[n\pi - \frac{\pi}{2}, n\pi]$. In particular, 
\begin{equation}\label{ineq : tan function}
    \sqrt{\frac{|\tan(s)|(s-n\pi +\frac{\pi}{2})}{s}} \leq \sqrt{\frac{1}{n\pi - \frac{\pi}{2}}}
\end{equation}
for all $s \in [n\pi -\frac{\pi}{2}, n\pi].$ Using the proof of Corollary 2.26 from \cite{EHRNSTROM2019on_whitham_conjecture}, we have
\begin{align*}
    \frac{1}{\pi}\int_{n\pi -\frac{\pi}{2}}^{n\pi}e^{-sx}\sqrt{\frac{|\tan(s)|}{s}}ds  &=  \frac{1}{\pi}\int_{n\pi -\frac{\pi}{2}}^{n\pi}\frac{e^{-sx}\sqrt{\frac{|\tan(s)|(s-n\pi + \frac{\pi}{2})}{s}}}{\sqrt{s-n\pi + \frac{\pi}{2}}}ds\\
    &\leq  \frac{1}{\pi}\sqrt{\frac{1}{n\pi-\frac{\pi}{2}}}\int_{n\pi -\frac{\pi}{2}}^{n\pi}\frac{e^{-sx}}{\sqrt{s-n\pi + \frac{\pi}{2}}}ds\\
    &= \frac{1}{\pi}\sqrt{\frac{1}{n\pi-\frac{\pi}{2}}}\frac{e^{-(n\pi-\frac{\pi}{2})x}}{\sqrt{x}}\int_0^{\frac{\pi}{2}x} \frac{e^{-t}}{\sqrt{t}}dt.
\end{align*}
 Since $t \to \sqrt{\frac{1}{t\pi-\frac{\pi}{2}}}e^{-(t\pi-\frac{\pi}{2})x}$ is decreasing an positive, one can prove using integral estimates that
\begin{align}\label{ineq : integral estimates}
    \sum_{n=1}^\infty \sqrt{\frac{1}{n\pi-\frac{\pi}{2}}}e^{-(n\pi-\frac{\pi}{2})x} 
    & \leq e^{-\frac{\pi x}{2}} \left(\sqrt{\frac{2}{\pi}}  + \frac{1}{\sqrt{\pi x}}\right).
\end{align}
Moreover, 
\begin{align}\label{ineq : series estimates}
    \sum_{n=1}^\infty \sqrt{\frac{1}{n\pi-\frac{\pi}{2}}}e^{-(n\pi-\frac{\pi}{2})x}  \leq  \sqrt{\frac{2}{\pi}} \sum_{n=1}^\infty e^{-(n\pi-\frac{\pi}{2})x} = \sqrt{\frac{2}{\pi}}\frac{e^{-\frac{\pi x}{2}}}{1-e^{-\pi x}}.
\end{align}
In addition, observe that 
\begin{align}\label{ineq : last term}
    \int_0^{\frac{\pi}{2}x} \frac{e^{-t}}{\sqrt{t}}dt \leq \min\left\{\sqrt{2\pi x},~  \sqrt{\pi}\right\}.
\end{align}
Therefore, combining \eqref{ineq : integral estimates}, \eqref{ineq : series estimates} and \eqref{ineq : last term}, it yields
\begin{align}\label{eg : g0 ineq in proof}
     g_0(x) \leq \min \left\{\frac{1}{\pi}e^{-\frac{\pi}{2}|x|}(2 + \frac{\sqrt{2}}{\sqrt{|x|}}),~ \frac{1}{\pi}\sqrt{\frac{2}{|x|}}\frac{e^{-\frac{\pi |x|}{2}}}{1-e^{-\pi |x|}}\right\}  \leq  \frac{{C}_{\mathcal{Y}_0,0}}{\pi} \frac{e^{-\frac{\pi |x|}{2}}}{\sqrt{|x|}}
\end{align}
for all $x \neq 0$, where we used the parity of $g_0$ and where 
\[
{C}_{\mathcal{Y}_0,0} \bydef \max_{s >0}\min\left\{\sqrt{s}+\sqrt{2}, \frac{\sqrt{2}}{1-e^{-\pi s}}\right\}.
\]
Let $x \neq 0$, then
\begin{align}\label{ineq : g1}
    g_1(x) = \frac{1}{\pi \sqrt{T}}\int_{1}^\infty\frac{e^{-\frac{s}{\sqrt{T}}|x|}}{\sqrt{s^2-1}}ds =  \frac{e^{-\frac{|x|}{\sqrt{T}}}}{\pi \sqrt{T}} \int_{0}^\infty\frac{e^{-\frac{s}{\sqrt{T}}|x|}}{\sqrt{s(s+2)}}ds \leq \frac{e^{-\frac{|x|}{\sqrt{T}}}}{\sqrt{2\pi T^{\frac{1}{2}}|x|}}.
\end{align}
 Now that $g_0$ and $g_1$ have been estimated, we can estimate $f_{\mathcal{Y}_0,T} = g_0*g_1.$ Let $y >0$, then
\begin{align}
    |f_{\mathcal{Y}_0,T}(y)| &\leq \frac{\mathcal{C}}{\pi \sqrt{T}} \int_\R \frac{e^{-\frac{\pi |x|}{2}}}{\sqrt{|x|}} \int_{1}^\infty \frac{e^{-\frac{s}{\sqrt{T}}|y-x|}}{\sqrt{s^2-1}}dsdx\\
    &=  \frac{\mathcal{C}}{\pi \sqrt{T}} \int_{1}^\infty \int_{-\infty}^0 \frac{e^{\frac{\pi x}{2}}}{\sqrt{-x}} \frac{e^{-\frac{s}{\sqrt{T}}(y-x)}}{\sqrt{s^2-1}}dsdx 
    + \frac{\mathcal{C}}{\pi \sqrt{T}} \int_{1}^\infty \int_{0}^y \frac{e^{-\frac{\pi x}{2}}}{\sqrt{x}}  \frac{e^{-\frac{s}{\sqrt{T}}(y-x)}}{\sqrt{s^2-1}}dsdx\\
    &~~~+ \frac{\mathcal{C}}{\pi \sqrt{T}} \int_{1}^\infty \int_{y}^\infty \frac{e^{-\frac{\pi x}{2}}}{\sqrt{x}}  \frac{e^{\frac{s}{\sqrt{T}}(y-x)}}{\sqrt{s^2-1}}dsdx.
\end{align}
Denoting $a_0 \bydef \min\{\frac{\pi}{2},~\frac{1}{\sqrt{T}}\}$, notice that
\begin{align*}
   \int_{1}^\infty \int_{-\infty}^0 \frac{e^{\frac{\pi x}{2}}}{\sqrt{-x}} \frac{e^{-\frac{s}{\sqrt{T}}(y-x)}}{\sqrt{s^2-1}}dsdx = \int_1^\infty \frac{\sqrt{\pi}}{\sqrt{\frac{\pi}{2} + \frac{s}{\sqrt{T}}}} \frac{e^{\frac{-sy}{\sqrt{T}}}}{\sqrt{s^2-1}}ds \leq  e^{-a_0y}\int_1^\infty\frac{\sqrt{\pi}}{\sqrt{\frac{\pi}{2} + \frac{s}{\sqrt{T}}}} \frac{1}{\sqrt{s^2-1}}ds.
\end{align*}
Then, we have 
\begin{align*}
    \int_1^\infty\frac{\sqrt{\pi}}{\sqrt{\frac{\pi}{2} + \frac{s}{\sqrt{T}}}} \frac{1}{\sqrt{s^2-1}}ds \leq T^{\frac{1}{4}}\sqrt{\pi}\int_1^\infty\frac{1}{\sqrt{s-1}(s+a_0\sqrt{T}) }ds = \frac{T^{\frac{1}{4}}\pi^{\frac{3}{2}}}{\sqrt{1+{a_0}T^{\frac{1}{2}}}}.
\end{align*}
Similarly,
\begin{align*}
   \int_{1}^\infty \int_{y}^\infty \frac{e^{-\frac{\pi x}{2}}}{\sqrt{x}}  \frac{e^{\frac{s}{\sqrt{T}}(y-x)}}{\sqrt{s^2-1}}dsdx &=  e^{-\frac{\pi y}{2}}\int_{1}^\infty \int_{0}^\infty \frac{e^{-\frac{\pi x}{2}}}{\sqrt{x+y}}  \frac{e^{-\frac{s}{\sqrt{T}}x}}{\sqrt{s^2-1}}dsdx\\
   &\leq   e^{-\frac{\pi y}{2}}\int_1^\infty\frac{\sqrt{\pi}}{\sqrt{\frac{\pi}{2} + \frac{s}{\sqrt{T}}}} \frac{1}{\sqrt{s^2-1}}ds\\
   & \leq \frac{T^{\frac{1}{4}}\pi^{\frac{3}{2}}}{\sqrt{1+{a_0}T^{\frac{1}{2}}}} e^{-a_0y}.
\end{align*}
Finally, using \eqref{ineq : g1}, we get
\begin{align*}
    \frac{1}{\pi \sqrt{T}}\int_{1}^\infty \int_{0}^y \frac{e^{-\frac{\pi x}{2}}}{\sqrt{x}}  \frac{e^{-\frac{s}{\sqrt{T}}(y-x)}}{\sqrt{s^2-1}}dsdx &\leq \frac{1}{\sqrt{2\pi T^{\frac{1}{2}}}}\int_{0}^y \frac{e^{-\frac{\pi x}{2}}}{\sqrt{x}}\frac{e^{-\frac{(y-x)}{\sqrt{T}}}}{\sqrt{y-x}}dx \\
    &\leq  \frac{1}{\sqrt{2\pi T^{\frac{1}{2}}}}e^{-\min\{\frac{\pi}{2}, \frac{1}{\sqrt{T}}\}y} \int_{0}^{y} \frac{1}{\sqrt{x}}\frac{1}{\sqrt{y-x}}dx \\
    &= \frac{\sqrt{\pi}}{\sqrt{2 T^{\frac{1}{2}}}}  e^{-a_0y}. 
\end{align*}
Using the parity of $f_{\mathcal{Y}_0,T}$, this concludes the proof for $f_{\mathcal{Y}_0,T}$ when $T>0$. If $T=0$, then 
\begin{align*}
    f_{\mathcal{Y}_0,0} &=  \mathcal{F}^{-1}\left(\frac{m_0(2\pi \cdot)}{l_\nu(2\pi \cdot)}\right)\\
    &=   \mathcal{F}^{-1}\left(m_0(2\pi \cdot)\right)*\mathcal{F}^{-1}\left(\frac{1}{{1+\nu(2\pi \cdot)^2}}\right)
\end{align*}
where $\nu = \frac{4}{\pi^2}$ (cf. \eqref{def : nu}) if $T=0$. Now, using that 
\begin{align*}
    \mathcal{F}^{-1}\left(\frac{1}{{1+\nu(2\pi \cdot)^2}}\right)(x) = \frac{\pi}{4}{e^{-\frac{\pi}{2}|x|}}
\end{align*}
for all $x \in \R$, we get
\begin{align}
    |f_{\mathcal{Y}_0,0}(y)| &\leq \frac{{C}_{\mathcal{Y}_0,0}}{4} \int_\R \frac{e^{-\frac{\pi |x|}{2}}}{\sqrt{|x|}} e^{-\frac{\pi}{2}|x-y|} dx
\end{align}
for all $y \in \R,$ where we used \eqref{eg : g0 ineq in proof}. Let $0 \leq  y$, then
\begin{align}
   \int_\R \frac{e^{-\frac{\pi |x|}{2}}}{\sqrt{|x|}} e^{-\frac{\pi}{2}|x-y|} dx &\leq e^{-\frac{\pi}{2}y}\int_{-\infty}^0 \frac{e^{{\pi x}}}{\sqrt{-x}} dx + \int_0^y \frac{e^{-\frac{\pi y}{2}}}{\sqrt{x}}  dx + e^{\frac{\pi}{2}y}\int_y^{\infty} \frac{e^{-{\pi x}}}{\sqrt{x}} dx\\
   &=e^{-\frac{\pi}{2}y} + 2\sqrt{y}e^{-\frac{\pi}{2}y} + e^{-\frac{\pi}{2}y}\int_0^{\infty} \frac{e^{-{\pi x}}}{\sqrt{x+y}} dx\\
   & \leq 2e^{-\frac{\pi}{2}y}\left(1 + \sqrt{y}\right).
\end{align}
Now, using that $e^{-(\frac{\pi}{2}-1)y}\left(1 + \sqrt{y}\right) \leq 2$ for all $y \geq 0$, we obtain that
\begin{align*}
     \int_\R \frac{e^{-\frac{\pi |x|}{2}}}{\sqrt{|x|}} e^{-\frac{\pi}{2}|x-y|} dx \leq 4 e^{-y}.
\end{align*}
Therefore, using the parity of $f_{\mathcal{Y}_0,0}$ we obtain that
\begin{align*}
    \left|f_{\mathcal{Y}_0,0}(x)\right| &\leq  C_{\mathcal{Y}_0,0}e^{-|x|}
\end{align*}
for all $x \in \R$. This finishes the proof for $f_{\mathcal{Y}_0,0}.$

Let us now take care of $f_{i,T}$ ($i \in \{0,1,2\}$). The proof is based on Cauchy's theorem and on the results obtained in Proposition \ref{prop : analyticity and value of a}. Notice first that 
\begin{align*}
    |\text{tanh}(\xi+ia)|^2 = \frac{|\cosh(2\xi) - \cos(2a)|}{|\cosh(2\xi) + \cos(2a)|} = \frac{|1 - \frac{\cos(2a)}{\cosh(2\xi)}|}{|1 + \frac{\cos(2a)}{\cosh(2\xi)}|}
\end{align*}
therefore
\begin{align}\label{ineq : comput f1}
  \frac{1}{C_a} = \frac{1-|\cos(2a)|}{1+|\cos(2a)|}  \leq |\text{tanh}(\xi+ia)|^2 \leq \frac{1+|\cos(2a)|}{1-|\cos(2a)|} =C_a
\end{align}
for all $\xi \in \R$. Moreover, given $T \geq 0$, we have 
\begin{equation}\label{ineq : comput f2}
    |1+T(\xi+ia)^2|^2 = \left(1+T(\xi^2-a^2)\right)^2 + 4T\xi^2a^2. 
\end{equation}
Now, let $x \geq 0$, then using Proposition \ref{prop : analyticity and value of a}, Cauchy's theorem is applicable to $\frac{1}{(l)l_\nu}$ on $S$. Consequently
\begin{align*}
    f_{0,T}(x) = \frac{ e^{-ax}}{2\pi} \int_{\mathbb{R}}\frac{1}{(m_T(\xi+ia)-c)l_\nu(\xi+ia)}e^{i\xi x}d\xi,
\end{align*}
which implies that
\begin{align*}
    |f_{0,T}(x)| \leq \frac{ e^{-ax}}{2\pi \sigma_0} \int_{\mathbb{R}}\frac{1}{|l_\nu(\xi+ia)|}d\xi
\end{align*}
using \eqref{eq : alpha assumption}. But now, using \eqref{ineq : comput f2}, we have 
\begin{align*}
     \int_{\mathbb{R}}\frac{1}{|l_\nu(\xi+ia)|}d\xi &=  \int_{\mathbb{R}}\frac{1}{\left(\left(1+\nu(\xi^2-a^2)\right)^2 + 4\nu\xi^2a^2\right)^{\frac{1}{2}}}d\xi\\
     &\leq \int_{\mathbb{R}}\frac{1}{\left((1-\nu a^2)^2 + \nu^2\xi^4\right)^{\frac{1}{2}}}d\xi\\
     &\leq  \frac{2}{1-\nu a^2} + \frac{2}{\nu} = 2\pi \sigma_0 C_{0,T}.
\end{align*}
This concludes the proof for $f_{0,T}$. Then, switching to $f_{1,T}$ and using a similar analysis, we get
\begin{align*}
    f_{1,T}(x) = \frac{ e^{-ax}}{2\pi} \int_{\mathbb{R}}\frac{\xi+ia}{(m_T(\xi+ia)-c)l_\nu(\xi+ia)}e^{i\xi x}d\xi.
\end{align*}
If $T>0$, using \eqref{eq : alpha assumption}, we get
\begin{align*}
    |f_{1,T}(x)| &\leq \frac{ e^{-ax}}{2\pi} \int_{\mathbb{R}}\frac{|\xi|+a}{|l(\xi)|\left(\left(1+\nu(\xi^2-a^2)\right)^2 + 4\nu\xi^2a^2\right)^{\frac{1}{2}}}d\xi\\
    &\leq \frac{ e^{-ax}}{2\pi}\left( \int_{|\xi| \leq 1}\frac{|\xi|+a}{\sigma_0(1-\nu a^2)}d\xi  + \int_{|\xi| \geq  1}\frac{|\xi|+a}{\sigma_1\sqrt{T} \nu |\xi|^{\frac{5}{2}}} d\xi\right)\\
     &\leq \frac{ e^{-ax}}{2\pi}\left( \frac{2(1+a)}{\sigma_0(1-\nu a^2)}  + \frac{4(1+a)}{\sigma_1\sqrt{T} \nu}\right) = C_{1,T} e^{-ax}.
\end{align*}
If $T = 0$, then notice that
\begin{align*}
    \frac{1}{(l(\xi))l_\nu(\xi)} = -\frac{1}{cl_\nu(\xi)} + \frac{1}{l_{\nu}(\xi)}\left(\frac{1}{m_0(\xi)-c} + \frac{1}{c}\right) = -\frac{1}{cl_\nu(\xi)} + \frac{1}{cl_{\nu}(\xi)}\left(\frac{m_0(\xi)}{m_0(\xi)-c}\right).
\end{align*}
Therefore, we obtain 
\begin{align*}
    f_{1,T} &= -\frac{1}{c}\mathcal{F}^{-1}\left(\frac{2\pi \xi}{l_\nu(2\pi \xi)}\right) + \mathcal{F}^{-1}\left( \frac{2\pi \xi }{cl_{\nu}(2\pi\xi)}\left(\frac{m_0(2\pi\xi)}{m_0(2\pi\xi)-c}\right)\right).
\end{align*}
But using \cite{poularikas2018transforms}, we know that
\begin{align}\label{fourier transform of lnu}
    \mathcal{F}^{-1}\left(\frac{2\pi i \xi}{l_\nu(2\pi \xi)}\right) = \mathcal{F}^{-1}\left(\frac{2\pi i \xi}{1 + \nu(2\pi\xi)^2}\right)  = -\text{sign}(x)\frac{e^{-\frac{|x|}{\sqrt{\nu}}}}{2\nu}.
\end{align}
Moreover, using \eqref{ineq : comput f1}, we have 
\begin{equation}\label{step m0 for f1}
   |m_0(\xi+ia)| \leq \left(C_a\right)^{\frac{1}{4}} \frac{1}{(\xi^2 + a^2)^{\frac{1}{4}}}. 
\end{equation}
Therefore, combining \eqref{eq : alpha assumption} and \eqref{step m0 for f1}, we obtain
\begin{align*}
    \frac{ |\xi + ia| }{|cl_{\nu}(\xi+ia)|}\left|\frac{m_0(\xi+ia)}{m_0(\xi+ia)-c}\right| \leq \frac{(\xi^2 + a^2)^{\frac{1}{4}}}{|c|\sigma_0} \frac{\left(C_a\right)^{\frac{1}{4}}}{\left(\left(1+\nu(\xi^2-a^2)\right)^2 + 4\nu\xi^2a^2\right)^{\frac{1}{2}}}.
\end{align*}
Similarly as above, Cauchy's theorem combined with \eqref{fourier transform of lnu} yields
\begin{align*}
    |f_{1,0}(x)| &\leq \frac{e^{-\frac{|x|}{\sqrt{\nu}}}}{2|c|\nu} +  \frac{e^{-a|x|}}{2\pi}\int_\R \frac{(\xi^2 + a^2)^{\frac{1}{4}}}{|c|\sigma_0}\frac{\left(C_a\right)^{\frac{1}{4}}}{\left(\left(1+\nu(\xi^2-a^2)\right)^2 + 4\nu\xi^2a^2\right)^{\frac{1}{2}}} d\xi\\
    &\leq \frac{e^{-\frac{|x|}{\sqrt{\nu}}}}{2|c|\nu} +  \frac{e^{-a|x|}}{\pi}\int_0^1 \frac{1+\sqrt{a}}{|c|\sigma_0} \frac{\left(C_a\right)^{\frac{1}{4}}}{1-\nu a^2} d\xi + \frac{e^{-a|x|}}{\pi}\int_1^\infty \frac{1+\sqrt{a}}{|c|\sigma_0} \frac{\left(C_a\right)^{\frac{1}{4}}}{\nu\xi^{\frac{3}{2}}} d\xi\\
   & = \frac{e^{-\frac{|x|}{\sqrt{\nu}}}}{2|c|\nu} +  \frac{e^{-a|x|}}{\pi} \frac{1+\sqrt{a}}{|c|\sigma_0} \frac{\left(C_a\right)^{\frac{1}{4}}}{1-\nu a^2}  + \frac{2e^{-a|x|}}{\pi} \frac{1+\sqrt{a}}{|c|\sigma_0} \frac{\left(C_a\right)^{\frac{1}{4}}}{\nu}.
\end{align*}
We conclude the proof for $f_{1,0}$ noticing that $\frac{1}{\sqrt{\nu}}\geq a$ by assumption on $a$.

Finally, let us focus on $f_{2,T}$. Let $T>0$ and  $\xi > 0$, then
\begin{align*}
    \frac{1}{l(\xi)} = \frac{1}{\sqrt{T|\xi|}} + \frac{1}{l(\xi)} - \frac{1}{\sqrt{T|\xi|}}.
\end{align*}
First, notice that given $x \neq 0$, we have
\begin{equation}\label{eq : fourier transform sqrt}
     \mathcal{F}^{-1}\left(\frac{1}{\sqrt{T|\xi|}}\right)(x) = \frac{1}{\sqrt{2\pi T|x|}}.
\end{equation}
Then, 
\begin{align*}
    \frac{1}{\sqrt{T|\xi|}} - \frac{1}{l(\xi)}   &= \frac{1}{\sqrt{T\xi }}\left(\frac{\sqrt{\frac{\tanh(\xi)(1+T\xi^2)}{\xi}} - c - \sqrt{T\xi}}{\sqrt{\frac{\tanh(\xi)}{\xi}(1+T\xi^2)}-c} \right)\\
    &=\frac{\sqrt{\tanh(\xi) + \frac{\tanh(\xi)}{T\xi^2}} - \frac{c}{\sqrt{T\xi}} - 1}{\sqrt{\frac{\tanh(\xi)}{\xi}(1+T\xi^2)}-c}.
\end{align*}
Therefore, using that $|\tanh(\xi)| \leq 1$ and $|\tanh(\xi)| \leq |\xi|$ for all $\xi \in \R$ combined with \eqref{eq : alpha assumption}, we get
\begin{align}\label{square root f2}
    \bigg|\frac{1}{\sqrt{T|\xi|}} - \frac{1}{l(\xi)}\bigg| 
    &\leq  \frac{1}{\sigma_0}\left(2 + \frac{1+|c|}{\sqrt{T|\xi|}}\right).
\end{align}
 
Let $g : \R^+ \to \R$ be defined as  $g(\xi) = \tanh(\xi) + \frac{\tanh(\xi)}{T\xi^2}$ for all $\xi \in \R^+$. Then notice that
\[
g'(\xi) = 1-\tanh(\xi)^2 + \frac{1-\tanh(\xi)^2}{T\xi^2} - 2\frac{\tanh(\xi)}{T\xi^3}
\]
and $g'(\xi) \leq 2 e^{-2\xi} - 2 \frac{\tanh(\xi)}{T\xi^3} \leq 2e^{-2\xi}\left( 1 -  \frac{\tanh(\xi_0)e^{2\xi}}{T\xi^3}\right)$ for all $\xi \geq \xi_0$. Then using that $e^{2\xi} \geq \frac{2}{3}\xi^4$ for all $\xi \geq 0$, we get
$g'(\xi) \leq 2e^{-2\xi}\left( 1 -  \frac{2\tanh(\xi_0)\xi}{3T}\right)$. Therefore, choosing $\xi_0$ such that $1 \leq \frac{2\tanh(\xi_0)\xi_0}{3T}$ we obtain that $g'(\xi) \leq 0$ for all $\xi \geq \xi_0$ and so $g$ is decreasing for all $\xi \geq \xi_0$. In particular,
\[
\sqrt{\tanh(\xi) + \frac{\tanh(\xi)}{T\xi^2}} -1 \geq  1 -1 =0
\]
for all $\xi \geq \xi_0$. Moreover, this implies that
\begin{equation}\label{ineq : tanh increasing tata}
    \left|\sqrt{\tanh(\xi) + \frac{\tanh(\xi)}{T\xi^2}} -1\right| = \sqrt{\tanh(\xi) + \frac{\tanh(\xi)}{T\xi^2}} -1 \leq \sqrt{1 + \frac{1}{T\xi^2}} -1 \leq \frac{1}{2T\xi^2}
\end{equation}
for all $\xi \geq \xi_0$ as $\xi_0 \geq \frac{1}{\sqrt{T}}$.  
 
Let $\xi \geq \xi_0$, then noticing that $\frac{c}{T\xi}\left(l(\xi)\right) = \frac{c}{\sqrt{T\xi}}\left(\sqrt{\tanh(\xi) + \frac{\tanh(\xi)}{T\xi^2}}-\frac{c}{\sqrt{T\xi}}\right)$ and using \eqref{ineq : tanh increasing tata}, we get
\begin{align*}
      \bigg|\frac{1}{\sqrt{T|\xi|}} - \frac{1}{l(\xi)} + \frac{c}{T\xi}\bigg| 
    & = \bigg| \frac{\sqrt{\tanh(\xi) + \frac{\tanh(\xi)}{T\xi^2}} -1 +\frac{c}{\sqrt{T\xi}}\left(\sqrt{\tanh(\xi) + \frac{\tanh(\xi)}{T\xi^2}}-\frac{c}{\sqrt{T\xi}}-1\right) }{\sqrt{\tanh(\xi)(\frac{1+T\xi^2}{\xi})}-c}\bigg|\\
    & = \bigg| \frac{\left(\sqrt{\tanh(\xi) + \frac{\tanh(\xi)}{T\xi^2}} -1\right)\left(1 + \frac{c}{\sqrt{T\xi}}\right) -\frac{c^2}{T\xi} }{\sqrt{\tanh(\xi)(\frac{1+T\xi^2}{\xi})}-c}\bigg|\\
    &  \leq  \frac{ \frac{1}{2T\xi^2}\left(1+\frac{|c|}{\sqrt{T\xi}}\right)  +\frac{c^2}{{T\xi}} }{\sqrt{\tanh(\xi)(\frac{1+T\xi^2}{\xi})}-c}.
\end{align*}

Now, for all $\xi \geq \xi_0$,
\begin{align}\label{ineq : estimation at infinity 3 terms}
    \sqrt{\tanh(\xi)(\frac{1+T\xi^2}{\xi})}-c \geq \sqrt{\tanh(\xi_0)T\xi} - c \geq \frac{1}{2}\sqrt{\tanh(\xi_0)T\xi}
\end{align}
as $\frac{1}{2}\sqrt{\tanh(\xi_0)T\xi_0} \geq c$ by assumption on $\xi_0.$
Therefore,
\begin{align*}
    \bigg|\frac{1}{\sqrt{T|\xi|}} - \frac{1}{l(\xi)} + \frac{c}{T\xi}\bigg|  \leq \frac{2}{\sqrt{\tanh(\xi_0)T}}\left(\frac{1}{2T\xi^{\frac{5}{2}}} + \frac{|c|}{2T^{\frac{3}{2}}\xi^{3}} + \frac{c^2}{T\xi^{\frac{3}{2}}}\right)
\end{align*}
for all $\xi \geq \xi_0$. We are now set up to compute the inverse Fourier transform of $f_{2,T}$. Let $x>0$ then, using the parity of $f_{2,T}$, we have
\begin{align}\label{computation fourier transform f2}
\nonumber
    &~~~~\mathcal{F}^{-1}(f_{2,T}-\frac{1}{\sqrt{T|\xi|}})(x) \\ \nonumber
    &= \frac{1}{2\pi}\int_\mathbb{R}(\frac{1}{l(\xi)}-\frac{1}{\sqrt{T|\xi|}})e^{-i\xi x}d\xi\\ \nonumber
    & = \frac{1}{\pi}\int_0^\infty (\frac{1}{l(\xi)}-\frac{1}{\sqrt{T\xi}})\cos(\xi x)d\xi\\ \nonumber
    & = \frac{1}{\pi}\int_0^{\xi_0}(\frac{1}{l(\xi)}-\frac{1}{\sqrt{T\xi}})\cos(\xi x)d\xi + \frac{1}{\pi}\int_{\xi_0}^\infty (\frac{1}{l(\xi)}-\frac{1}{\sqrt{T\xi}}+ \frac{c}{T\xi})\cos(\xi x)d\xi\\
    &~~~~-   \frac{1}{\pi}\int_{\xi_0}^\infty  \frac{c}{T\xi}\cos(\xi x)d\xi .
\end{align}

The last term of \eqref{computation fourier transform f2} is a cosine integral. We can simplify the integral as follows
\begin{align*}
    \int_{\xi_0}^\infty  \frac{c}{T\xi}\cos(\xi x)d\xi &= \frac{c}{T}\int_{\xi_0x}^\infty  \frac{\cos(\xi)}{\xi}d\xi \\
     &= \frac{c}{T}\left(-C_{Euler} - \ln(\xi_0x) + \int^{\xi_0x}_0  \frac{1-\cos(\xi)}{\xi}d\xi\right),
\end{align*}
where $C_{Euler}$ is the  Euler–Mascheroni constant.
Now suppose that $\xi_0 x >1$, then we obtain 
\begin{align*}
    \left|\int_{\xi_0}^\infty  \frac{c}{T\xi}\cos(\xi x)d\xi\right| &\leq  \frac{|c|}{T}\left(C_{Euler} + \ln(\xi_0) + \int^{1}_0  \frac{|1-\cos(\xi)|}{\xi}d\xi + \int^{\xi_0 x}_1  \frac{|1-\cos(\xi)|}{\xi}d\xi\right)\\
    &\leq  \frac{|c|}{T}\left(C_{Euler} + \ln(\xi_0) + \int^{1}_0  \frac{\xi^2}{2\xi}d\xi + \int^{\xi_0 x}_1  \frac{2}{\xi}d\xi\right)\\
     &\leq  \frac{|c|}{T}\left(C_{Euler}  + \frac{1}{4} + 3\ln(\xi_0)\right)\\
\end{align*}
as $0<x \leq 1$ and $\xi_0 \geq 1$ by assumption. Similarly, if $0<\xi_0 x \leq 1$ we obtain that
\begin{align*}
    |\int_{\xi_0}^\infty  \frac{c}{T\xi}\cos(\xi x)d\xi| 
     &\leq  \frac{|c|}{T}\left(C_{Euler} + |\ln(\xi_0 x)| + \frac{1}{4}  \right) \leq \frac{|c|}{T}\left(C_{Euler} + |\ln(x)| + \frac{1}{4}  \right).
\end{align*}
Combining both cases, for all $0<x\leq 1$ we obtain that
\begin{align}\label{cosine integral computation}
\nonumber
    |\int_{\xi_0}^\infty  \frac{c}{T\xi}\cos(\xi x)d\xi| 
     &\leq  \frac{|c|}{T}\left(C_{Euler} + |\ln(x)| + \frac{1}{4} + 3\ln(\xi_0) \right)\\
     &\leq \frac{|c|}{T}\left(1+ \frac{1}{\sqrt{x}}  + 3\ln(\xi_0)\right) \leq \frac{|c|}{T\sqrt{|x|}}\left(2+3\ln(\xi_0)\right)
\end{align}
as $C_{Euler} + \frac{1}{4} \leq 1.$ Then, equation \eqref{ineq : estimation at infinity 3 terms} yields
\begin{align}\label{second term in the f2 comput}
\nonumber
     \frac{1}{\pi}\left|\int_{\xi_0}^\infty (\frac{1}{l(\xi)}-\frac{1}{\sqrt{T\xi}}+ \frac{c}{T\xi})\cos(\xi x)d\xi \right|&\leq  \frac{1}{\pi}\int_{\xi_0}^\infty \frac{2}{\sqrt{\tanh(\xi_0)T}}\left(\frac{1}{2T\xi^{\frac{5}{2}}} + \frac{|c|}{2T^{\frac{3}{2}}\xi^{3}} + \frac{c^2}{T\xi^{\frac{3}{2}}}\right)\\
     & \leq \frac{2}{\pi\sqrt{\tanh(\xi_0)T}\sqrt{\xi_0}}\left(\frac{1}{3T} + \frac{|c|}{4T^{\frac{3}{2}}} + \frac{2c^2}{T}\right).
\end{align}
Moreover, using \eqref{square root f2} we get
\begin{align}\label{first term in the f2 comput}
    \frac{1}{\pi}\left|\int_0^{\xi_0}(\frac{1}{l(\xi)}-\frac{1}{\sqrt{T\xi}})\cos(\xi x)d\xi\right| \leq \frac{1}{\pi}\int_0^{\xi_0}\frac{1}{\sigma_0}\left(2 + \frac{1+|c|}{\sqrt{T|\xi|}}\right) \leq \frac{2\xi_0}{\pi\sigma_0}+ \frac{2\sqrt{\xi_0}(1+|c|)}{\pi \sigma_0\sqrt{T}}.
\end{align}

Finally, combining \eqref{cosine integral computation}, \eqref{second term in the f2 comput} and \eqref{first term in the f2 comput}, we obtain that
\begin{align}\label{eq : fourier transform f minus sqrt}
    \bigg|\mathcal{F}^{-1}(f_{2,T}-\frac{1}{\sqrt{T|\xi|}})(x)\bigg|
    & \leq \frac{\tilde{K}_{1,T}}{\sqrt{|x|}}
\end{align}
for all  $0<x\leq 1$ where 
\begin{align*}
    \tilde{K}_{1,T} \bydef \frac{2\xi_0}{\pi\sigma_0}+ \frac{2\sqrt{\xi_0}(1+c)}{\pi \sigma_0\sqrt{T}} + \frac{2}{\pi\sqrt{\tanh(\xi_0)T}\sqrt{\xi_0}}\left(\frac{1}{3T} + \frac{|c|}{4T^{\frac{3}{2}}} + \frac{2c^2}{T}\right) + \frac{|c|}{T\pi}\left(2  + 3\ln(\xi_0)\right).
\end{align*}

Consequently, combining \eqref{eq : fourier transform sqrt} and \eqref{eq : fourier transform f minus sqrt} with the parity of $f_{2,T}$, we get
\begin{align*}
    \bigg|\mathcal{F}^{-1}(f_{2,T})(x)\bigg| \leq \frac{\tilde{K}_{2,T} + \frac{1}{\sqrt{2\pi T}}}{\sqrt{|x|}} = \frac{K_{1,T,c}}{\sqrt{|x|}}
\end{align*}
for all $0 < |x| \leq 1$. 

Let $x >1$, then using Cauchy's theorem,
\begin{align*}
    \mathcal{F}^{-1}(f_{2,T})(x) = \frac{e^{-ax}}{2\pi}\int_{\mathbb{R}}e^{ix\xi}f_{2,T}(\xi+ia)d\xi.
\end{align*}
Now using integration by parts we obtain
\begin{align*}
    |\mathcal{F}^{-1}(f_{2,T})(x)| \leq \frac{e^{-ax}}{2\pi|x|}\int_\mathbb{R}|f_{2,T}'(\xi+ia)|d\xi .
\end{align*}

Defining $u$ as  $u(\xi) \bydef \frac{\tanh(\xi)(1+T\xi^2)}{\xi}$ for all $\xi \in \R$ and letting $z \bydef {\xi + ia}$, we get
\begin{align*}
    f_{2,T}'(\xi) &=\frac{u'(z)}{2\sqrt{u(z)}(\sqrt{u(z)}-c)^2}
    \\&=\frac{1}{2\sqrt{u(z)}(\sqrt{u(z)}-c)^2}\frac{(1-Tz^2)\tanh(z) - z\text{sech}(z)^2(1+Tz^2)}{z^2}.
\end{align*}

First, assume that $\xi \geq \xi_0$, then
\begin{align}\label{eq : estimation sech}
    |\text{sech}(z)| = \frac{|e^{\xi-ia} + e^{-\xi+ia}|}{|2\cosh(2\xi)+2\cos(2a)|}
    = \frac{1}{2\cosh(2\xi)}\frac{|e^{\xi-ia} + e^{-\xi+ia}|}{|1-\frac{|\cos(2a)|}{\cosh(2\xi)}|} 
    \leq C_a{e^{-|\xi|}}. 
\end{align}

Moreover using \eqref{ineq : comput f1}, we get
\begin{align}\label{ineq : u estimation below}
    |u(z)|^2 = \frac{|\tanh(z)^2(1+Tz^2)^2|}{|z|^2} \geq \frac{1}{C_a}\frac{|1+Tz^2|^2}{|z^2|}.
\end{align}

Then,
\begin{align}
\nonumber
    |1+Tz^2|^2 &= (T(\xi^2 - a^2)+1)^2 + 4T^2\xi^2a^2\\ \nonumber
    &= T^2(\xi^2-a^2)^2 +2T(\xi^2-a^2) +1 + 4T^2\xi^2a^2\\ \nonumber
    & = T^2\xi^4 -2T^2a^2\xi^2 + T^2a^4 + 2T\xi^2 - 2Ta^2 + 1 + 4T^2\xi^2a^2\\ \nonumber
    & = T^2\xi^4 + \xi^2\left(2T +  2T^2a^2 \right) + T^2a^4 - 2Ta^2 + 1\\ \nonumber
   & = T^2\xi^4 + \xi^2\left(2T +  2T^2a^2 \right) + (Ta^2-1)^2\\
   & \geq T^2\xi^4 \label{eq : estimation last step}.
\end{align}

Therefore, combining \eqref{ineq : u estimation below}  and \eqref{eq : estimation last step}, we obtain  
\begin{align}\label{eq : estimation of u without squared}
    |u(z)| \geq \frac{1}{\sqrt{C_a}}T|\xi|.
\end{align}
Now, using that $\frac{1}{2}\left(\frac{1}{C_a}\right)^{\frac{1}{4}}\sqrt{T\xi}-|c| \geq \frac{1}{2}\left(\frac{1}{C_a}\right)^{\frac{1}{4}}\sqrt{T\xi_0}-|c| \geq 0$ by assumption on $\xi_0$, we get
\begin{align}\label{eq : last step estimation u}
    |\sqrt{|u(z)|}-c| \geq \left(\frac{1}{C_a}\right)^{\frac{1}{4}}\frac{\sqrt{T|\xi|}}{2}.
\end{align}

Therefore, combining \eqref{ineq : comput f1},  \eqref{eq : estimation sech}, \eqref{eq : estimation of u without squared} and \eqref{eq : last step estimation u} and using that $C_a \geq 1$, it yields 
\begin{align*}
    |f_{2,T}'(z)| &\leq 2C_a^2\left(\frac{1}{T|\xi|}\right)^{\frac{3}{2}}\left( (1 + T) + (1+ T|z|)e^{-2|\xi|}\right)
\end{align*}
for all $\xi \geq \xi_0$ as $|z| = |\xi + ia| \geq 1$ (using that $\xi \geq \xi_0 \geq 1$).  But then notice that
\begin{align*}
    \frac{1+T|z|}{(T|\xi|)^{\frac{3}{2}}} \leq \frac{1+ T|\xi| + aT}{(T|\xi|)^{\frac{3}{2}}}  \leq \frac{1+ aT}{(T|\xi_0|)^{\frac{3}{2}}} + \frac{1}{(T|\xi_0|)^{\frac{1}{2}}} \leq 2+a
\end{align*}
as $\xi_0 \geq \max\{1,\frac{1}{\sqrt{T}}\}$
therefore 
\begin{align}\label{eq : f2 computation at infinity}
  |f_{2,T}'(z)| &\leq 2C_a^{2}\left( \frac{(1 + T)}{(T|\xi|)^{\frac{3}{2}}} + (2+a){e^{-2|\xi|}}\right)
\end{align}
for all $|\xi| \geq \xi_0$ using the parity of $f_{2,T}.$

Now suppose that $|\xi| \leq \xi_0$ and let $z = \xi + ia$, then using \eqref{ineq : u estimation below} and \eqref{eq : estimation last step}, we get
\begin{align}\label{estimation from below z small}
    |u(z)|^2 \geq \frac{1}{C_a}\frac{(1-Ta^2)^2}{\xi_0^2+a^2}.
\end{align}
Therefore, combining \eqref{eq : alpha assumption}, \eqref{ineq : comput f1}, \eqref{eq : estimation sech} and \eqref{estimation from below z small}, we obtain
\begin{align}\label{f2 computation at zer0}
    |f_{2,T}'(z)| &\leq \frac{C_a\sqrt{\xi_0^2+a^2}}{2\sigma_0^2(1-Ta^2)^2}\left( (\frac{1}{a^2}+T) + C_ae^{-2|\xi|}(\frac{1}{a}+T\sqrt{\xi_0^2+a^2})\right)
\end{align}
for all $|\xi| \leq \xi_0$. Hence, combining \eqref{eq : f2 computation at infinity} and \eqref{f2 computation at zer0}, we get
\begin{align*}
   &~~~~\frac{1}{2\pi}\int_\mathbb{R}|f_{2,T}'(\xi+ia)|d\xi\\
   &= \frac{1}{2\pi}\int_{|\xi| \leq \xi_0}|f_{2,T}'(\xi+ia)|d\xi + \frac{1}{2\pi}\int_{|\xi| \geq \xi_0}|f_{2,T}'(\xi+ia)|d\xi \\
   &\leq \frac{C_a\sqrt{\xi_0^2+a^2}}{2\pi\sigma_0^2(1-Ta^2)^2}\left( (\frac{1}{a^2}+T) + C_a(\frac{1}{a}+T\sqrt{\xi_0^2+a^2})\right) + \frac{2C_a^{2}}{\pi}\left( \frac{2(1 + T)}{(T|\xi_0|)^{\frac{1}{2}}} + \frac{(2+a)}{2}{e^{-2|\xi_0|}}\right) \\
   &= K_{2,T}
\end{align*}
 To conclude the proof for $f_{2,T}$ ($T>0$), recall that 
we obtained $|f_{2,T}(x)| \leq \frac{K_{1,T,c}}{\sqrt{|x|}}$ for all $|x| \leq 1$ and $|f_{2,T}(x)| \leq \frac{K_{2,T}e^{-a|x|}}{|x|}$ for all $|x| \geq 1$. Therefore, it implies that 
\begin{equation}\label{ineq with square root and 1/x}
  |f_{2,T}(x)| \leq \max\left\{ K_{2,T}, K_{1,T,c}e^{a}\right\}\frac{e^{-a|x|}}{\sqrt{|x|}}   
\end{equation}
for all $x \in \R.$

Finally, we consider the case $T=0$. Observe that
\begin{align*}
    \frac{1}{m_0(\xi)-c} = -\frac{1}{c} + \frac{1}{m_0(\xi)-c} + \frac{1}{c}.
\end{align*}
Then, we have that $\mathcal{F}^{-1}(\frac{1}{c})(x) = \frac{1}{c}\delta(x)$ where $\delta$ is the Dirac-delta function. Now,
let us denote $h(\xi) \bydef \frac{1}{m_0(\xi)-c} + \frac{1}{c}$. Moreover, let $0 < |x| \leq 1$ and let $\xi >0$, then
\begin{align*}
    h(\xi) = -\frac{1}{c^2\sqrt{|\xi|}} +  h(\xi) + \frac{1}{c^2\sqrt{|\xi|}}
\end{align*}
and notice that 
\begin{align}
    \mathcal{F}^{-1}\left(-\frac{1}{c^2\sqrt{|\xi|}}\right)(x) = -\frac{1}{c^2\sqrt{2\pi|x|}}. 
\end{align}
Moreover, we have
\begin{align}\label{estimation h at zero}
    \left|h(\xi) + \frac{1}{c^2\sqrt{|\xi|}}\right|  \leq \frac{m_0(\xi)}{|c||m_0(\xi)-c|} +  \frac{1}{c^2\sqrt{|\xi|}} \leq \frac{1}{\sigma_0|c|} + \frac{1}{c^2\sqrt{|\xi|}}
\end{align}
and
\begin{align*}
    h(\xi) + \frac{1}{c^2\sqrt{|\xi|}} + \frac{1}{c^3|\xi|} &= \frac{\sqrt{\tanh(\xi)-1}}{|c|\sqrt{|\xi|}(m_0(\xi)-c)} + \frac{m_0(\xi)}{c^2\sqrt{|\xi|}(m_0(\xi)-c)}  + \frac{1}{c^3|\xi|}\\
    &= \frac{\sqrt{\tanh(\xi)-1}}{|c|\sqrt{|\xi|}(m_0(\xi)-c)} + \frac{m_0(\xi)c^3|\xi| + c^2\sqrt{|\xi|}m_0(\xi) - c^3\sqrt{|\xi|}}{c^5|\xi|^{\frac{3}{2}}(m_0(\xi)-c)} \\
    &= \frac{\sqrt{\tanh(\xi)-1}}{|c|\sqrt{|\xi|}(m_0(\xi)-c)} + \frac{\sqrt{\tanh(\xi)}-1}{c^2|\xi|(m_0(\xi)-c)} +     \frac{m_0(\xi)}{c^3|\xi|(m_0(\xi)-c)}.
\end{align*}
Notice that $|\tanh(\xi)-1| \leq 2e^{-2\xi}$ for all $\xi \geq 0$, therefore we have
\begin{align}\label{estimation h at infinity}
\nonumber
   \left|h(\xi) + \frac{1}{c^2\sqrt{|\xi|}} + \frac{1}{c^3|\xi|}\right| &\leq  e^{-2\xi}\left(\frac{2}{|c|\sqrt{\xi_0}\sigma_0} + \frac{2}{c^2\xi_0\sigma_0}\right) + \frac{1}{|\xi|^{\frac{3}{2}}c^3\sigma_0} \\
   &\leq \frac{1}{\min\{|c|^3,1\}\sigma_0}\left( \frac{1}{|\xi|^{\frac{3}{2}}} + \frac{4e^{-2\xi}}{\sqrt{\xi_0}} \right) 
\end{align}
for all $\xi \geq \xi_0$ as $\xi_0 \geq 1$.

Therefore, combining \eqref{estimation h at zero} and \eqref{estimation h at infinity}, and using the above reasoning for the case $T>0$ (starting at \eqref{computation fourier transform f2}), we get
\begin{align*}
    \left|\mathcal{F}^{-1}(h)(x)\right|= \left|f_{2,0}(x)+\frac{1}{c}\delta(x)\right| \leq \frac{\tilde{K}_{1,0} + \frac{1}{c^2\sqrt{2\pi}}}{\sqrt{|x|}} = \frac{K_{1,0}}{\sqrt{|x|}}
\end{align*}
for all $0 < x\leq 1$ where
\[
\tilde{K}_{1,0} \bydef \frac{1}{\pi \min\{1,|c|^3\}\sqrt{\xi_0}\sigma_0}\left(2+4e^{-2\xi_0}\right) + \frac{1}{\pi \sigma_0 |c|} + \frac{2}{\pi c^2}  +  \frac{1}{\pi |c|^3}\left(2+3\ln(\xi_0)\right).
\]

Now, let $\xi \geq 0$ and denote $z = ia + \xi$, then 
\begin{align*}
    f_{2,0}'(z) = \frac{-m_0'(z)}{(m_0(z)-c)^2} = \frac{\tanh(z)-z\text{sech}(z)^2}{2m_0(z)z^2(m_0(z)-c)^2} = \frac{\sqrt{\tanh(z)}}{2z^{\frac{3}{2}}(m_0(z)-c)^2} - \frac{\text{sech}(z)^2}{2m_0(z)z(m_0(z)-c)^2}.
\end{align*}
Therefore, using \eqref{ineq : comput f1} and \eqref{eq : estimation sech}, we get
\begin{align}\label{estimation f2 for T=0}
    |f_{2,0}'(z)| \leq \left(C_a\right)^{\frac{1}{4}}\left(\frac{1}{2\sigma_0^2|\xi^2+a^2|^{\frac{3}{4}}} + \frac{e^{-2|\xi|}}{2\sigma_0^2(1-|\cos(2a)|)^2(a^2+\xi^2)^{\frac{1}{4}}}\right).
\end{align}
Finally, using Cauchy's theorem and \eqref{estimation f2 for T=0}, we obtain
\begin{align*}
    \left|\mathcal{F}^{-1}(f_{2,0})(x)\right| &\leq \frac{e^{-a|x|}}{\pi|x|} \int_0^\infty |f_{2,0}'(\xi+ia)|d\xi\\
    &\leq \frac{e^{-a|x|}}{\pi|x|}\left(C_a\right)^{\frac{1}{4}}\left( \frac{1}{2\sigma_0^2}(\frac{1}{a^{\frac{3}{2}}} + 2) + \frac{1}{4\sigma_0^2(1-|\cos(2a)|)^2\sqrt{a}}\right)\\
    &= \frac{K_{2,0}e^{-a|x|}}{|x|}
\end{align*}
for all $|x| \geq 1$. We conclude the proof using \eqref{ineq with square root and 1/x}.
\end{proof}

\subsection{Computation of \texorpdfstring{$a$}{a}, \texorpdfstring{$\sigma_0$}{a1} and \texorpdfstring{$\sigma_1$}{a2}}\label{sec : computation of constants in appendix}

In this subsection we provide the details of the rigorous computations of the constants $a, \sigma_0$ and $\sigma_1$, which are required in the analysis developed in Lemma \ref{lem : computation of f}. Our first goal is to obtain some $0 < a < \min\{\frac{1}{\sqrt{\nu}}, \frac{\pi}{2}\}$ and $\sigma_0 >0$ such that 
\begin{align}\label{ineq : requirement on a}
\nonumber
    |m_T(\xi + ia)-c|,  |l(\xi)|&\geq \sigma_0 ~~\text{for all}~~ \xi \in \R\\
    |m_T(z)-c| &> 0 ~~\text{for all}~~ z \in S
\end{align}
where $S = \{z \in \mathbb C, ~~ |\text{Im}(z)| \leq a\}$. The analysis of the constant $\sigma_1$ will be presented later on in this section. The code for the computation of the constants $a, \sigma_0$ and $\sigma_1$ is available at \cite{julia_cadiot}.

In particular, following Lemma \ref{prop : analyticity and value of a}, we choose $0 < a < \min\left\{\frac{1}{\sqrt{\nu}}, \frac{\pi}{2}\right\}$. Moreover, if $T=0$, then $c>1$ by Assumption \ref{ass : value of c and T}. This implies that  $|c - m_0(\xi)| = c - m_0(\xi)\geq c$ for all $\xi \in \R$. Consequently, if $T=0$, we impose 
 $$\sigma_0 < c.$$
Numerically, we start by fixing  candidate values for $a$ and $\sigma_0$. In particular, as mentioned above, we chose $0< a< \min\left\{\frac{1}{\sqrt{\nu}}, \frac{\pi}{2}\right\}$ and $\sigma_0 < c$.  These candidate values are usually obtained by studying the graph of $|l|$ numerically. Then, the goal is to prove that $a$ and $\sigma_0$ satisfy \eqref{ineq : requirement on a}. To do so, given $x\geq 0$, we define 
\[
S_x \bydef \{z \in \mathbb C, ~~ |\text{Im}(z)| \leq a \text{ and } |\text{Re}(z)| \leq x\}
\]
and we use the following result.
\begin{prop}\label{prop : minimal value for x in proof of a}
    Let $x>0$ be big enough so that
    \begin{align}\label{condition on x for proof of a}
    x^2 > \begin{cases}
        \left(\frac{(x^2+a^2)}{T^2\frac{\cosh(2x)-1}{\cosh(2x)+1}}\right)^{\frac{1}{2}}(c+\sigma_0)^2  &\text{ if } T > 0\\
         \frac{1 + \frac{|\cos(2a)|}{\cosh(2x)}|}{(c-\sigma_0)^4\left(1 - \frac{|\cos(2a)|}{\cosh(2x)}\right)}  &\text{ if } T = 0.
    \end{cases} 
    \end{align}
  If $|m_T(z)-c| \geq \sigma_0$ for all $z \in S_x$, then $|m_T(z)-c| \geq \sigma_0$ for all $z \in S.$
\end{prop}

\begin{proof}
  We need to prove that $|m_T(z)-c| \geq \sigma_0$ for all $z \in S\setminus S_x$.  Let $z = \xi + iy \in S\setminus S_x$ and suppose that $|y| \leq \min\{a, \frac{\pi}{4}\}$. Let $T>0$, then using \eqref{eq : estimation last step}, we get
    \begin{align*}
        |m_T(\xi + iy)|^4 &= \frac{|\tanh(\xi+iy)|^2}{|\xi+iy|^2}|1+T(\xi+iy)^2|^2\\
        &= \frac{|1 - \frac{\cos(2y)}{\cosh(2\xi)}|}{|1 + \frac{\cos(2y)}{\cosh(2\xi)}|} \frac{|1+T(\xi+iy)^2|^2}{|\xi+iy|^2}\\
        &\geq \left(\frac{1 - \frac{1}{\cosh(2x)}}{1 + \frac{1}{\cosh(2x)}}\right)\frac{T^2\xi^4}{\xi^2 + a^2}\\
        &\geq  \left(\frac{\cosh(2x)-1}{\cosh(2x)+1}\right) \frac{T^2x^4}{x^2 + a^2} > (c+\sigma_0)^4
    \end{align*}
    as $\cos(2y) \geq 0$ and as $\xi \to \frac{T^2\xi^4}{\xi^2 + \frac{\pi^2}{16}}$ is increasing on $(0, \infty)$. Now if $T=0$, then
    \begin{align*}
        |m_0(\xi + iy)|^4 = \frac{|1 - \frac{\cos(2y)}{\cosh(2\xi)}|}{|1 + \frac{\cos(2y)}{\cosh(2\xi)}|} \frac{1}{\xi^2+y^2} \leq \frac{1}{\xi^2}\leq \frac{1}{x^2} \leq \frac{1 + \frac{|\cos(2a)|}{\cosh(2x)}|}{1 - \frac{|\cos(2a)|}{\cosh(2x)}} \frac{1}{x^2} < (c-\sigma_0)^4.
    \end{align*}
    If $a \leq \frac{\pi}{4}$, then this concludes the proof.
    If $a \geq \frac{\pi}{4}$, let $\frac{\pi}{4} \leq |y| \leq a$. First, let $T>0$, then using \eqref{eq : estimation last step}  again we obtain
     \begin{align*}
        |m_T(\xi + iy)|^4 
        &\geq  \frac{|1 - \frac{\cos(2y)}{\cosh(2\xi)}|}{|1 + \frac{\cos(2y)}{\cosh(2\xi)}|} \frac{T^2\xi^4}{\xi^2+ a^2}\\
        &\geq  \frac{T^2x^4}{x^2 + a^2} \geq \frac{\cosh(2x)-1}{\cosh(2x)+1} \frac{T^2x^4}{x^2 + a^2} > (c+\sigma_0)^4
    \end{align*}
    as $\cos(2y) \leq 0.$ Finally, let $T=0$ and observe that
    \begin{align*}
        |m_0(\xi + iy)|^4 = \frac{|1 - \frac{\cos(2y)}{\cosh(2\xi)}|}{|1 + \frac{\cos(2y)}{\cosh(2\xi)}|} \frac{1}{\xi^2+y^2} \leq \frac{1 + \frac{|\cos(2a)|}{\cosh(2x)}|}{1 - \frac{|\cos(2a)|}{\cosh(2x)}} \frac{1}{x^2} < (c-\sigma_0)^4.
    \end{align*}
\end{proof}

Given $x>0$ satisfying \eqref{condition on x for proof of a}, Proposition \ref{prop : minimal value for x in proof of a} provides that it is enough to prove that $|m_T(z)-c| >0$ for all $z \in S_x$ in order to obtain \eqref{ineq : requirement on a}. The proof on $S_x$ is achieved numerically using the arithmetic on intervals on Julia (cf. \cite{julia_interval}). Indeed, we write 
\[
S_x = \bigcup_{k=1}^{M_1}\bigcup_{j=1}^{M_2} I_k + i I_j
\]
for some $M_1, M_2 \in \mathbb N$ where $(I_k)$ and $(I_j)$ are families of intervals. Then if one can prove that 
\[
|m_T(I_k + iI_j) - c|\cap \{ 0 \} = \emptyset
\]
for all $k \in \{1, \dots, M_1\}$ and $j \in  \{1, \dots, M_2\}$, then $a$ satisfies the second inequality of  \eqref{ineq : requirement on a}. 
Then, using a similar approach, we write 
\begin{align*}
    [-x,x] = \bigcup_{k=1}^{M_1} I_k
\end{align*}
for a family of disjoint intervals $(I_k)_{k \in \{1, \dots, M_1\}}$ and verify that $$\inf (|m_T(I_k+ia)-c|), \inf (|m_T(I_k)-c|) \geq \sigma_0$$ for all $k \in \{1, \dots, M_1\}.$ This ensures that $\sigma_0$ and $a$ satisfy \eqref{ineq : requirement on a}.
The algorithmic details are presented in \cite{julia_cadiot}.

Now, in a similar fashion, we can determine a value for $\sigma_1$ satisfying \eqref{eq : alpha assumption}. In particular, the next lemma allows controlling the asymptotics of $m_T(\xi + ia)-c$ when $|\xi|$ gets big enough. 
\begin{prop}
    Let $x\geq a$ satisfying 
    \begin{align*}
        \frac{1}{2} \left(\left(\frac{\cosh(2x)-1}{\cosh(2x)+1}\right) \frac{T^2x^4}{x^2 + a^2}\right)^{\frac{1}{4}}  - |c| \geq 0.
    \end{align*}
    Now let $0 <\sigma_1 \leq \left(\frac{1}{32}\left(\frac{\cosh(2x)-1}{\cosh(2x)+1}\right)\right)^{\frac{1}{4}}$.
    If $|m_T(\xi+ia) - c| \geq \sigma_1\sqrt{T|\xi|}$ for all $|\xi| \leq x$, then $\sigma_1$ satisfies \eqref{eq : alpha assumption}. 
\end{prop}

\begin{proof}
To prove the proposition, we need to prove that $|m_T(\xi+ia) - c| \geq \sigma_1\sqrt{T|\xi|}$ for all $|\xi| \geq x$.
     Using the proof of Proposition \ref{prop : minimal value for x in proof of a}, we have
\begin{align*}
    |m_T(\xi + ia)|^4 
        &\geq   \left(\frac{\cosh(2x)-1}{\cosh(2x)+1}\right) \frac{T^2\xi^4}{\xi^2 + a^2} 
\end{align*}
for all $\xi \geq x$. Let $\xi \geq x$, then using that
\begin{align*}
   \frac{1}{2}\left(\left(\frac{\cosh(2x)-1}{\cosh(2x)+1}\right) \frac{T^2\xi^4}{\xi^2 + a^2}\right)^{\frac{1}{4}}  - |c| \geq \frac{1}{2} \left(\left(\frac{\cosh(2x)-1}{\cosh(2x)+1}\right) \frac{T^2x^4}{x^2 + a^2}\right)^{\frac{1}{4}}  - |c| \geq 0
\end{align*}
by assumption, we obtain
\begin{align*}
  |m_T(\xi+ia)-c|^4 \geq \frac{1}{16}\left(\frac{\cosh(2x)-1}{\cosh(2x)+1}\right) \frac{T^2\xi^4}{\xi^2 + a^2}.
\end{align*}
Now, since $\xi \geq x \geq a$, then
$\frac{T^2\xi^4}{\xi^2 + a^2} \geq \frac{1}{2}{T^2\xi^2}$.  Therefore,  
\[
|m_T(\xi+ia)-c|^4 \geq \frac{1}{32}\left(\frac{\cosh(2x)-1}{\cosh(2x)+1}\right){T^2\xi^2}
\]
for all $\xi \geq x$.
\end{proof}

Using the previous lemma, it is enough to prove that  $|m_T(\xi+ia) - c| \geq \sigma_1\sqrt{T|\xi|}$ for all $|\xi| \leq x$ in order to prove that $\sigma_1$ satisfies \eqref{eq : alpha assumption}. Consequently, we can again break the interval $[-x,x]$ in sub-intervals and ensure that $|m_T(\xi+ia) - c| \geq \sigma_1\sqrt{T|\xi|}$ for all $|\xi| \leq x$ using the arithmetic on intervals (cf. \cite{julia_interval}). The computations details are presented in \cite{julia_cadiot}.

\bibliographystyle{abbrv}
\bibliography{biblio}

\end{document}